\documentclass[10pt]{article}
\usepackage{graphicx}
\usepackage{epsfig}
\usepackage{amssymb}
\usepackage{amsmath}
\usepackage{amsthm}
\newtheorem{lemma}{Lemma}[section]
\newtheorem{theorem}{Theorem}[section]

\newenvironment{remark}{\begin{description} \item[Remark.] }%
{\end{description}}
{\end{description}}

{\hfill $\bullet$ \\}

\newcommand{\be}{\begin{equation}}
\newcommand{\ee}{\end{equation}}

\begin{document}
    \title{Error Estimate of MacCormack Rapid Solver Method for 2D Incompressible Navier-Stokes Problems}
   \author{\Large{Eric Ngondiep}
       \thanks{Tel.: +966506048689. E-mail addresses:\ ericngondiep@gmail.com or engondiep@imamu.edu.sa
        (Eric Ngondiep).\ }}
    %\author{
   \date{\small{Department of Mathematics and Statistics, College of Science, Al-Imam Muhammad Ibn Saud Islamic University (IMSIU), 90950 Riyadh, Kingdom of Saudi Arabia}\\
       \text{\,}\\
       \small{Hydrological Research Centre, Institute for Geological and Mining Research, 4110 Yaounde-Cameroon}}

    \maketitle
   \textbf{Abstract.}
    The error estimates and convergence rate of a two-level MacCormack rapid solver method for solving a two-dimensional incompressible Navier-Stokes equations are analyzed. This represents a continuation of the work on the stability analysis of the method. The theoretical result suggests that the rapid solver method is both convergent and second order accurate with respect to time step $\Delta t.$ A wide set of numerical evidences confirm this theoretical analysis.\\
    \text{\,}
   \\
   {\bf AMS Subject Classification (MSC). 65M10, 65M05}.

      \section{Introduction and motivation}\label{I}
     Let $\Omega_{f}\subset\mathbb{R}^{d}$ ($d=2$ or $3$), be a fluid flow domain assumed to be bounded, to have a lipschitz-continuous boundary $\Gamma$ and to satisfy a further condition given by $(\ref{15b})$ below. Let $T$ be a positive parameter ($T$ can be equal $\infty$). We consider the $2$D time dependent nonlinear partial differential equations (PDEs) describing the flow of a fluid confined in $\Omega_{f}:$
     \begin{equation}\label{1}
     u_{t}-\nu\triangle u+(u\cdot\nabla)u+ \nabla p=f, \text{\,\,\,\,\,in\,\,\,}\Omega_{f}\times[0,\text{\,}T],
     \end{equation}
     \begin{equation}\label{2}
     \nabla\cdot u = 0, \text{\,\,\,\,\,in\,\,\,}\Omega_{f}\times[0,\text{\,}T],
     \end{equation}
     with the initial condition
     \begin{equation}\label{3}
      u(x,0)=u_{0}(x), \text{\,\,\,\,\,in\,\,\,}\Omega_{f},
     \end{equation}
     and the boundary condition
     \begin{equation}\label{4}
     u=0, \text{\,\,\,\,\,on\,\,\,}\Gamma\times[0,\text{\,}T],
     \end{equation}
    where: $u_{t}=\frac{\partial u}{\partial t},$ $u=(u_{1},u_{2})$ is the velocity, $p$ is the pressure, $f$ represents the density of the body forces, $\nu$ is the viscosity and $u_{0}$ denotes the initial velocity.
    $2$D unsteady incompressible Navier-Stokes problems are a mixed set of elliptic-parabolic equations. These equations have been studied extensively due to its analogous nature to many practical applications, and several numerical schemes have been developed to provide solutions dedicated to different environment conditions (such as different Reynolds numbers). More recently, methods have been developed to efficiently solve the compressible Navier-Stokes equations at very low Mach numbers \cite{65rr,95rr}. For these flows, the mesh must be highly refined in order to accurately resolve the viscous regions. This leads to small time steps and subsequently, long computing times if an explicit scheme or implicit method is used. A possible improvement is to use a two-level or multilevel method \cite{2mx,15ynh,hmr,mx,38mx,20hmr,24ynh} and a two-level explicit-implicit scheme such as the MacCormack rapid solver (MCRS) method, which is the hybrid version of the two-level explicit MacCormack \cite{mc1}. This hybrid approach is considered in this paper as a coupled explicit MacCormack algorithm and Crank-Nicolson scheme. In reality, the MacCormack algorithm which is a predictor-corrector, finite difference scheme provides good resolution at discontinuities and the best resolution of discontinuities occurs when the difference in the predictor is in the direction of the propagation of the discontinuity. This method has been widely used to solve certain class of nonlinear PDEs. As a consequence, the authors \cite{en4} applied this approach for solving a complex nonlinear PDE and they obtained satisfactory results regarding the stability and the convergence rate of the scheme. For multidimensional problems, a time-split version of the MacCormack method has been developed and deeply studied  (for instance, see \cite{en3,en5}, \cite{apt} pages: $230$-$231$). However, this scheme is not a satisfactory approach for solving high Reynolds numbers flow, where the viscous regions becomes very thin (\cite{apt}, p. $631$-$632$).\\

     The aim of our study is to find an efficient numerical solution of the initial-value boundary problem $(\ref{1})$-$(\ref{4}),$ using a two-level MacCormack rapid solver algorithm, that is, a combination of an explicit MacCormack method and a Crank-Nicolson scheme. An application of the MCRS approach to the $3$D unsteady Navier-Stokes equations can be found in \cite{apt}, pages: $631$-$632.$ There are many reasons as discussed in \cite{apt,rr,1ynh,2ynh} that have led to active research and developing effective and efficient techniques for both stationary and nonstationary models so that existing single-model solvers can be applied locally with little extra computational and software overhead. In \cite{7ynh,21ynh}, the authors analyzed the local error estimates, stability and convergence of a two-level method obtained by the semi-discretization in space together with the full discretization in space-time of the $2$D and $3$D time dependent Navier-Stokes equations, but the global error estimates are not provided. Furthermore, the author \cite{rr}, section $4.2,$ P. $36$-$39,$ applied the Crank-Nicolson approach and the Fractional-Step-$\theta$ scheme to the nonstationary and incompressible flows in the cross-section of a channel at Reynolds number, $Re=500,$ and he has observed that both methods have shown equally satisfactory results. More recently \cite{en1,en2,en}, MCRS method has been deeply studied for both coupled Stokes-Darcy model and $2$D incompressible Navier-Stokes equations. In this paper, we are interested by the error estimates and the convergence rate of the rapid solver method applied to $2$D incompressible nonstationary equations $(\ref{1})$-$(\ref{4})$. This represents a continuation of the work studied in \cite{en}. Some numerical experiments that confirm the theoretical analysis are also considered.\\

     The remainder of the paper is organized as follows: Section $\ref{II}$ deals with the variational formulation of the $2$D nonstationary incompressible Navier-Stokes equations. In section $\ref{III}$ we analyze the error estimates of the two-level hybrid algorithm for problem $(\ref{1})$-$(\ref{4}).$ Section $\ref{IV}$ considers some numerical experiments while in section $\ref{V},$ we draw the general conclusion and present the future direction of works.

      \section{Weak formulation of the $2$D time dependent incompressible Navier-Stokes equations}\label{II}
      This section considers some notations together with the basic theoretical concepts that help to analyze the error estimates and convergence rate of MCRS scheme. The variational formulation of $2$D nonstationary incompressible Navier-Stokes model along with the discrete weak formulation of the rapid solver method for solving problem $(\ref{1})$-$(\ref{4})$ are presented.\\

      In order to introduce the weak formulation of problem $(\ref{1})$-$(\ref{4}),$ we define the following spaces
      \begin{equation}\label{5}
        W=(H_{0}^{1}(\Omega_{f}))^{2},\text{\,\,\,\,\,}X=(L^{2}(\Omega_{f}))^{2}\text{\,\,\,\,\,and\,\,\,\,\,}
        Q=W_{2}^{0}(\Omega_{f}):=\left\{z\in L^{2}(\Omega_{f}):\text{\,}\int_{\Omega_{f}}z(x)dx=0\right\}.
      \end{equation}
      $W_{2}^{0}(\Omega_{f})$ is endowed with the usual $L^{2}$-scalar product $(\cdot,\cdot)$ and the related norm is represented by $\|\cdot\|_{L_{f}^{2}}.$ Now, let $L^{2}$-norm on $\Omega_{f}$ (resp. $\Gamma$) be also denoted by $\|\cdot\|_{L_{f}^{2}}$ (resp. $\|\cdot\|_{\Gamma}$) and the corresponding inner product be denoted by $(\cdot,\cdot).$ Furthermore, the space $H_{0}^{1}(\Omega_{f})$ is equipped with the scalar product and norm
      \begin{equation}\label{7}
     (u,v)_{1}=(\nabla u,\nabla v),\text{\,\,}\|u\|_{1}=\|\nabla u\|_{L_{f}^{2}}=\sqrt{(u,u)_{1}},\text{\,\,}
     \forall u,v\in H_{0}^{1}(\Omega_{f}).
      \end{equation}
      We also equip both spaces $X$ and $W$ with the following scalar products and norms
      \begin{equation}\label{8}
      (u,v)_{X}=(u,v),\text{\,\,}\|u\|_{L_{f}^{2}}=\sqrt{(u,u)},\text{\,\,}\forall u,v\in X; \text{\,\,}(u,v)_{W}=\nu(\nabla u,\nabla v),
      \text{\,\,}\|u\|_{W}=\sqrt{(u,u)_{W}},\text{\,\,}\forall u,v\in W.
      \end{equation}
      Since $H_{0}^{1}(\Omega_{f})\subset H^{1}(\Omega_{f})\subset L^{2}(\Omega_{f}),$ so $W$ is a subspace of $X$
      and it comes from equations $(\ref{8})$
      \begin{equation}\label{9}
        \|u\|_{W}=\sqrt{\nu}\|\nabla u\|_{L_{f}^{2}},\text{\,\,\,\,\,\,}\forall u\in W.
      \end{equation}
       In this study, the initial condition $u_{0}$ and the external force $f$ are assumed to be regular enough
       so that the initial-boundary value problem $(\ref{1})$-$(\ref{4}),$ admits a smooth solution. We recall the Poincar\'{e}-Friedrichs inequality $(\ref{10})$ which plays a crucial role in our study.
      \begin{equation}\label{10}
        \|z\|_{L_{f}^{2}}\leq \frac{C_{f}}{\sqrt{\nu}}\|z\|_{W}, \text{\,\,\,for\,\,some\,\,constant\,\,}C_{f}>0.
      \end{equation}
      We define both bilinear forms $a(\cdot,\cdot)$ and $\chi(\cdot,\cdot)$ on $W\times W$ and $W\times Q,$ respectively, by
      \begin{equation}\label{12}
        a(u,v)=(u,v)_{W},\text{\,\,\,}\forall u,v\in W,\text{\,\,\,\,\,\,and\,\,\,\,\,\,\,}
        \chi(u,g)=(g,\nabla\cdot u),\text{\,\,\,}\forall u\in W, \text{\,\,\,}\forall g\in Q,
      \end{equation}
      and let introduce the closed subspace $F$ of $W$ defined by
      \begin{equation}\label{14}
        F=\{v\in W: \chi(v,g)=0,\text{\,}\forall g\in Q\}=\{v\in W:\text{\,}\nabla\cdot v=0,\text{\,\,in\,\,} \Omega_{f}\}.
      \end{equation}
       Now, let us introduce the unbounded linear operator $A$ defined on $X$ by $Au=-\Delta u.$
       We choose $\Omega_{f}$ so that $D(A)$ (the domain of $A$) satisfies
      \begin{equation}\label{15b}
        D(A)=W\cap\left(W_{2}^{2}(\Omega_{f})\right)^{2}.
      \end{equation}
      In literature \cite{9ynh}, it is proven that equation $(\ref{15b})$ holds if $\Gamma$ is of class $C^{2}$
      or if $\Omega_{f}$ is a convex plane polygonal domain.\\

      Finally, setting $b(u,v,w)=\frac{1}{2}((u\cdot\nabla)v,w)+\frac{1}{2}((\nabla\cdot u)v,w)=\frac{1}{2}((u\cdot\nabla)v,w)
      -\frac{1}{2}((u\cdot\nabla)w,v),$ for all $u,v,w\in W.$ In \cite{1ynh,8ynh,11ynh} the authors showed that the mapping $b(\cdot,\cdot,\cdot)$ is a trilinear form, continuous and satisfying the following properties:
      \begin{equation}\label{16}
        b(u,v,w)=-b(u,w,v),\text{\,\,\,\,\,\,\,\,} b(u,v,v)=0,\text{\,\,\,}\forall u,v,w\in W,
      \end{equation}
      and
      \begin{equation}\label{17}
        |b(u,v,w)|\leq\alpha\|u\|_{L_{f}^{2}}^{\frac{1}{2}}\|u\|_{W}^{\frac{1}{2}}\|v\|_{W}\|w\|_{L_{f}^{2}}^{\frac{1}{2}}
        \|w\|_{W}^{\frac{1}{2}}+\alpha\|u\|_{W}\|v\|_{L_{f}^{2}}^{\frac{1}{2}}\|v\|_{W}^{\frac{1}{2}}
        \|w\|_{W}^{\frac{1}{2}}\|w\|_{W}^{\frac{1}{2}},\text{\,\,\,}\forall u,v,w\in W.
      \end{equation}
      Using the Poincar\'{e}-Friedrichs inequality, estimate $(\ref{17})$ implies
      \begin{equation}\label{18}
        |b(u,v,w)|\leq\alpha\|u\|_{W}\|v\|_{W}\|w\|_{W},\text{\,\,\,}\forall u,v,w\in W,
      \end{equation}
      where $\alpha$ is a positive constant whose value may be different from place to place.\\

       Given two functions $f\in L^{\infty}(\mathbb{R}_{0}^{+};X)$ and $u_{0}\in F$, it comes from the
       definitions of both bilinear and trilinear forms $a(\cdot,\cdot)$ and $b(\cdot,\cdot,\cdot),$ that a weak formulation of the $2$D time dependent Navier-Stokes model $(\ref{1})$-$(\ref{4}),$ reads as follows:
       find a pair $(u,p)$ with
      \begin{equation}\label{19}
        u\in L^{\infty}(\mathbb{R}_{0}^{+};X)\cap L^{2}(0,T;F),\text{\,\,}u_{t}\in L^{2}(0,T;F^{'}) \text{\,\,and\,\,}
         p\in D^{'}(\Omega_{f}\times(0,T)),
      \end{equation}
      satisfying
      \begin{equation}\label{20}
        (u_{t},v)+a(u,v)+b(u,u,v)-\chi(v,p)+\chi(u,q)=(f,v),\text{\,\,\,}\forall u,v\in W,\text{\,\,}
        \forall q\in Q,
      \end{equation}
      \begin{equation}\label{21}
        u(0)=u_{0}.
      \end{equation}
      It is well known in literature (for instant, we refer the readers to \cite{8ynh,11ynh}) that the system of equations $(\ref{20})$-$(\ref{21})$ has a unique solution $(u,p).$

       To discretize the time dependent Navier-Stokes problem $(\ref{20})$-$(\ref{21})$ in space by finite element method (FEM), we construct finite element spaces
        \begin{equation*}
            \text{\,velocity:\,\,}W_{h}\subset W,\text{\,\,Stokes\,\,pressure:\,\,}Q_{h}\subset Q,
        \end{equation*}
        based on a conforming FEM triangulation $\Pi_{h}$ of the domain $\Omega_{f}\cup\Gamma,$ with maximum triangle (or tetrahedra) diameter denoted by $"h".$ Furthermore, the velocity-pressure FEM spaces $W_{h}$ and $Q_{h}$ are assumed to satisfy the well known discrete $\inf$-$\sup$ condition for stability of the pressure, that is, for every $q_{h}\in Q_{h},$ there is a $u_{h}\in W_{h},$ $u_{h}\neq0,$ such that
        \begin{equation}\label{25}
            \chi(u_{h},q_{h})\geq\beta\|u_{h}\|_{W}\|q_{h}\|_{L_{f}^{2}},
        \end{equation}
        where $\beta>0,$ is a constant independent of $h.$ We denote the discretely divergence free velocity by
        \begin{equation}\label{26}
            V_{\mu}=W_{\mu}\cap\{v_{\mu}:\text{\,}\chi(v_{\mu},q_{\mu})=0,\text{\,}\forall q_{\mu}\in Q_{\mu}\},
            \text{\,\,\,where\,\,\,}\mu=H,h.
        \end{equation}
        Now, let $\Phi_{\mu}:X\rightarrow W_{\mu},$ and $\Psi_{\mu}:Q\rightarrow Q_{\mu},$ (where $\mu=H,h$) be the $L^{2}$-orthogonal projections defined by
        \begin{equation}\label{27}
            (\Phi_{\mu}u,u_{\mu})=(u,u_{\mu}),\text{\,\,}\forall u\in X,\text{\,\,}\forall u_{\mu}\in W_{\mu},
        \end{equation}
        \begin{equation*}\label{28}
            (\Psi_{\mu}q,q_{\mu})=(q,q_{\mu}),\text{\,\,}\forall q\in Q,\text{\,\,}\forall q_{\mu}\in Q_{\mu},
        \end{equation*}
        respectively. We assume that the finite element spaces $W_{h}$ and $Q_{h}$ satisfy the first order approximation $O(h).$ The corresponding inverse estimate is well known and given by
        \begin{equation*}\label{30}
            \|v_{h}\|_{W}\leq \widetilde{C}h^{-1}\|v_{h}\|_{L_{f}^{2}},\text{\,\,\,\,}\forall v_{h}\in W_{h},
        \end{equation*}
        where $\widetilde{C}$ is a generic constant depending on the data $\Omega_{f},$ $\nu,$ and $f,$ which may stand for values at different occurrences, but is independent of the mesh size $h$ and time step $\Delta t.$ Finally, we also will use the following Poincar\'{e} inequality
         \begin{equation}\label{34}
            \lambda_{1}\nu\|u\|_{L_{f}^{2}}^{2}\leq\|u\|_{W}^{2},\text{\,\,\,}\forall u\in W,
         \end{equation}
         where $\lambda_{1}$ is the first eigenvalue of the operator $A$ given by equation $(\ref{15b}).$\\

         We propose a decoupling scheme based on the semi-discretization in space on the coupling terms. This leads to an efficient decoupled marching algorithm and easy implementation. Following the works developed in \cite{24ynh,15ynh,7ynh,21ynh,1ynh}, we describe how to approximate $u_{h}(t)$ in
         the coupling term $b(u_{h},u_{h},v_{h})$ by an appropriate extrapolation of the computed solutions from the previous steps. With the two-level MCRS method, we should approximate $u_{h}(t)|_{b(\cdot,\cdot,\cdot)}$ by the corresponding spatial extrapolation $u_{H}(t)|_{b(\cdot,\cdot,\cdot)}$. More specifically, find $(u_{h},p_{h}):(0,T)\rightarrow W_{h}\times Q_{h},$ that satisfies:
         \begin{description}
           \item[\textbf{Step1:}] using the explicit MacCormack algorithm to find $(u_{H}(t),p_{H}(t))\in W_{H}\times Q_{H},$ such that for all $(v_{H},q_{H})\in W_{H}\times Q_{H},$
               \begin{equation}\label{35a}
                \left.
                  \begin{array}{ll}
                    (u_{H,t},v_{H})+a(u_{H},v_{H})+b(u_{H},u_{H},v_{H})-\chi(v_{H},p_{H})+\chi(u_{H},q_{H})=(f,v_{H}), & \hbox{} \\
                    \text{\,}\\
                    u_{H}(0)=\Phi_{H}u_{0}; & \hbox{}
                  \end{array}
                \right.
               \end{equation}

           \item[\textbf{Step2:}] with the Crank-Nicolson scheme, find $(u_{h}(t),p_{h}(t))\in
           W_{h}\times Q_{h},$ such that for all $(v_{h},q_{h})\in W_{h}\times Q_{h},$
               \begin{equation}\label{35b}
                \left.
                  \begin{array}{ll}
                    (u_{h,t},v_{h})+a(u_{h},v_{h})+b(u_{H},u_{H},v_{h})-\chi(v_{h},p_{h})+\chi(u_{h},q_{h})=(f,v_{h}), & \hbox{} \\
                    \text{\,}\\
                    u_{h}(0)=\Phi_{h}u_{0}. & \hbox{}
                  \end{array}
                \right.
               \end{equation}
         \end{description}
         where $\Phi_{H}$ is defined by relation $(\ref{27}).$ It is worth noticing to mention that $(W_{H},Q_{H})\subset(W_{h},Q_{h}).$ We recall that the aim of this paper is to give a general picture of both error estimates and convergence rate of the hybrid method. Since the formulas can become quite heavy, for the sake of readability, we should use the same time step $\Delta t,$ in the two sets $\Omega_{f}$ and $\Gamma.$ Noting $t^{n}=n\Delta t,$ and $T=N\Delta t$ (if $T=\infty$ then $N=\infty$). Also denote $u^{n}:=u(t_{n})$ (and similarly for other variables). We introduce the following discrete norms,
         \begin{equation}\label{n1}
            \||v|\|_{L^{2}(0,T;X)}^{2}:=\Delta t\overset{N}{\underset{n=0}\sum}\|v^{n}\|_{L_{f}^{2}}^{2},
            \text{\,\,\,for\,\,}v\in L^{2}(0,T;X).
         \end{equation}

         The discrete variational formulation of the explicit MacCormack scheme for equations $(\ref{1})$-$(\ref{4})$ on coarse mesh reads: given $(u_{H}^{n},p_{H}^{n})\in V_{H}\times Q_{H},$ find an approximation $(u_{H}^{n+1},p_{H}^{n+1})\in V_{H}\times Q_{H},$ for $n\geq0,$ such that, for all $(v_{H},q_{H})\in V_{H}\times Q_{H},$ it holds
         \begin{description}
           \item[Predictor:]
           \begin{equation}\label{36}
            \left(\frac{u_{H}^{\overline{n+1}}-u_{H}^{n}}{\Delta t},v_{H}\right)+a(u_{H}^{n},v_{H})+b(u_{H}^{n},u_{H}^{n},v_{H})-\chi(v_{H},p_{H}^{n})
            +\chi(u_{H}^{n},q_{H})=(f^{n},v_{H}),
           \end{equation}
           \begin{equation}\label{37}
            u_{H}^{\overline{0}}=\Phi_{H}u(\overline{t_{0}});
           \end{equation}

           \item[Corrector:]
           \begin{equation*}
            \left(\frac{u_{H}^{n+1}-\frac{u_{H}^{\overline{n+1}}+u_{H}^{n}}{2}}{\frac{\Delta t}{2}},v_{H}\right)+a(u_{H}^{\overline{n+1}},v_{H})
            +b(u_{H}^{\overline{n+1}},u_{H}^{\overline{n+1}},v_{H})-\chi(v_{H},p_{H}^{\overline{n+1}})+\chi(u_{H}^{\overline{n+1}},q_{H})=
           \end{equation*}
           \begin{equation}\label{38}
            (f^{\overline{n+1}},v_{H}),
           \end{equation}
           \begin{equation}\label{39}
            u_{H}^{0}=\Phi_{H}u(t_{0}),
           \end{equation}
         \end{description}
         where $\Phi_{H}$ is defined by relation $(\ref{27}).$ Furthermore, to keep the second order accuracy of Crank-Nicolson scheme we use the second order approximations $u_{h}^{n}=\frac{u_{h}^{n+1}+u_{h}^{n-1}}{2}$ and $p_{h}^{n}=\frac{p_{h}^{n+1}+p_{h}^{n-1}}{2}.$ So, a (monolithic) weak formulation of the Crank-Nicolson method for problem $(\ref{1})$-$(\ref{4})$ on fine grid reads: given $(u_{h}^{n-1},p_{h}^{n-1})$, $(u_{h}^{n},p_{h}^{n})\in V_{h}\times Q_{h},$ find an approximation $(u_{h}^{n+1},p_{h}^{n+1})\in V_{h}\times Q_{h},$ for $n\geq1,$ such that, for all $(v_{h},q_{h})\in V_{h}\times Q_{h}$
           \begin{equation*}
            \left(\frac{u_{h}^{n}-u_{h}^{n-1}}{\Delta t},v_{h}\right)+a\left(\frac{u_{h}^{n+1}+u_{h}^{n-1}}{2},v_{h}\right)
            +b(u_{H}^{n},u_{H}^{n},v_{h})-\chi\left(v_{h},\frac{p_{h}^{n+1}+p_{h}^{n-1}}{2}\right)
            +\chi\left(\frac{u_{h}^{n+1}+u_{h}^{n-1}}{2},q_{h}\right)=
           \end{equation*}
           \begin{equation}\label{40a}
            (f^{n+1},v_{h}),
           \end{equation}
           \begin{equation}\label{40b}
            u_{h}^{0}=\Phi_{h}u(t_{0}).
           \end{equation}
         In $(\ref{36})$-$(\ref{38}),$ $u_{H}^{\overline{n+1}}$ and $p_{H}^{\overline{n+1}}$ are temporary "predicted" values of $u_{H}$ and $p_{H},$ respectively, at the time level $n+1.$\\

         We know from the initial condition $(\ref{3})$ that $u_{H}^{0}$=$u_{H}^{\overline{0}}$=$u_{h}^{0}$ and $p_{H}^{0}$=$p_{H}^{\overline{0}}$=$p_{h}^{0}.$ In the subsequent sections, we suppose that $p_{H}^{\overline{1}}=p_{H}^{1}=p_{h}^{1}.$ Assuming that the superscript $\overline{n+1}$ is a time level, it is obvious to see that the two-level MCRS is a three steps method, so the initial data $u_{H}^{0},$ $p_{H}^{0}$ and the terms $u_{h}^{1}$ and $p_{H}^{1}$ are needed to begin the algorithm. Both terms $u_{h}^{1}$ and $p_{H}^{1}$ can be obtained by a two-step method that solves the system, such as by two-level method introduced in \cite{21ynh}.\\

         The following result plays a crucial role in the proof of the main result of this paper (namely Theorem $\ref{t2}$).

         \begin{lemma}\label{l1}. Let $f\in L^{2}(0,T;X)$ and $u_{0}\in D(A)\cap F$ be given. Consider the MacCormack algorithm $(\ref{36})$ and $(\ref{38}).$ For all $N\in\mathbb{N},$ $N\geq1,$ it holds
        \begin{equation}\label{41}
        3\|u_{H}^{N}\|_{X}^{2}+\|u_{H}^{\overline{N}}\|_{X}^{2}+2\nu\Delta t\underset{n=0}{\overset{N-1}\sum}\left(\|u_{H}^{n}\|_{W}^{2}
         +\|u_{H}^{\overline{n+1}}\|_{W}^{2}\right)\leq\frac{2\Delta t}{\lambda_{1}\nu^{2}}\underset{n=0}{\overset{N-1}\sum}
         \left(\|f^{\overline{n+1}}\|_{X}^{2}+\|f^{n}\|_{X}^{2}\right)+4\|u_{H}^{0}\|_{X}^{2},
        \end{equation}
        where $\lambda_{1}>0,$ represents the first eigenvalue of the unbounded linear operator $A$ that satisfies relation $(\ref{15b}).$
        \end{lemma}

        \begin{proof}
        It comes from equations $(\ref{36})$ and $(\ref{38}),$ that
        \begin{equation}\label{1*}
        \frac{1}{\Delta t}\left(u_{H}^{\overline{n+1}}-u_{H}^{n},v_{H}\right)+a(u_{H}^{n},v_{H})+b(u_{H}^{n},u_{H}^{n},v_{H})-
        \chi(v_{H},p_{H}^{n})+\chi(u_{H}^{n},q_{H})=(f^{n},v_{H}),
        \end{equation}
        and
       \begin{equation}\label{2*}
        \frac{1}{\Delta t}\left(2u_{H}^{n+1}-u_{H}^{\overline{n+1}}-u_{H}^{n},v_{H}\right)+a(u_{H}^{\overline{n+1}},v_{H})+
        b(u_{H}^{\overline{n+1}},u_{H}^{\overline{n+1}},v_{H})-\chi(v_{H},p_{H}^{\overline{n+1}})
        +\chi(u_{H}^{\overline{n+1}},q_{H})=(f^{\overline{n+1}},v_{H}).
        \end{equation}
         Taking $v_{H}=u_{H}^{n},$ $q_{H}=p_{H}^{n}$ in equation $(\ref{1*})$, and $v_{H}=u_{H}^{\overline{n+1}},$ $q_{H}=p_{H}^{\overline{n+1}},$ in relation $(\ref{2*}),$ the terms $-\chi(v_{H},p_{H}^{n})+\chi(u_{H}^{n},q_{H})$ and $-\chi(v_{H},p_{H}^{\overline{n+1}})+\chi(u_{H}^{\overline{n+1}},q_{H})$ cancel. In addition, since the trilinear form $b(\cdot,\cdot,\cdot)$ is skew-symmetric, the terms $b(u_{H}^{n},u_{H}^{n},v_{H})$ and
         $b(u_{H}^{\overline{n+1}},u_{H}^{\overline{n+1}},v_{H})$ also cancel. So, equations $(\ref{1*})$ and $(\ref{2*})$ become
         \begin{equation}\label{3*}
        \frac{1}{\Delta t}\left(u_{H}^{\overline{n+1}}-u_{H}^{n},u_{H}^{n}\right)+a(u_{H}^{n},u_{H}^{n})=(f^{n},u_{H}^{n}),
        \end{equation}
       \begin{equation}\label{4*}
        \frac{1}{\Delta t}\left(2u_{H}^{n+1}-u_{H}^{\overline{n+1}}-u_{H}^{n},u_{H}^{\overline{n+1}}\right)+
        a(u_{H}^{\overline{n+1}},u_{H}^{\overline{n+1}})=(f^{\overline{n+1}},u_{H}^{\overline{n+1}}).
        \end{equation}
         Now, we should approximate the terms $u_{H}^{n}$ and $u_{H}^{\overline{n+1}}.$ Since the MacCormack scheme is second order accurate, following the MacCormack approach, we must approximate these terms using the central difference representation. Expanding the Taylor series with time step
         $\frac{\Delta t}{2},$ for both predicted and corrected values yields
         \begin{equation}\label{5*}
            u_{H}^{n+1}=u_{H}^{\overline{n+1}}+\frac{\Delta t}{2}u_{H,t}^{\overline{n+1}}+O(\Delta t^{2});\text{\,\,\,\,\,\,\,\,}
            u_{H}^{n}=u_{H}^{\overline{n+1}}-\frac{\Delta t}{2}u_{H,t}^{\overline{n+1}}+O(\Delta t^{2});
         \end{equation}
         and
         \begin{equation}\label{7*}
            u_{H}^{\overline{n+1}}=u_{H}^{n}+\frac{\Delta t}{2}u_{H,t}^{n}+O(\Delta t^{2});\text{\,\,\,\,\,\,\,\,}
            u_{H}^{\overline{n}}=u_{H}^{n}-\frac{\Delta t}{2}u_{H,t}^{n}+O(\Delta t^{2}).
         \end{equation}
         Utilizing equations $(\ref{5*})$ and $(\ref{7*}),$ simple calculations give
         \begin{equation}\label{9*}
            u_{H}^{\overline{n+1}}=\frac{u_{H}^{n+1}+u_{H}^{n}}{2}+O(\Delta t^{2})\text{\,\,\,\,and\,\,\,\,}
            u_{H}^{n}=\frac{u_{H}^{\overline{n+1}}+u_{H}^{\overline{n}}}{2}+O(\Delta t^{2}).
         \end{equation}
         Substituting $(\ref{9*})$ into $(\ref{3*})$ and $(\ref{4*})$, it is easy to see that
         \begin{equation*}
        \frac{1}{\Delta t}\left(u_{H}^{\overline{n+1}}-\frac{u_{H}^{\overline{n+1}}
        +u_{H}^{\overline{n}}}{2}+O(\Delta t^{2}),\frac{u_{H}^{\overline{n+1}}+u_{H}^{\overline{n}}}{2}+O(\Delta t^{2})\right)+a(u_{H}^{n},u_{H}^{n})=(f^{n},u_{H}^{n}),
        \end{equation*}
        and
       \begin{equation*}
        \frac{1}{\Delta t}\left(2u_{H}^{n+1}-\frac{u_{H}^{n+1}+u_{H}^{n}}{2}-u_{H}^{n}+O(\Delta t^{2}),\frac{u_{H}^{n+1}+u_{H}^{n}}{2}+O(\Delta t^{2})\right)+a(u_{H}^{\overline{n+1}},
        u_{H}^{\overline{n+1}})=(f^{\overline{n+1}},u_{H}^{\overline{n+1}}),
        \end{equation*}
        which imply
        \begin{equation*}
        \frac{1}{4\Delta t}\left(u_{H}^{\overline{n+1}}-u_{H}^{\overline{n}},u_{H}^{\overline{n+1}}+u_{H}^{\overline{n}}\right)
        +a(u_{H}^{n},u_{H}^{n})=(f^{n},u_{H}^{n})+O(\Delta t),
        \end{equation*}
        and
       \begin{equation*}
        \frac{1}{4\Delta t}\left(3(u_{H}^{n+1}-u_{H}^{n}),u_{H}^{n+1}+u_{H}^{n}\right)+a(u_{H}^{\overline{n+1}},
        u_{H}^{\overline{n+1}})=(f^{\overline{n+1}},u_{H}^{\overline{n+1}})+O(\Delta t),
        \end{equation*}
        which are equivalent to
        \begin{equation}\label{11*}
        \frac{1}{4\Delta t}\left[\|u_{H}^{\overline{n+1}}\|_{X}^{2}-\|u_{H}^{\overline{n}}\|_{X}^{2}\right]
        +a(u_{H}^{n},u_{H}^{n})=(f^{n},u_{H}^{n})+O(\Delta t),
        \end{equation}
        and
       \begin{equation}\label{12*}
        \frac{3}{4\Delta t}\left[\|u_{H}^{n+1}\|_{X}^{2}-\|u_{H}^{n}\|_{X}^{2}\right]+a(u_{H}^{\overline{n+1}},
        u_{H}^{\overline{n+1}})=(f^{\overline{n+1}},u_{H}^{\overline{n+1}})+O(\Delta t).
        \end{equation}
        Plugging relations $(\ref{11*})$ and $(\ref{12*})$ provides
        \begin{equation*}
         \frac{1}{4\Delta t}\left\{3(\|u_{H}^{n+1}\|_{X}^{2}-\|u_{H}^{n}\|_{X}^{2})+\|u_{H}^{\overline{n+1}}\|_{X}^{2}
         -\|u_{H}^{\overline{n}}\|_{X}^{2}\right\}
         +a(u_{H}^{n},u_{H}^{n})+a(u_{H}^{\overline{n+1}},u_{H}^{\overline{n+1}})=(f^{n},u_{H}^{n})+
        \end{equation*}
        \begin{equation}\label{13*}
            (f^{\overline{n+1}},u_{H}^{\overline{n+1}})+O(\Delta t).
        \end{equation}
        On the other hand, a combination of Cauchy-Schwarz inequality and Poincar\'{e} inequality $(\ref{34})$ gives
        \begin{equation}\label{14*}
        (f^{n},u_{H}^{n})\leq\|f^{n}\|_{X}\|u_{H}^{n}\|_{X}\leq\frac{1}{2\lambda_{1}\nu^{2}}
        \|f^{n}\|_{X}^{2}+\frac{\nu}{2}\|u_{H}^{n}\|_{W}^{2}.
        \end{equation}
        In way similar
        \begin{equation}\label{15*}
        (f^{\overline{n+1}},u_{H}^{\overline{n+1}})\leq\|f^{\overline{n+1}}\|_{X}\|u_{H}^{\overline{n+1}}\|_{X}
        \leq\frac{1}{2\lambda_{1}\nu^{2}}\|f^{\overline{n+1}}\|_{X}^{2}+\frac{\nu}{2}\|u_{H}^{\overline{n+1}}
        \|_{W}^{2}.
        \end{equation}
        Substituting estimates $(\ref{14*})$ and $(\ref{15*})$ into relation $(\ref{13*})$ and using equations $(\ref{12})$ and $(\ref{8})$
        results in
        \begin{equation*}
         \frac{1}{4\Delta t}\left\{3(\|u_{H}^{n+1}\|_{X}^{2}-\|u_{H}^{n}\|_{X}^{2})+\|u_{H}^{\overline{n+1}}\|_{X}^{2}
         -\|u_{H}^{\overline{n}}\|_{X}^{2}\right\}+\nu\left(\|u_{H}^{n}\|_{W}^{2}+\|u_{H}^{\overline{n+1}}
         \|_{W}^{2}\right)\leq
        \end{equation*}
        \begin{equation*}
          \frac{\nu}{2}\left(\|u_{H}^{n}\|_{W}^{2}+\|u_{H}^{\overline{n+1}}\|_{W}^{2}\right)+  \frac{1}{2\lambda_{1}\nu^{2}}\left(\|f^{\overline{n+1}}\|_{X}^{2}+\|f^{n}\|_{X}^{2}\right)
          +O(\Delta t).
        \end{equation*}
        Multiplying both sides of this estimate by $4\Delta t$, and after simplification, we get
        \begin{equation*}
         3(\|u_{H}^{n+1}\|_{X}^{2}-\|u_{H}^{n}\|_{X}^{2})+\|u_{H}^{\overline{n+1}}\|_{X}^{2}
         -\|u_{H}^{\overline{n}}\|_{X}^{2}+2\nu\Delta t\left(\|u_{H}^{n}\|_{W}^{2}
         +\|u_{H}^{\overline{n+1}}\|_{W}^{2}\right)\leq
        \end{equation*}
        \begin{equation}\label{16*}
         \frac{2\Delta t}{\lambda_{1}\nu^{2}}\left(\|f^{\overline{n+1}}\|_{X}^{2}+\|f^{n}\|_{X}^{2}\right)+O(\Delta t^{2}).
        \end{equation}
         Summing inequality $(\ref{16*})$ up from $n=0$ to $N-1,$ provides
        \begin{equation*}
         3\|u_{H}^{N}\|_{X}^{2}+\|u_{H}^{\overline{N}}\|_{X}^{2}+2\nu\Delta t\underset{n=0}{\overset{N-1}\sum}\left(\|u_{H}^{n}\|_{W}^{2}+\|u_{H}^{\overline{n+1}}\|_{W}^{2}\right)
         \leq\frac{2\Delta t}{\lambda_{1}\nu^{2}}\underset{n=0}{\overset{N-1}\sum}
         \left(\|f^{\overline{n+1}}\|_{X}^{2}+\|f^{n}\|_{X}^{2}\right)+4\|u_{H}^{0}\|_{X}^{2}
         +O(\Delta t).
        \end{equation*}
         The last estimate comes from the initial condition $u_{H}^{0}=u_{H}^{\overline{0}}.$ Neglecting the error term $O(\Delta t)$, the proof of Lemma $\ref{l1}$ is completed. Indeed, $\Delta t$ is small,
         so the tracking of the infinitesimal term $O(\Delta t)$ does not compromise the result.
        \end{proof}

        \section{Analysis of convergence rate of MCRS method}\label{III}
        In this section, we analyze both error estimates and rate of convergence of MCRS discrete variational formulation $(\ref{36})$-$(\ref{40b})$ for incompressible Navier-Stokes problem $(\ref{1})$-$(\ref{4}).$ We assume that our finite element method (FEM) spaces $W_{h}$ and $Q_{h}$ satisfy the usual approximation properties of the piecewise polynomial of degrees $m-1,$ $m,$ and $m+1$
        \begin{equation}\label{80b}
        \underset{u_{h}\in W_{h}}{\inf}\|u-u_{h}\|_{L^{2}_{f}}\leq Ch^{m+1}\|u\|_{W^{m+1}},\text{\,\,\,\,}\forall
        u\in W_{h},
        \end{equation}
        \begin{equation}\label{81b}
        \underset{u_{h}\in W_{h}}{\inf}\|u-u_{h}\|_{W}\leq Ch^{m}\|u\|_{W^{m+1}},\text{\,\,\,\,}\forall
        u\in W_{h},
        \end{equation}
        \begin{equation}\label{82b}
        \underset{\lambda_{h}\in Q_{h}}{\inf}\|p-\lambda_{h}\|_{L^{2}_{f}}\leq Ch^{m+1}\|p\|_{W^{m+1}_{2}},
        \text{\,\,\,\,}\forall p\in Q_{h},
        \end{equation}
         where $L_{f}^{2}=L^{2}(\Omega_{f}),$ $W^{m}$ is a subspace of $W_{2}^{m}(\Omega_{f})\times W_{2}^{m}(\Omega_{f}),$ and the Stokes velocity-pressure spaces satisfy the discrete $\inf$-$\sup$ condition $(\ref{25}).$ For example
         \begin{equation*}
         W_{h}=\{v\in H_{0}^{1}(\Omega_{f})^{d}:\text{\,}v|_{K}\in\mathcal{Q}_{m+1}(K),\text{\,}\forall K\in\Pi_{h}\},
         \end{equation*}
         \begin{equation*}
         Q_{h}=\{q\in W_{2}^{0}(\Omega_{f}):\text{\,}q=q_{m}+q_{0},\text{\,}q_{m}\in\mathcal{C}(\overline{\Omega}_{f}),
         \text{\,}q_{m}|_{K}\in\mathcal{Q}_{m}(K),\text{\,}q_{0}|_{K}\in\mathcal{Q}_{0}(K),\text{\,}\forall K\in\Pi_{h}\},
         \end{equation*}
         where $\mathcal{Q}_{m}(K)=\{r=\widehat{r}\circ F_{K}^{-1},\text{\,}\widehat{r}\in\mathcal{Q}_{m}\},$ $F_{K}$ is an invertible mapping which maps the reference cell $\widehat{K}=[0,1]^{d}$ onto a generic quadrilateral hexahedral element $K=conv\{a_{i}\in\mathbb{R}^{d},\text{\,}1\leq i\leq2^{d}\},$ and $Q_{m}$ is the space of tensor product polynomials defined as
         $Q_{m}=span\left\{\underset{j=1}{\overset{d}\prod}x_{j}^{\alpha_{j}}:\text{\,}\underset{1\leq j\leq d}{\max}\alpha_{j}\leq m\right\}.$ We recall that the discrete divergence free velocities are given by
        \begin{equation*}
        V_{h}:=W_{h}\cap\{u_{h}:\text{\,\,}\chi(u_{h},q_{h})=0, \text{\,\,}\forall q_{h}\in Q_{h}\}.
        \end{equation*}
        As a consequence, there exists a positive constant $\beta$ such that, for $u\in V_{h},$
        \begin{equation}\label{83b}
        \underset{u_{h}\in V_{h}}{\inf}\|u-u_{h}\|_{W}\leq \beta\underset{u_{h}\in W_{h}}{\inf}\|u-u_{h}\|_{W},
        \end{equation}
        for example, see \cite{8ynh}, chap. II, proof of Theorem $1.1,$ in the case where $W=H^{1}(\Omega_{f})$ and $W_{h}=X_{f}^{h}.$ Now, let $N$ be a positive integer, denote $u^{n}=u(t^{n})$ and $T=N\cdot\Delta t.$ For $m\neq0,$ we introduce the following discrete norms:
        \begin{equation}\label{84b}
        \||u|\|_{L^{\infty}(0,T,W)}:=\underset{0\leq n\leq N}{\max}\|u^{n}\|_{L_{f}^{2}},\text{\,\,}
         \||\nabla u|\|_{L^{\infty}(0,T,W^{m})}:=\underset{0\leq n\leq N}{\max}\|u^{n}\|_{W^{m}},
        \end{equation}
         and
        \begin{equation}\label{85b}
          \||u|\|_{L^{2}(0,T,W)}:=\left(\Delta t\cdot\underset{n=0}{\overset{N}\sum}\|u^{n}\|^{2}_{L_{f}^{2}}\right)^{1/2},\text{\,\,}
          \||\nabla u|\|_{L^{2}(0,T,W^{m})}:=\left(\Delta t\cdot\underset{n=0}{\overset{N}\sum}
          \|u^{n}\|^{2}_{W^{m}}\right)^{1/2}.
        \end{equation}

        In the remainder of this paper, we assume the following regularity of the analytical solution
        \begin{equation}\label{86b}
             u\in H^{4}\left(0,T;W\right)\cap L^{2}(0,T;W^{l+1})\cap L^{\infty}(0,T;W^{l+1}).
        \end{equation}
        Let denote the exact errors by $e_{u\mu}^{m}=u^{m}-u_{\mu}^{m}$ and $e_{p\mu}^{m}=p^{m}-p_{\mu}^{m},$ the "predicted" errors by $e_{u\mu}^{\overline{m}}=u^{\overline{m}}-u_{\mu}^{\overline{m}}$ and $e_{p\mu}^{\overline{m}}=p^{\overline{m}}-p_{\mu}^{\overline{m}},$ where $\mu=H,h.$ A difference between the weak formulation of the continuous problem $(\ref{20})$-$(\ref{21})$ and a monolithic formulation of the discrete problem $(\ref{36})$-$(\ref{40b}),$ evaluated at time $t^{n}$ and $t^{\overline{n+1}},$ respectively, yields:\\
        $\bullet$ \textbf{Step I}\\
        \textbf{Predictor}:
        \begin{equation*}
        \left(\frac{1}{2}u_{t}^{n}-\frac{u_{H}^{\overline{n+1}}-u_{H}^{n}}{\Delta t},v_{H}\right)+a\left(u^{n}-u^{n}_{H},v_{H}\right)+b(u^{n},u^{n},v_{H})-
        b(u_{H}^{n},u_{H}^{n},v_{H})-\chi(v_{H},p^{n})+
       \end{equation*}
       \begin{equation}\label{4n}
        \chi(v_{H},p_{H}^{n})+\chi(u^{n},q_{H})-\chi(u_{H}^{n},q_{H})=0;
       \end{equation}
       \textbf{Corrector}:
       \begin{equation*}
        \left(u_{t}^{\overline{n+1}}-\frac{u_{H}^{n+1}-\frac{u_{H}^{\overline{n+1}}+u_{H}^{n}}{2}}
        {\frac{\Delta t}{2}},v_{H}\right)+a\left(u^{\overline{n+1}}-u^{\overline{n+1}}_{H},v_{H}\right)
        +b(u^{\overline{n+1}},u^{\overline{n+1}},v_{H})-b(u_{H}^{\overline{n+1}},u_{H}^{\overline{n+1}},v_{H})
       \end{equation*}
       \begin{equation}\label{5n}
        -\chi(v_{H},p^{\overline{n+1}})+\chi(v_{H},p_{H}^{\overline{n+1}})+\chi(u^{\overline{n+1}},q_{H})
        -\chi(u_{H}^{\overline{n+1}},q_{H})=0;
       \end{equation}
       $\bullet$ \textbf{Step II}\\
       \begin{equation*}
        \left(u_{t}^{\overline{n}}-\frac{u_{h}^{n}-u_{h}^{n-1}}{\Delta t},v_{h}\right)
        +a\left(u^{n}-\frac{u^{n+1}_{h}+u^{n-1}_{h}}{2},v_{h}\right)
        +b(u^{n},u^{n},v_{h})-b(u_{H}^{n},u_{H}^{n},v_{h})
       \end{equation*}
       \begin{equation}\label{6n}
        -\chi(v_{h},p^{n})+\chi\left(v_{h},\frac{p_{h}^{n+1}+p_{h}^{n-1}}{2}\right)+\chi(u^{n},q_{h})
        -\chi\left(\frac{u_{h}^{n+1}+u_{h}^{n-1}}{2},q_{h}\right)=0.
       \end{equation}
        After straightforward computations and rearranging terms, equations $(\ref{4n})$-$(\ref{6n})$ give
       \begin{equation*}
        \frac{1}{\Delta t}\left(e_{uH}^{\overline{n+1}}-e_{uH}^{n},v_{H}\right)+a\left(e_{uH}^{n},v_{H}\right)
        -\chi(v_{H},e_{pH}^{n})+\chi(e_{uH}^{n},q_{H})=-\left(\frac{1}{2}u_{t}-\frac{u^{\overline{n+1}}
        -u^{n}}{\Delta t},v_{H}\right)
       \end{equation*}
       \begin{equation}\label{7n}
        -b(e_{uH}^{n},u^{n},v_{H})-b(u_{H}^{n},e_{uH}^{n},v_{H});
       \end{equation}
       \begin{equation*}
        \frac{1}{\Delta t}\left(2e_{uH}^{n+1}-e_{uH}^{\overline{n+1}}-e_{uH}^{n},v_{H}\right)+a\left(e_{uH}^{\overline{n+1}},
        v_{H}\right)
        -\chi(v_{H},e_{pH}^{\overline{n+1}})+\chi(e_{uH}^{\overline{n+1}},q_{H})=
       \end{equation*}
       \begin{equation}\label{8n}
        -\left(u_{t}^{\overline{n+1}}-\frac{u^{n+1}-\frac{u^{\overline{n+1}}+u^{n}}{2}}{\frac{\Delta t}{2}}
        ,v_{H}\right)-b(e_{uH}^{\overline{n+1}},u^{\overline{n+1}},v_{H})-b(u_{H}^{\overline{n+1}},
        e_{uH}^{\overline{n+1}},v_{H});
       \end{equation}
       \begin{equation*}
        \frac{1}{\Delta t}\left(e_{uh}^{n}-e_{uh}^{n-1},v_{h}\right)+\frac{1}{2}a\left(e_{uh}^{n+1}+e_{uh}^{n-1},v_{h}\right)
        -\chi\left(v_{h},p^{n}-\frac{p_{h}^{n+1}+p_{h}^{n-1}}{2}\right)+\frac{1}{2}\chi\left(e_{2h}^{n+1}
        +e_{2h}^{n-1},q_{h}\right)=
       \end{equation*}
       \begin{equation*}
        -\left(u_{t}^{n}-\frac{u^{n}-u^{n-1}}{\Delta t},v_{h}\right)-a\left(u^{n}-\frac{u^{n+1}
        +u^{n-1}}{2},v_{h}\right)-\frac{1}{2}\chi\left(2u^{n}-u^{n+1}-u^{n-1},q_{h}\right)-
       \end{equation*}
       \begin{equation}\label{9n}
        \frac{1}{2}\chi\left(e_{1h}^{n+1}+e_{1h}^{n-1},q_{h}\right)-b(e_{uH}^{n},u^{n},v_{h})
        -b(u_{H}^{n},e_{uH}^{n},v_{h}).
       \end{equation}

         Armed with the above tools, we are ready to state and prove some fundamental tools (Lemmas $\ref{l2}$-$\ref{l3}$) that we shall use for the proof of the main result of this paper
         (namely Theorem $\ref{t2}$).

        \begin{lemma}\label{l2}
        Let $n$ be a nonnegative integer, $n\geq1.$ Define the "predicted" consistency errors in the coarse-grid region by $\xi^{\overline{n+1}}_{H}(v_{H})=-\left(\frac{1}{2}u_{t}^{n}-\frac{u^{\overline{n+1}}-u^{n}}
        {\Delta t},v_{H}\right),$ and the consistency error by $\xi_{H}^{n+1}(v_{H})=\\
        -\left(u_{t}^{\overline{n+1}}-\frac{u^{n+1}-\frac{u^{\overline{n+1}}+u^{n}}{2}}{\frac{\Delta t}{2}},v_{H}\right),$ while the errors in the fine grid region are given by $\xi^{n}_{h}(v_{h})=\\
        -\left(u_{t}^{\overline{n}}-\frac{u^{n}-u^{n-1}}{\Delta t},v_{h}\right)-a\left(u^{n}
        -\frac{u^{n+1}+u^{n-1}}{2},v_{h}\right).$ Then it holds
       \begin{equation*}
        \frac{1}{4\Delta t}\left\{\|e_{2H}^{n+1}\|_{L^{2}_{f}}^{2}-\|e_{2H}^{n}\|_{L^{2}_{f}}^{2}-\|e_{2H}^{n+1}
        -e_{2H}^{n}\|_{L^{2}_{f}}^{2}-\frac{\Delta t^{2}}{2}(e_{2H}^{n},e_{2H,tt}^{\overline{n+1}})+4\Delta t a\left(e_{2H}^{n},e_{2H}^{n}\right)\right\}=\xi^{\overline{n+1}}_{H}(e_{2H}^{n})
       \end{equation*}
       \begin{equation*}
        -b(e_{1H}^{n},u^{n},e_{2H}^{n})-b(e_{2H}^{n},u^{n},e_{2H}^{n})-b(u_{H}^{n},e_{1H}^{n},e_{2H}^{n})
        -a\left(e_{1H}^{n},e_{2H}^{n}\right)-\chi(e_{1H}^{n},e_{pH}^{n})
       \end{equation*}
       \begin{equation}\label{19n}
        -\frac{1}{\Delta t}\left(e_{1H}^{\overline{n+1}}-e_{1H}^{n},e_{2H}^{n}\right);
       \end{equation}
       \begin{equation*}
        \frac{1}{4\Delta t}\left\{3\left(\|e_{2H}^{n+1}\|_{L^{2}_{f}}^{2}-\|e_{2H}^{n}\|_{L^{2}_{f}}^{2}\right)
        -\frac{\Delta t^{2}}{2}(e_{2H}^{n+1}-2e_{2H}^{n},e_{2H,tt}^{\overline{n+1}})+4\Delta t a\left(e_{2H}^{\overline{n+1}},e_{2H}^{\overline{n+1}}\right)\right\}=
       \end{equation*}
       \begin{equation*}
        \xi^{n+1}_{H}(e_{2H}^{\overline{n+1}})-b(e_{1H}^{\overline{n+1}},u^{\overline{n+1}},
        e_{2H}^{\overline{n+1}})-b(e_{2H}^{\overline{n+1}},u^{\overline{n+1}},e_{2H}^{\overline{n+1}})
        -b(u_{H}^{\overline{n+1}},e_{1H}^{\overline{n+1}},e_{2H}^{n+1})-\chi(e_{1H}^{\overline{n+1}},
        e_{pH}^{\overline{n+1}})
       \end{equation*}
       \begin{equation}\label{20n}
        -a\left(e_{1H}^{\overline{n+1}},e_{2H}^{\overline{n+1}}\right)-\frac{1}{\Delta t}\left(2e_{1H}^{n+1}
        -e_{1H}^{\overline{n+1}}-e_{1H}^{n},e_{2H}^{\overline{n+1}}\right);
       \end{equation}
       \begin{equation*}
        \frac{1}{4\Delta t}\left\{\|e_{2h}^{n+1}\|_{L^{2}_{f}}^{2}-\|e_{2h}^{n-1}\|_{L^{2}_{f}}^{2}
        -\Delta t^{2}(e_{2h}^{n+1}+e_{2h}^{n-1},e_{2h,tt}^{n})+\Delta t a\left(e_{2h}^{n+1}
        +e_{2h}^{n-1},e_{2h}^{n+1}+e_{2h}^{n-1}\right)\right\}=
       \end{equation*}
       \begin{equation*}
        \xi^{n}_{h}(\frac{e_{2h}^{n+1}+e_{2h}^{n-1}}{2})-b(e_{1H}^{n},u^{n},\frac{e_{2h}^{n+1}
        +e_{2h}^{n-1}}{2})-b(e_{2H}^{n},u^{n},\frac{e_{2h}^{n+1}+e_{2h}^{n-1}}{2})-b(u_{H}^{n},e_{1H}^{n},
        \frac{e_{2h}^{n+1}+e_{2h}^{n-1}}{2})
       \end{equation*}
       \begin{equation*}
        -b(u_{H}^{n},e_{2H}^{n},\frac{e_{2h}^{n+1}+e_{2h}^{n-1}}{2})-\frac{1}{4}a\left(e_{1h}^{n+1}
        +e_{1h}^{n-1},e_{2h}^{n+1}+e_{2h}^{n-1}\right)-\frac{1}{2\Delta t}\left(e_{1h}^{n}-e_{1h}^{n-1},
        e_{2h}^{n+1}+e_{2h}^{n-1}\right)
       \end{equation*}
       \begin{equation}\label{21n}
       -\frac{1}{2}\chi\left(e^{n+1}_{1h}+e^{n-1}_{1h},p^{n}-\lambda_{h}^{n}\right),
       \end{equation}
       where $\lambda_{h}^{n}=\frac{p_{h}^{n+1}+p_{h}^{n-1}}{2}.$
        \end{lemma}

        \begin{proof}
        We decompose both "predicted" and exact error terms into
       \begin{equation}\label{10an}
        e^{m}_{u\mu}=(u^{m}-\widetilde{u}^{m})+(\widetilde{u}^{m}-u_{\mu}^{m})=e_{1\mu}^{m}+e_{2\mu}^{m},
       \end{equation}
        where $m\in\{n,\overline{n+1},n+1\}$ and $\mu=H,h.$ In addition, we assume that the terms $\widetilde{u}^{m}\in V_{h},$ so that $e_{2\mu}^{m}\in V_{h},$ for $m\in\{n,\overline{n+1},n+1\}$ and $\mu=H,h.$ Relation $e_{2H}^{m}\in V_{h}$ holds because $V_{H}$ is a subset of $V_{h}.$ Replacing the terms $e^{m}_{u\mu}$ by $e_{1\mu}^{m}+e_{2\mu}^{m}$ into relations $(\ref{7n})$-$(\ref{9n})$
       provide the following error equations
       \begin{equation*}
        \frac{1}{\Delta t}\left(e_{2H}^{\overline{n+1}}-e_{2H}^{n},v_{H}\right)+a\left(e_{2H}^{n},v_{H}\right)-
        \chi(v_{H},e_{pH}^{n})+\chi(e_{2H}^{n},q_{H})=\xi^{\overline{n+1}}_{H}(v_{H})-\chi(e_{1H}^{n},q_{H})
        -b(e_{1H}^{n},u^{n},v_{H})
       \end{equation*}
       \begin{equation}\label{10n}
        -b(e_{2H}^{n},u^{n},v_{H})-b(u_{H}^{n},e_{1H}^{n},v_{H})-b(u_{H}^{n},e_{2H}^{n},v_{H})
        -a\left(e_{1H}^{n},v_{H}\right)-\frac{1}{\Delta t}\left(e_{1H}^{\overline{n+1}}-e_{1H}^{n},v_{H}\right);
       \end{equation}
       \begin{equation*}
        \frac{1}{\Delta t}\left(2e_{2H}^{n+1}-e_{2H}^{\overline{n+1}}-e_{2H}^{n},
        v_{H}\right)+a\left(e_{2H}^{\overline{n+1}},v_{H}\right)-\chi(v_{H},e_{p_{pH}}^{\overline{n+1}})
        +\chi(e_{2H}^{\overline{n+1}},q_{H})=\xi^{n+1}_{H}(v_{H})
       \end{equation*}
       \begin{equation*}
        -b(e_{2H}^{\overline{n+1}},u^{\overline{n+1}},v_{H})-b(e_{1H}^{\overline{n+1}},u^{\overline{n+1}},v_{H})
        -b(u_{H}^{\overline{n+1}},e_{2H}^{\overline{n+1}},v_{H})-b(u_{H}^{\overline{n+1}},
        e_{1H}^{\overline{n+1}},v_{H})-a\left(e_{1H}^{\overline{n+1}},v_{H}\right)
       \end{equation*}
       \begin{equation}\label{11n}
       -\chi(e_{1H}^{\overline{n+1}},q_{H})-\frac{1}{\Delta t}\left(2e_{1H}^{n+1}-e_{1H}^{\overline{n+1}}-e_{1H}^{n},v_{H}\right);
       \end{equation}
       \begin{equation*}
        \frac{1}{\Delta t}\left(e_{2h}^{n}-e_{2h}^{n-1},v_{h}\right)+\frac{1}{2}a\left(e_{2h}^{n+1}+e_{2h}^{n-1},v_{h}\right)
        -\chi\left(v_{h},p^{n}-\frac{p_{h}^{n+1}+p_{h}^{n-1}}{2}\right)+\frac{1}{2}\chi\left(e_{uh}^{n+1}
        +e_{uh}^{n-1},q_{h}\right)=
       \end{equation*}
       \begin{equation*}
        \xi^{n}_{h}(v_{h})-b(e_{2H}^{n},u^{n},v_{h})-b(e_{1H}^{n},u^{n},v_{h})-b(u_{H}^{n},e_{2H}^{n},v_{h})
        -b(u_{H}^{n},e_{1H}^{n},v_{h})-\frac{1}{2}a\left(e_{1h}^{n+1}+e_{1h}^{n-1},v_{h}\right)
       \end{equation*}
       \begin{equation}\label{12n}
        -\frac{1}{2}\chi\left(2u^{n}-u^{n+1}-u^{n-1},q_{h}\right)-\frac{1}{\Delta t}
        \left(e_{1h}^{n}-e_{1h}^{n-1},v_{h}\right).
       \end{equation}
       Putting $\chi_{H}^{m}(v_{H},q_{H})=\chi(v_{H},e^{m}_{pH})-\chi(e_{1H}^{m},q_{H}),$ for $m\in\{\overline{n+1},n\},$ and
       $\chi_{h}^{n}(v_{h},q_{h})=\chi(v_{h},p^{n}-\frac{p_{h}^{n+1}+p_{h}^{n-1}}{2})-\frac{1}{2}
       \chi(e^{n+1}_{uh}+e^{n-1}_{uh}-(u^{n+1}+u^{n-1}-2u^{n}),q_{h}).$ Taking into account the requirement: $\chi(\widetilde{u}_{\mu},q_{\mu})=0,$ for all $\widetilde{u}_{\mu}\in V_{\mu},$ and for every
       $q_{\mu}\in Q_{\mu},$ where $\mu=h,H,$ simple calculations yield
       \begin{equation*}
        \frac{1}{\Delta t}\left(e_{2H}^{\overline{n+1}}-e_{2H}^{n},v_{H}\right)+a\left(e_{2H}^{n},v_{H}\right)
        =\xi^{\overline{n+1}}_{H}(v_{H})+\chi_{H}^{n}(v_{H},q_{H})-b(e_{1H}^{n},u^{n},v_{H})-
       \end{equation*}
       \begin{equation}\label{13n}
        b(e_{2H}^{n},u^{n},v_{H})-b(u_{H}^{n},e_{1H}^{n},v_{H})-b(u_{H}^{n},e_{2H}^{n},v_{H})
        -a\left(e_{1H}^{n},v_{H}\right)-\frac{1}{\Delta t}\left(e_{1H}^{\overline{n+1}}-e_{1H}^{n},v_{H}\right);
       \end{equation}
       \begin{equation*}
        \frac{1}{\Delta t}\left(2e_{2H}^{n+1}-e_{2H}^{\overline{n+1}}-e_{2H}^{n},v_{H}\right)+a\left(e_{2H}^{\overline{n+1}},
        v_{H}\right)=\xi^{n+1}_{H}(v_{H})+\chi_{H}^{\overline{n+1}}(v_{H},q_{H})-\frac{1}{\Delta t}\left(2e_{1H}^{n+1}-e_{1H}^{\overline{n+1}}
        -e_{1H}^{n},v_{H}\right)
       \end{equation*}
       \begin{equation}\label{14n}
        -b(e_{2H}^{\overline{n+1}},u^{\overline{n+1}},v_{H})-b(e_{1H}^{\overline{n+1}},u^{\overline{n+1}},v_{H})
        -b(u_{H}^{\overline{n+1}},e_{2H}^{\overline{n+1}},v_{H})-b(u_{H}^{\overline{n+1}},e_{1H}^{\overline{n+1}},
        v_{H})-a\left(e_{1H}^{\overline{n+1}},v_{H}\right);
       \end{equation}
       \begin{equation*}
        \frac{1}{\Delta t}\left(e_{2h}^{n}-e_{2h}^{n-1},v_{h}\right)+\frac{1}{2}a\left(e_{2h}^{n+1}+e_{2h}^{n-1},v_{h}\right)=
        \xi^{n}_{h}(v_{h})+\chi_{h}^{n}(v_{h},q_{h})-b(e_{2H}^{n},u^{n},v_{h})-b(e_{1H}^{n},u^{n},v_{h})
       \end{equation*}
       \begin{equation}\label{15n}
        -b(u_{H}^{n},e_{2H}^{n},v_{h})-b(u_{H}^{n},e_{1H}^{n},v_{h})-\frac{1}{2}a\left(e_{1h}^{n+1}
        +e_{1h}^{n-1},v_{h}\right)-\frac{1}{\Delta t}\left(e_{1h}^{n}-e_{1h}^{n-1},v_{h}\right).
       \end{equation}
       Since $e_{2\mu}^{m}\in V_{h}$ and $p\in Q_{\mu},$ for $\mu=h,H,$ and for every $m\in\{n,\overline{n+1},n+1\},$ taking $v_{H}=e_{2H}^{n}$ and $q_{H}=e_{pH}^{n},$ $v_{H}=e_{2H}^{\overline{n+1}}$ and $q_{H}=e_{pH}^{\overline{n+1}},$ $v_{h}=\frac{e_{2h}^{n+1}+e_{2h}^{n-1}}{2}$
        and $q_{h}=p^{n}-\frac{p_{h}^{n+1}+p_{h}^{n-1}}{2},$ in equations
       $(\ref{13n}),$ $(\ref{14n})$ and $(\ref{15n}),$ respectively, to get
       \begin{equation*}
        \frac{1}{\Delta t}\left(e_{2H}^{\overline{n+1}}-e_{2H}^{n},e_{2H}^{n}\right)+a\left(e_{2H}^{n},
        e_{2H}^{n}\right)=\xi^{\overline{n+1}}_{H}(e_{2H}^{n})+\chi_{H}^{n}(e_{2H}^{n},e_{pH}^{n})
        -b(e_{1H}^{n},u^{n},e_{2H}^{n})-
       \end{equation*}
       \begin{equation}\label{16n}
        b(e_{2H}^{n},u^{n},e_{2H}^{n})-b(u_{H}^{n},e_{1H}^{n},e_{2H}^{n})-a\left(e_{1H}^{n},e_{2H}^{n}\right)
        -\frac{1}{\Delta t}\left(e_{1H}^{\overline{n+1}}-e_{1H}^{n},e_{2H}^{n}\right);
       \end{equation}
       \begin{equation*}
        \frac{1}{\Delta t}\left(2e_{2H}^{n+1}-e_{2H}^{\overline{n+1}}-e_{2H}^{n},e_{2H}^{\overline{n+1}}\right)
        +a\left(e_{2H}^{\overline{n+1}},e_{2H}^{\overline{n+1}}\right)=\xi^{n+1}_{H}(e_{2H}^{\overline{n+1}})
        +\chi_{H}^{\overline{n+1}}(e_{2H}^{\overline{n+1}},e_{pH}^{\overline{n+1}})-a\left(e_{1H}^{\overline{n+1}},
        e_{2H}^{\overline{n+1}}\right)
       \end{equation*}
       \begin{equation}\label{17n}
        -b(e_{2H}^{\overline{n+1}},u^{\overline{n+1}},e_{2H}^{\overline{n+1}})-b(e_{1H}^{\overline{n+1}},
        u^{\overline{n+1}},
        e_{2H}^{\overline{n+1}})-b(u_{H}^{\overline{n+1}},e_{1H}^{\overline{n+1}},e_{2H}^{\overline{n+1}})
        -\frac{1}{\Delta t}\left(2e_{1H}^{n+1}-e_{1H}^{\overline{n+1}}-e_{1H}^{n},e_{2H}^{\overline{n+1}}\right);
       \end{equation}
       \begin{equation*}
        \frac{1}{2\Delta t}\left(e_{2h}^{n}-e_{2h}^{n-1},e_{2h}^{n+1}+e_{2h}^{n-1}\right)+\frac{1}{4}a\left(e_{2h}^{n+1}
        +e_{2h}^{n-1},e_{2h}^{n+1}+e_{2h}^{n-1}\right)=\xi^{n}_{h}(\frac{e_{2h}^{n+1}+e_{2h}^{n-1}}{2})
       \end{equation*}
       \begin{equation*}
        +\chi_{h}^{n}(\frac{e_{2h}^{n+1}+e_{2h}^{n-1}}{2},p^{n}-\frac{p_{h}^{n+1}+p_{h}^{n-1}}{2})-
        b(e_{2H}^{n},u^{n},\frac{e_{2h}^{n+1}+e_{2h}^{n-1}}{2})-b(e_{1H}^{n},u^{n},\frac{e_{2h}^{n+1}
        +e_{2h}^{n-1}}{2})
       \end{equation*}
       \begin{equation*}
        -b(u_{H}^{n},e_{1H}^{n},\frac{e_{2h}^{n+1}+e_{2h}^{n-1}}{2})-b(u_{H}^{n},e_{2H}^{n},\frac{e_{2h}^{n+1}
        +e_{2h}^{n-1}}{2})-\frac{1}{4}a\left(e_{1h}^{n+1}+e_{1h}^{n-1},e_{2h}^{n+1}+e_{2h}^{n-1}\right)
       \end{equation*}
       \begin{equation}\label{18n}
        -\frac{1}{2\Delta t}\left(e_{1h}^{n}-e_{1h}^{n-1},e_{2h}^{n+1}+e_{2h}^{n-1}\right).
       \end{equation}
       Following the MacCormack approach, the application of the Taylor series with time step
       $\frac{\Delta t}{2}$ for both "predicted" and exact values provides $\widetilde{u}^{n}=\widetilde{u}^{\overline{n+1}}-\frac{\Delta t}{2}\widetilde{u}_{t}^{\overline{n+1}}
       +\frac{\Delta t^{2}}{8}\widetilde{u}_{tt}^{\overline{n+1}}+O(\Delta t^{3})$ and $\widetilde{u}^{n+1}=\widetilde{u}^{\overline{n+1}}+
       \frac{\Delta t}{2}\widetilde{u}_{t}^{\overline{n+1}}+\frac{\Delta t^{2}}{8}
       \widetilde{u}_{tt}^{\overline{n+1}}+O(\Delta t^{3}),$ adding both expansions, it is obvious that $\widetilde{u}^{\overline{n+1}}=\frac{1}{2}(\widetilde{u}^{n+1}+\widetilde{u}^{n})-
       \frac{\Delta t^{2}}{8}\widetilde{u}_{tt}^{\overline{n+1}}+O(\Delta t^{3}).$ In way similar, $u_{2H}^{\overline{n+1}}=\frac{1}{2}(u_{2H}^{n+1}
       +u_{2H}^{n})-\frac{\Delta t^{2}}{8}u_{2H,tt}^{\overline{n+1}}+O(\Delta t^{3}),$ where $u_{2H,tt}^{\overline{n+1}}=\left(\frac{\partial^{2}u_{2H}}{\partial t^{2}}\right)^{\overline{n+1}}.$
        Utilizing this, simple calculations give
       $e_{2H}^{\overline{n+1}}=\frac{1}{2}(e_{2H}^{n+1}+e_{2H}^{n})-\frac{\Delta t^{2}}{8}e_{2H,tt}^{\overline{n+1}}+O(\Delta t^{3}).$ Likewise, applying the Taylor series with time step $\Delta t$, it is easy to see that $e_{2h}^{n}=\frac{1}{2}(e_{2h}^{n+1}+e_{2h}^{n-1})-\frac{\Delta t^{2}}{2}e_{2h,tt}^{n}+O(\Delta t^{3}).$ Substituting $e_{2H}^{\overline{n+1}}$ into relations $(\ref{16n}),$ $(\ref{17n})$ and $e_{2h}^{n}$ into equation $(\ref{18n}),$ straightforward computations result in
       \begin{equation*}
        \frac{1}{4\Delta t}\left\{\|e_{2H}^{n+1}\|_{L^{2}_{f}}^{2}-\|e_{2H}^{n}\|_{L^{2}_{f}}^{2}-\|e_{2H}^{n+1}
        -e_{2H}^{n}\|_{L^{2}_{f}}^{2}-\frac{\Delta t^{2}}{2}(e_{2H}^{n},e_{2H,tt}^{\overline{n+1}})+4\Delta t a\left(e_{2H}^{n},e_{2H}^{n}\right)\right\}= \xi^{\overline{n+1}}_{H}(e_{2H}^{n})
       \end{equation*}
       \begin{equation*}
        -b(e_{1H}^{n},u^{n},e_{2H}^{n})-b(e_{2H}^{n},u^{n},e_{2H}^{n})-b(u_{H}^{n},e_{1H}^{n},e_{2H}^{n})
        -a\left(e_{1H}^{n},e_{2H}^{n}\right)-\chi(e_{1H}^{n},e_{pH}^{n})
       \end{equation*}
       \begin{equation*}
        -\frac{1}{\Delta t}\left(e_{1H}^{\overline{n+1}}-e_{1H}^{n},e_{2H}^{n}\right)+O(\Delta t^{2});
       \end{equation*}
       \begin{equation*}
        \frac{1}{4\Delta t}\left\{3\left(\|e_{2H}^{n+1}\|_{L^{2}_{f}}^{2}-\|e_{2H}^{n}\|_{L^{2}_{f}}^{2}\right)
        -\frac{\Delta t^{2}}{2}(e_{2H}^{n+1}-2e_{2H}^{n},e_{2H,tt}^{\overline{n+1}})+4\Delta t a\left(e_{2H}^{\overline{n+1}},e_{2H}^{\overline{n+1}}\right)\right\}=
       \end{equation*}
       \begin{equation*}
        \xi^{n+1}_{H}(e_{2H}^{\overline{n+1}})-b(e_{1H}^{\overline{n+1}},u^{\overline{n+1}},
        e_{2H}^{\overline{n+1}})-b(e_{2H}^{\overline{n+1}},u^{\overline{n+1}},e_{2H}^{\overline{n+1}})
        -b(u_{H}^{\overline{n+1}},e_{1H}^{\overline{n+1}},e_{2H}^{n+1})-\chi(e_{1H}^{\overline{n+1}},
        e_{pH}^{\overline{n+1}})
       \end{equation*}
       \begin{equation*}
        -a\left(e_{1H}^{\overline{n+1}},e_{2H}^{\overline{n+1}}\right)-\frac{1}{\Delta t}\left(2e_{1H}^{n+1}
        -e_{1H}^{\overline{n+1}}-e_{1H}^{n},e_{2H}^{\overline{n+1}}\right)+O(\Delta t^{2});
       \end{equation*}
       \begin{equation*}
        \frac{1}{4\Delta t}\left\{\|e_{2h}^{n+1}\|_{L^{2}_{f}}^{2}-\|e_{2h}^{n-1}\|_{L^{2}_{f}}^{2}
        -\Delta t^{2}(e_{2h}^{n+1}+e_{2h}^{n-1},e_{2h,tt}^{n})+\Delta t a\left(e_{2h}^{n+1}+e_{2h}^{n-1},
        e_{2h}^{n+1}+e_{2h}^{n-1}\right)\right\}=
       \end{equation*}
       \begin{equation*}
        \xi^{n}_{h}(\frac{e_{2h}^{n+1}+e_{2h}^{n-1}}{2})-b(e_{1H}^{n},u^{n},\frac{e_{2h}^{n+1}
        +e_{2h}^{n-1}}{2})-b(e_{2H}^{n},u^{n},\frac{e_{2h}^{n+1}+e_{2h}^{n-1}}{2})-b(u_{H}^{n},e_{1H}^{n},
        \frac{e_{2h}^{n+1}+e_{2h}^{n-1}}{2})
       \end{equation*}
       \begin{equation*}
        -b(u_{H}^{n},e_{2H}^{n},\frac{e_{2h}^{n+1}+e_{2h}^{n-1}}{2})-\frac{1}{4}a\left(e_{1h}^{n+1}
        +e_{1h}^{n-1},e_{2h}^{n+1}+e_{2h}^{n-1}\right)-\frac{1}{2\Delta t}\left(e_{1h}^{n}
        -e_{1h}^{n-1},e_{2h}^{n+1}+e_{2h}^{n-1}\right)
       \end{equation*}
       \begin{equation*}
       -\frac{1}{2}\chi\left(e^{n+1}_{1h}+e^{n-1}_{1h},p^{n}-\lambda_{h}^{n}\right)+O(\Delta t^{2}).
       \end{equation*}
        \end{proof}
        Neglecting the infinitesimal term $O(\Delta t^{2}),$ the proof of Lemma $\ref{l2}$ is completed thanks to equalities $\chi_{h}^{n}(\frac{e_{2h}^{n+1}+e_{2h}^{n-1}}{2},p^{n}-\frac{p_{h}^{n+1}+p_{h}^{n-1}}{2})=
        \chi(\frac{e_{2h}^{n+1}+e_{2h}^{n-1}}{2},p^{n}-\frac{p_{h}^{n+1}+p_{h}^{n-1}}{2})-\frac{1}{2}\chi(e^{n+1}_{uh}+e^{n-1}_{uh}-(u^{n+1}
        +u^{n-1}-2u^{n}),p^{n}-\frac{p_{h}^{n+1} +p_{h}^{n-1}}{2})=-\frac{1}{2}\chi(e_{1h}^{n+1}+e_{1h}^{n-1},p^{n}
        -\frac{p_{h}^{n+1}+p_{h}^{n-1}}{2}).$ In fact, $u^{m}\in F,$ $p^{m}$ and $p_{h}^{m}\in Q_{h}\subset Q,$ for $m=n-1,n,n+1,$ imply $\chi(u^{n+1}+u^{n-1}-2u^{n},p^{n}-\frac{p_{h}^{n+1}+p_{h}^{n-1}}{2})=0,$ where $F$ and $Q$ are given by relations $(\ref{14})$ and $(\ref{5}),$ respectively.\\

        \begin{remark}
        In general, the time step $\Delta t$ is small, so the truncation of the term $O(\Delta t^{2})$ does not compromise the result of this work.
        \end{remark}

         \begin{lemma}\label{l3} Let $f\in L^{2}(0,T;X)$ and $u_{0}\in D(A)\cap F$ be given. Consider the two-level hybrid algorithm $(\ref{36})$ and $(\ref{38}).$ Without time step restriction, the following estimate holds over $0\leq t^{n}<+\infty.$ More specifically, for all $N\in\mathbb{N},$ $N\neq0,$ we have
        \begin{equation*}
         \|e_{uH}^{N}\|_{X}^{2}+\frac{\Delta t}{6}\overset{N-1}{\underset{n=0}\sum}\left(\|e_{uH}^{\overline{n+1}}\|_{W}^{2}
         +\|e_{uH}^{n}\|_{W}^{2}\right) \leq \widetilde{C}\left\{H^{2m}\left(H^{2}\||u|\|_{L^{\infty}(0,T;W^{m+1})}^{2}+
         \||u|\|_{L^{2}(0,T;W^{m+1})}^{2}+\right.\right.
        \end{equation*}
        \begin{equation}\label{43}
        \left.\left.H^{2}\||u_{t}|\|_{L^{2}(0,T;W^{m+1})}^{2}\right)+\Delta t^{2}\left(\||u_{tt}|\|_{L^{2}(0,T;W^{m+1})}^{2}+\Delta t^{2}
        \||u_{ttt}|\|_{L^{2}(0,T;W^{m+1})}^{2}\right)+\||\nabla u|\|_{L^{2}(0,T;W^{m+1})}^{2}\right\},
        \end{equation}
        for all $m\in\mathbb{N},$ where $\widetilde{C}$ is a positive constant.
        \end{lemma}

        \begin{proof}
        First, using the expansion $e_{2H}^{\overline{n+1}}=\frac{1}{2}(e_{2H}^{n+1}+e_{2H}^{n})
        -\frac{\Delta t^{2}}{8}e_{2H,tt}^{\overline{n+1}}+O(\Delta t^{3}),$ we have
        \begin{equation*}
            \left(e_{2H}^{\overline{n+1}},e_{2H,tt}^{\overline{n+1}}\right)=\left(\frac{1}{2}(e_{2H}^{n+1}
        +e_{2H}^{n})-\frac{\Delta t^{2}}{8}e_{2H,tt}^{\overline{n+1}}+O(\Delta t^{3}),
        e_{2H,tt}^{\overline{n+1}}\right)=\frac{1}{2}\left(e_{2H}^{n+1}+e_{2H}^{n},e_{2H,tt}^{\overline{n+1}}\right)
        \end{equation*}
        \begin{equation*}
            -\frac{\Delta t^{2}}{8}\|e_{2H,tt}^{\overline{n+1}}\|_{L^{2}_{f}}^{2}+O(\Delta t^{3}).
        \end{equation*}
        From this, we get
        \begin{equation*}
        \frac{\Delta t^{2}}{2}\left(e_{2H}^{n+1}-e_{2H}^{n},e_{2H,tt}^{\overline{n+1}}\right)=
        \Delta t^{2}\left(\frac{e_{2H}^{n+1}+e_{2H}^{n}}{2}-e_{2H}^{n},e_{2H,tt}^{\overline{n+1}}\right)=
        -\Delta t^{2}\left(e_{2H}^{n},e_{2H,tt}^{\overline{n+1}}\right)+\Delta t^{2}
        \left(e_{2H}^{\overline{n+1}},e_{2H,tt}^{\overline{n+1}}\right)+O(\Delta t^{3}).
        \end{equation*}

        This fact together with equations $(\ref{19n})$ and $(\ref{20n})$ yield
       \begin{equation*}
        \frac{1}{4\Delta t}\left\{4\left(\|e_{2H}^{n+1}\|_{L^{2}_{f}}^{2}-\|e_{2H}^{n}\|_{L^{2}_{f}}^{2}\right)
        -\|e_{2H}^{n+1}-e_{2H}^{n}\|_{L^{2}_{f}}^{2}-\Delta t^{2}\left[\left(e_{2H}^{\overline{n+1}},
        e_{2H,tt}^{\overline{n+1}}\right)-\left(e_{2H}^{n},e_{2H,tt}^{\overline{n+1}}
        \right)\right]\right\}+a\left(e_{2H}^{\overline{n+1}},e_{2H}^{\overline{n+1}}\right)+
       \end{equation*}
       \begin{equation*}
        +a\left(e_{2H}^{n},e_{2H}^{n}\right)=\xi^{\overline{n+1}}_{H}(e_{2H}^{n})
        +\xi^{n+1}_{H}(e_{2H}^{\overline{n+1}})+\chi_{H}^{n}(e_{2H}^{n},e_{pH}^{n})+
        \chi_{H}^{\overline{n+1}}(e_{2H}^{\overline{n+1}},e_{pH}^{\overline{n+1}})-b(e_{1H}^{n},u^{n},e_{2H}^{n})
       \end{equation*}
       \begin{equation*}
        -b(e_{1H}^{\overline{n+1}},u^{\overline{n+1}},e_{2H}^{\overline{n+1}})-b(e_{2H}^{\overline{n+1}},
        u^{\overline{n+1}},e_{2H}^{\overline{n+1}})-b(u_{H}^{\overline{n+1}},e_{1H}^{\overline{n+1}},
        e_{2H}^{\overline{n+1}})-b(e_{2H}^{n},u^{n},e_{2H}^{n})-b(u_{H}^{n},e_{1H}^{n},e_{2H}^{n})
       \end{equation*}
       \begin{equation}\label{22n}
        -a\left(e_{1H}^{\overline{n+1}},e_{2H}^{\overline{n+1}}\right)-a\left(e_{1H}^{n},e_{2H}^{n}\right)
        -\frac{1}{\Delta t}\left(2e_{1H}^{n+1}-e_{1H}^{\overline{n+1}}-e_{1H}^{n},e_{2H}^{n+1}\right)-
        \frac{1}{\Delta t}\left(e_{1H}^{\overline{n+1}}-e_{1H}^{n},e_{2H}^{\overline{n+1}}\right).
       \end{equation}
       Since $\chi_{H}^{m}(e_{2H}^{m},e_{pH}^{m})=\chi(e_{2H}^{m},e^{m}_{pH})-\chi(e_{1H}^{m},e_{pH}^{m}),$ for
       $m\in\{\overline{n+1},n\},$ it is obvious that $\chi_{H}^{m}(e_{2H}^{m},e_{pH}^{m})=0.$
       Indeed, because $e_{pH}^{m}\in Q_{H}\subset Q_{h}\subset Q,$ $\chi(e_{2H}^{m},e^{m}_{pH})=0$ and
       $e_{1H}^{m}=u^{m}-\widetilde{u}^{m},$ with $u^{m}\in F$ ($u$ is the exact solution) and $\widetilde{u}^{m}\in V_{h}$ imply $\chi(e_{1H}^{m},e^{m}_{pH})=0,$ where $F$ and $Q$ are given by relations $(\ref{14})$ and $(\ref{5}),$ respectively. Using this, equation $(\ref{22n})$ becomes
       \begin{equation*}
        \frac{1}{4\Delta t}\left\{4\left(\|e_{2H}^{n+1}\|_{L^{2}_{f}}^{2}-\|e_{2H}^{n}\|_{L^{2}_{f}}^{2}\right)
        -\Delta t^{2}\left[\left(e_{2H}^{\overline{n+1}},e_{2H,tt}^{\overline{n+1}}\right)-\left(e_{2H}^{n},e_{2H,tt}^{\overline{n+1}}
        \right)\right]\right\}+a\left(e_{2H}^{\overline{n+1}},e_{2H}^{\overline{n+1}}\right)+a\left(e_{2H}^{n},e_{2H}^{n}\right)=
       \end{equation*}
       \begin{equation*}
        \frac{1}{4\Delta t}\|e_{2H}^{n+1}-e_{2H}^{n}\|_{L^{2}_{f}}^{2}+\xi^{\overline{n+1}}_{H}(e_{2H}^{n})+\xi^{n+1}_{H}(e_{2H}^{\overline{n+1}})
        -a\left(e_{1H}^{n},e_{2H}^{n}\right)-a\left(e_{1H}^{\overline{n+1}},e_{2H}^{\overline{n+1}}\right)-b(e_{1H}^{n},u^{n},e_{2H}^{n})
       \end{equation*}
       \begin{equation*}
       -b(e_{1H}^{\overline{n+1}},u^{\overline{n+1}},e_{2H}^{\overline{n+1}})
        -b(e_{2H}^{\overline{n+1}},u^{\overline{n+1}},e_{2H}^{\overline{n+1}})-b(u_{H}^{\overline{n+1}},e_{1H}^{\overline{n+1}},
        e_{2H}^{\overline{n+1}})-b(e_{2H}^{n},u^{n},e_{2H}^{n})-b(u_{H}^{n},e_{1H}^{n},e_{2H}^{n})
       \end{equation*}
       \begin{equation}\label{23n}
        -\frac{1}{\Delta t}
        \left(2e_{1H}^{n+1}-e_{1H}^{\overline{n+1}}-e_{1H}^{n},e_{2H}^{n+1}\right)-\frac{1}{\Delta t}
        \left(e_{1H}^{\overline{n+1}}-e_{1H}^{n},e_{2H}^{\overline{n+1}}\right).
       \end{equation}
       Applying the Cauchy-Schwarz and Young inequalities along with the Poincar\'{e} inequality (i.e., estimate $(\ref{34})$), we bound the terms $b(e_{2H}^{\overline{n+1}},u^{\overline{n+1}},e_{2H}^{\overline{n+1}})$ and $b(e_{2H}^{n},u^{n},e_{2H}^{n}),$ and using this, we bound each term of the RHS of estimate $(\ref{23n}).$ Combining estimates $(\ref{17})$ and $(\ref{34}),$ simple calculations provide
       \begin{equation*}
        b(e_{2H}^{\overline{n+1}},u^{\overline{n+1}},e_{2H}^{\overline{n+1}})\leq \alpha\|e_{2H}^{\overline{n+1}}\|_{L_{f}^{2}}^{\frac{1}{2}}
        \|e_{2H}^{\overline{n+1}}\|_{W}^{\frac{1}{2}}\|u^{\overline{n+1}}\|_{W}\|e_{2H}^{\overline{n+1}}\|_{L_{f}^{2}}^{\frac{1}{2}}
        \|e_{2H}^{\overline{n+1}}\|_{W}^{\frac{1}{2}}+\alpha\|e_{2H}^{\overline{n+1}}\|_{W}\|u^{\overline{n+1}}\|_{L_{f}^{2}}^{\frac{1}{2}}
        \|u^{\overline{n+1}}\|_{W}^{\frac{1}{2}}
       \end{equation*}
       \begin{equation}\label{25n}
        \|e_{2H}^{\overline{n+1}}\|_{L_{f}^{2}}^{\frac{1}{2}}\|e_{2H}^{\overline{n+1}}\|_{W}^{\frac{1}{2}}\leq
        2\alpha(\lambda_{1}\nu)^{\frac{-1}{2}}\|e_{2H}^{\overline{n+1}}\|_{W}^{2}\|u^{\overline{n+1}}\|_{W}\leq \frac{\alpha^{2}}{\lambda_{1}\nu}\|u^{\overline{n+1}}\|_{W}^{2}+\|e_{2H}^{\overline{n+1}}\|_{W}^{4}.
       \end{equation}
        Similarly,
        \begin{equation}\label{26n}
         b(e_{2H}^{n},u^{n},e_{2H}^{n})\leq \frac{\alpha^{2}}{\lambda_{1}\nu}\|u^{n}\|_{W}^{2}+\|e_{2H}^{n}\|_{W}^{4}.
        \end{equation}
        On the other hand,
       \begin{equation*}
       \bullet\text{\,\,} \xi^{\overline{n+1}}(e_{2H}^{n})=-\left(\frac{1}{2}u_{t}^{n}-\frac{u^{\overline{n+1}}-u^{n}}
       {\Delta t}, e_{2H}^{n}\right)\leq\frac{1}{\sqrt{\lambda_{1}\nu}}\left\|\frac{1}{2}u_{t}^{n}-\frac{u^{\overline{n+1}}-u^{n}}
       {\Delta t}\right\|_{L_{f}^{2}}\|e_{2H}^{n}\|_{W}\leq \frac{1}{6}\|e_{2H}^{n}\|_{W}^{2}+
       \end{equation*}
       \begin{equation}\label{27n}
        \frac{3}{2\lambda_{1}\nu}\left\|\frac{1}{2}u_{t}^{n}-\frac{u^{\overline{n+1}}-u^{n}}{\Delta t}\right\|_{L_{f}^{2}}^{2};
       \end{equation}
       \begin{equation*}
        b(e_{1H}^{\overline{n+1}},u^{\overline{n+1}},e_{2H}^{\overline{n+1}})\leq\alpha\|e_{1H}^{\overline{n+1}}\|_{L_{f}^{2}}^{\frac{1}{2}}
        \|e_{1H}^{\overline{n+1}}\|_{W}^{\frac{1}{2}}\|u^{\overline{n+1}}\|_{W}\|e_{2H}^{\overline{n+1}}\|_{L_{f}^{2}}^{\frac{1}{2}}
        \|e_{2H}^{\overline{n+1}}\|_{W}^{\frac{1}{2}}+\alpha\|e_{1H}^{\overline{n+1}}\|_{W}\|u^{\overline{n+1}}\|_{L_{f}^{2}}^{\frac{1}{2}}
        \|u^{\overline{n+1}}\|_{W}^{\frac{1}{2}}
       \end{equation*}
       \begin{equation}\label{28n}
        \|e_{2H}^{\overline{n+1}}\|_{L_{f}^{2}}^{\frac{1}{2}}\|e_{2H}^{\overline{n+1}}\|_{W}^{\frac{1}{2}}\leq
        \frac{2\alpha}{\sqrt{\lambda_{1}\nu}}\|e_{1H}^{\overline{n+1}}\|_{W}\|e_{2H}^{\overline{n+1}}\|_{W}\|u^{\overline{n+1}}\|_{W}\leq \frac{\alpha^{2}}{\lambda_{1}\nu}\|u^{\overline{n+1}}\|_{W}^{2}\|e_{2H}^{\overline{n+1}}\|_{W}^{2}+
        \|e_{1H}^{\overline{n+1}}\|_{W}^{2}.
       \end{equation}
       Likewise
       \begin{equation}\label{29n}
        b(u_{H}^{\overline{n+1}},e_{1H}^{\overline{n+1}},e_{2H}^{\overline{n+1}})\leq \frac{\alpha^{2}}{\lambda_{1}\nu}\|u_{H}^{\overline{n+1}}\|_{W}^{2}\|e_{2H}^{\overline{n+1}}\|_{W}^{2}+
        \|e_{1H}^{\overline{n+1}}\|_{W}^{2}.
       \end{equation}
        The application of the Cauchy-Schwarz inequality to equation $(\ref{12})$ provides
        \begin{equation}\label{30n}
        a\left(e_{1H}^{\overline{n+1}},e_{2H}^{\overline{n+1}}\right)\leq \frac{3}{2}\left\|e_{1H}^{\overline{n+1}}\right\|_{W}^{2}+
        \frac{1}{6}\left\|e_{2H}^{\overline{n+1}}\right\|_{W}^{2}.
       \end{equation}
       Furthermore, utilizing the Cauchy-Schwarz inequality and estimate $(\ref{34}),$ it holds
       \begin{equation*}
        \frac{1}{\Delta t}\left(e_{1H}^{\overline{n+1}}-e_{1H}^{n},e_{2H}^{n}\right)\leq \frac{1}{\Delta t}
        \left\|e_{1H}^{\overline{n+1}}-e_{1H}^{n}\right\|_{L^{2}_{f}}\left\|e_{2H}^{n}\right\|_{L^{2}_{f}}\leq\frac{1}
        {\Delta t\sqrt{\lambda_{1}\nu}}\left\|e_{1H}^{\overline{n+1}}-e_{1H}^{n}\right\|_{L^{2}_{f}}\left\|e_{2H}^{n}\right\|_{W}\leq
       \end{equation*}
       \begin{equation}\label{31n}
        \frac{1}{6}\left\|e_{2H}^{n}\right\|_{W}^{2}+\frac{3}{2\lambda_{1}\nu\Delta t^{2}}
        \left\|e_{1H}^{\overline{n+1}}-e_{1H}^{n}\right\|_{L^{2}_{f}}^{2}.
       \end{equation}
       In way similar
       \begin{equation}\label{32n}
        \xi_{H}^{n+1}(e_{2H}^{\overline{n+1}})\leq \frac{3}{2\lambda_{1}\nu}\left\|u_{t}^{\overline{n+1}}-\frac{u^{n+1}-
        \frac{u^{\overline{n+1}}+u^{n}}{2}}{\frac{\Delta t}{2}}\right\|_{L^{2}_{f}}^{2}+\frac{1}{6}\left\|e_{2H}^{\overline{n+1}}\right\|_{W}^{2};
       \end{equation}
       \begin{equation}\label{33n}
        \Delta t^{2}\left(e_{2H}^{\overline{n+1}},e_{2H,tt}^{\overline{n+1}}\right)\leq
        \frac{3\Delta t^{4}}{2\lambda_{1}\nu}\left\|e_{2H,tt}^{\overline{n+1}}\right\|_{L^{2}_{f}}^{2}+
        \frac{1}{6}\left\|e_{2H}^{\overline{n+1}}\right\|_{W}^{2};
       \end{equation}
       \begin{equation}\label{33nn}
        \Delta t^{2}\left(e_{2H}^{n},e_{2H,tt}^{\overline{n+1}}\right)\leq\frac{3\Delta t^{4}}{2\lambda_{1}\nu}
        \left\|e_{2H,tt}^{\overline{n+1}}\right\|_{L^{2}_{f}}^{2}+\frac{1}{6}\left\|e_{2H}^{n}\right\|_{W}^{2};
       \end{equation}
       \begin{equation}\label{34n}
        b(e_{1H}^{n},u^{n},e_{2H}^{n})\leq \frac{\alpha^{2}}{\lambda_{1}\nu}\|u^{n}\|_{W}^{2}\|e_{2H}^{n}\|_{W}^{2}+
        \|e_{1H}^{n}\|_{W}^{2};
       \end{equation}
       \begin{equation}\label{35n}
        a\left(e_{1H}^{n},e_{2H}^{n}\right)\leq \frac{3}{2}\left\|e_{1H}^{n}\right\|_{W}^{2}+
        \frac{1}{6}\left\|e_{2H}^{n}\right\|_{W}^{2};
       \end{equation}
       \begin{equation}\label{36n}
        b(u_{H}^{n},e_{1H}^{n},e_{2H}^{n})\leq \frac{\alpha^{2}}{\lambda_{1}\nu}\|u_{H}^{n}\|_{W}^{2}\|e_{2H}^{n}\|_{W}^{2}
        +\|e_{1H}^{n}\|_{W}^{2};
       \end{equation}
       \begin{equation}\label{37n}
        \frac{1}{\Delta t}\left(2e_{1H}^{n+1}-e_{1H}^{\overline{n+1}}-e_{1H}^{n},e_{2H}^{\overline{n+1}}\right)\leq
        \frac{1}{6}\left\|e_{2H}^{\overline{n+1}}\right\|_{W}^{2}+\frac{3}{2\lambda_{1}\nu\Delta t^{2}}
        \left\|2e_{1H}^{n+1}-e_{1H}^{\overline{n+1}}-e_{1H}^{n}\right\|_{L^{2}_{f}}^{2}.
       \end{equation}
       Replacing the bounds given by estimates $(\ref{25n})$-$(\ref{37n})$ into estimate $(\ref{23n}),$ grouping the remaining
       terms and multiplying both sides of the final inequality by $4\Delta t$ to obtain
       \begin{equation*}
        4\left(\left\|e_{2H}^{n+1}\right\|_{L^{2}_{f}}^{2}-\left\|e_{2H}^{n}\right\|_{L^{2}_{f}}^{2}\right)
        +\frac{4}{3}\Delta t\left(\|e_{2H}^{\overline{n+1}}\|_{W}^{2}+\|e_{2H}^{n}\|_{W}^{2}\right)\leq
        \left\|e_{2H}^{n+1}-e_{2H}^{n}\right\|_{L^{2}_{f}}^{2}+\frac{3\Delta t^{4}}{\lambda_{1}\nu}
        \left\|e_{2H,tt}^{\overline{n+1}}\right\|_{L^{2}_{f}}^{2}+
       \end{equation*}
       \begin{equation*}
        \frac{6\Delta t}{\lambda_{1}\nu}\left\{\left\|\frac{1}{2}u_{t}^{n}-\frac{u^{\overline{n+1}}-u^{n}}
        {\Delta t}\right\|_{L^{2}_{f}}^{2}+\left\|u_{t}^{\overline{n+1}}-\frac{u^{n+1}-\frac{u^{\overline{n+1}}
        +u^{n}}{2}}{\frac{\Delta t}{2}}\right\|_{L^{2}_{f}}^{2}+\frac{7\lambda_{1}\nu}{3}\left(\|e_{1H}^{\overline{n+1}}\|_{W}^{2}
        +\|e_{1H}^{n}\|_{W}^{2}\right)\right.
       \end{equation*}
       \begin{equation*}
        +\frac{2\alpha^{2}}{3}\left\{(\|u^{\overline{n+1}}\|_{W}^{2}+\|u_{H}^{\overline{n+1}}\|_{W}^{2})
        \|e_{2H}^{\overline{n+1}}\|_{W}^{2}+(\|u^{n}\|_{W}^{2}+\|u_{H}^{n}\|_{W}^{2})\|e_{2H}^{n}\|_{W}^{2}+
        \|u^{\overline{n+1}}\|_{W}^{2}+\|u^{n}\|_{W}^{2}\right\}+
       \end{equation*}
       \begin{equation}\label{38n}
       \left.\frac{2\lambda_{1}\nu}{3}\left(\|e_{2H}^{\overline{n+1}}\|_{W}^{4}+\|e_{2H}^{n}\|_{W}^{4}\right)\right\}
        +\frac{6}{\lambda_{1}\nu\Delta t}\left\{\left\|e_{1H}^{\overline{n+1}}-e_{1H}^{n}\right\|_{L^{2}_{f}}^{2}+
       \left\|2e_{1}^{n+1}-e_{1H}^{\overline{n+1}}-e_{1}^{n}\right\|_{L^{2}_{f}}^{2}\right\}.
       \end{equation}
       Summing relation $(\ref{38n})$ up from $n=0,1,...,N-1,$ results in
        \begin{equation*}
        4\left\|e_{2H}^{N}\right\|_{L^{2}_{f}}^{2}+\frac{4}{3}\Delta t\underset{n=0}{\overset{N-1}\sum}
        \left(\|e_{2H}^{\overline{n+1}}\|_{W}^{2}+\|e_{2H}^{n}\|_{W}^{2}\right)\leq 4\left\|e_{2H}^{0}\right\|_{L^{2}_{f}}^{2}+
       \underset{n=0}{\overset{N-1}\sum}\left\|e_{2H}^{n+1}-e_{2H}^{n}\right\|_{L^{2}_{f}}^{2}+
        \frac{3\Delta t^{4}}{\lambda_{1}\nu}\underset{n=0}{\overset{N-1}\sum}\left\|e_{2H,tt}^{\overline{n+1}}\right\|_{L^{2}_{f}}^{2}+
       \end{equation*}
       \begin{equation*}
        \frac{6\Delta t}{\lambda_{1}\nu}\underset{n=0}{\overset{N-1}\sum}\left\{\left\|\frac{1}{2}u_{t}^{n}-\frac{u^{\overline{n+1}}-u^{n}}
        {\Delta t}\right\|_{L^{2}_{f}}^{2}+\left\|u_{t}^{\overline{n+1}}-\frac{u^{n+1}-\frac{u^{\overline{n+1}}
        +u^{n}}{2}}{\frac{\Delta t}{2}}\right\|_{L^{2}_{f}}^{2}+\frac{7\lambda_{1}\nu}{3}\left(\|e_{1H}^{\overline{n+1}}\|_{W}^{2}
        +\|e_{1H}^{n}\|_{W}^{2}\right)+\right.
       \end{equation*}
       \begin{equation*}
        \frac{2\alpha^{2}}{3}\left\{(\|u^{\overline{n+1}}\|_{W}^{2}+\|u_{H}^{\overline{n+1}}\|_{W}^{2})
        \|e_{2H}^{\overline{n+1}}\|_{W}^{2}+(\|u^{n}\|_{W}^{2}+\|u_{H}^{n}\|_{W}^{2})\|e_{2H}^{n}\|_{W}^{2}+
        \|u^{\overline{n+1}}\|_{W}^{2}+\|u^{n}\|_{W}^{2}\right\}+
       \end{equation*}
       \begin{equation}\label{39n}
       \left.\frac{2\lambda_{1}\nu}{3}\left(\|e_{2H}^{\overline{n+1}}\|_{W}^{4}+\|e_{2H}^{n}\|_{W}^{4}\right)\right\}+
        \frac{6}{\lambda_{1}\nu\Delta t}\underset{n=0}{\overset{N-1}\sum}\left\{\left\|e_{1H}^{\overline{n+1}}
        -e_{1H}^{n}\right\|_{L^{2}_{f}}^{2}+\left\|2e_{1}^{n+1}-e_{1H}^{\overline{n+1}}-e_{1}^{n}\right\|_{L^{2}_{f}}^{2}\right\}.
       \end{equation}
       Using the H\"{o}lder inequality together with equations $(\ref{85b}),$ each term of the RHS of estimate $(\ref{39n})$ can be bounded as follows
      \begin{equation*}
        \underset{n=0}{\overset{N-1}\sum}\left\|e_{jH}^{n+1}-e_{jH}^{n}\right\|_{L^{2}_{f}}^{2}=\underset{n=0}{\overset{N-1}\sum}
        \int_{\Omega_{f}}\left|\int_{t^{n}}^{t^{n+1}}e_{jH,t}(\theta)d\theta\right|^{2}dx\leq
        \underset{n=0}{\overset{N-1}\sum}\int_{\Omega_{f}}\Delta t\int_{t^{n}}^{t^{n+1}}|e_{jH,t}(\theta)|^{2}d\theta dx\leq
       \end{equation*}
      \begin{equation}\label{40n}
        \Delta t^{2}\underset{n=0}{\overset{N}\sum}\|e_{jH,t}(\theta^{n})\|_{L^{2}_{f}}^{2}=\Delta t
        \||e_{jH,t}|\|_{L^{2}(0,T;W)}^{2},
       \end{equation}
        for $j=1,2,$ where $\theta^{n}\in(t^{n},t^{n+1}).$ Since $t^{\overline{n+1}}\in(t^{n},t^{n+1}),$ it comes from the Triangular inequality that $\int_{t^{n}}^{t^{\overline{n+1}}}|e_{1H,t}(\theta)|^{2}d\theta dx\leq\int_{t^{n}}^{t^{n+1}}|e_{1H,t} (\theta)|^{2}d\theta dx.$ This fact provides
        \begin{equation}\label{41n}
        \underset{n=0}{\overset{N-1}\sum}\left\|e_{1H}^{\overline{n+1}}-e_{1H}^{n}\right\|_{L^{2}_{f}}^{2}\leq
        \Delta t\||e_{1H,t}|\|_{L^{2}(0,T;W)}^{2};\text{\,\,\,\,}\underset{n=0}{\overset{N-1}\sum}
        \left\|e_{2H}^{n+1}-e_{2H}^{n}\right\|_{L^{2}_{f}}^{2}\leq\Delta t\||e_{2H,t}|\|_{L^{2}(0,T;W)}^{2};
       \end{equation}
       and
       \begin{equation}\label{42n}
        \underset{n=0}{\overset{N-1}\sum}\left\|2e_{1H}^{n+1}-e_{1H}^{\overline{n+1}}-e_{1H}^{n}\right\|_{L^{2}_{f}}^{2}\leq
        2\Delta t\||e_{1H,t}|\|_{L^{2}(0,T;W)}^{2}.
       \end{equation}
       Utilizing again equations $(\ref{85b}),$ it holds
        \begin{equation}\label{43n}
        \Delta t\underset{n=0}{\overset{N-1}\sum}(\|u^{\overline{n+1}}\|_{W}^{2}+\|u^{n}\|_{W}^{2})\leq
        \nu\Delta t\underset{n=0}{\overset{N}\sum}\left\{\|\nabla u^{\overline{n}}\|_{L_{f}^{2}}^{2}+
        \|\nabla u^{n}\|_{L_{f}^{2}}^{2}\right\}\leq2\nu\||\nabla u|\|_{L^{2}(0,T;W)}^{2};
       \end{equation}
       \begin{equation}\label{44n}
        \Delta t\underset{n=0}{\overset{N-1}\sum}(\|e_{1H}^{\overline{n+1}}\|_{W}^{2}+\|e_{1H}^{n}\|_{W}^{2})
        \leq2\nu\||\nabla e_{1H}|\|_{L^{2}(0,T;W)}^{2}\text{\,\,\,and\,\,\,}
        \Delta t\underset{n=0}{\overset{N-1}\sum}\|e_{2H,tt}^{\overline{n+1}}\|_{W}^{2}\leq
        \||e_{2H,tt}|\|_{L^{2}(0,T;W)}^{2}.
       \end{equation}
       Because $t^{\overline{n+1}}\in(t^{n},t^{n+1}),$ combining equations $(\ref{84b})$ and $(\ref{85b}),$ simple calculations give
       \begin{equation*}
        \Delta t\underset{n=0}{\overset{N-1}\sum}\left\{(\|u^{\overline{n+1}}\|_{W}^{2}+\|u_{H}^{\overline{n+1}}
        \|_{W}^{2})\|e_{2H}^{\overline{n+1}}\|_{W}^{2}+(\|u^{n}\|_{W}^{2}+\|u_{H}^{n}\|_{W}^{2})
        \|e_{2H}^{n}\|_{W}^{2}\right\}\leq
       \end{equation*}
       \begin{equation*}
        \nu\Delta t\underset{n=0}{\overset{N-1}\sum}\left\{\left(\||\nabla u|\|_{L^{\infty}(0,T;W)}^{2}+
        \||\nabla u_{H}|\|_{L^{\infty}(0,T;W)}^{2}\right)(\|e_{2H}^{\overline{n+1}}\|_{W}^{2}
        +\|e_{2H}^{n}\|_{W}^{2})\right\}\leq
       \end{equation*}
        \begin{equation}\label{46n}
        2\nu\left(\||\nabla u|\|_{L^{\infty}(0,T;W)}^{2}+\||\nabla u_{H}|\|_{L^{\infty}(0,T;W)}^{2}\right)
        \||\nabla e_{2H}|\|_{L^{2}(0,T;W)}^{2};
       \end{equation}
        and
       \begin{equation*}
        \Delta t\underset{n=0}{\overset{N-1}\sum}\left(\|e_{2H}^{\overline{n+1}}\|_{W}^{4}
        +\|e_{2H}^{n}\|_{W}^{4}\right)=\nu^{2}\Delta t\underset{n=0}{\overset{N-1}\sum}\left(\|\nabla e_{2H}^{\overline{n+1}}\|_{L_{f}^{2}}^{4}+\|\nabla e_{2H}^{n}\|_{L_{f}^{2}}^{4}\right)\leq
       \end{equation*}
       \begin{equation}\label{47n}
        \nu^{2}\Delta t\underset{n=0}{\overset{N-1}\sum}\left(\|\nabla e_{2H}^{\overline{n+1}}\|_{L_{f}^{2}}^{2}
        +\|\nabla e_{2H}^{n}\|_{L_{f}^{2}}^{2}\right)\||\nabla e_{2H}|\|_{L^{\infty}(0,T;W)}^{2}\leq
        2\nu^{2}\||\nabla e_{2H}|\|_{L^{\infty}(0,T;W)}^{2}\||\nabla e_{2H}|\|_{L^{2}(0,T;W)}^{2}.
       \end{equation}
        Now, applying the bounds given by inequalities $(\ref{41n})$-$(\ref{47n}),$ multiplying both sides of estimate $(\ref{39n})$ by $\frac{1}{4}$ and absorbing all the constants into a constant $C_{1},$ this yields
        \begin{equation*}
        \left\|e_{2H}^{N}\right\|_{L^{2}_{f}}^{2}+\frac{\Delta t}{3}\underset{n=0}{\overset{N-1}\sum}\left(\|e_{2H}^{\overline{n+1}}\|_{W}^{2}
        +\|e_{2H}^{n}\|_{W}^{2}\right)\leq C_{1}\left\{\left\|e_{2H}^{0}\right\|_{L^{2}_{f}}^{2}+
        \Delta t^{3}\||e_{2H,tt}|\|_{L^{2}(0,T;W)}^{2}+\Delta t\||e_{2H,t}|\|_{L^{2}(0,T;W)}^{2}\right.
       \end{equation*}
       \begin{equation*}
        +\Delta t\underset{n=0}{\overset{N-1}\sum}\left(\left\|\frac{1}{2}u_{t}^{n}-\frac{u^{\overline{n+1}}
        -u^{n}}{\Delta t}\right\|_{L^{2}_{f}}^{2}+\left\|u_{t}^{\overline{n+1}}-\frac{u^{n+1}
        -\frac{u^{\overline{n+1}}+u^{n}}{2}}{\frac{\Delta t}{2}}\right\|_{L^{2}_{f}}^{2}\right)
        +\||\nabla e_{1H}|\|_{L^{2}(0,T;W)}^{2}
       \end{equation*}
       \begin{equation*}
        +\||e_{1H,t}|\|_{L^{2}(0,T;W)}^{2}+\left(\||\nabla u|\|_{L^{\infty}(0,T;W)}^{2}+\||\nabla u_{H}|\|_{L^{\infty}(0,T;W)}^{2}\right)\||\nabla e_{2H}|\|_{L^{2}(0,T;W)}^{2}+
       \end{equation*}
       \begin{equation}\label{48n}
        \left.\||\nabla u|\|_{L^{2}(0,T;W)}^{2}+\||\nabla e_{2H}|\|_{L^{\infty}(0,T;W)}^{2}\||\nabla e_{2H}|\|_{L^{2}(0,T;W)}^{2}\right\}.
       \end{equation}
        Expanding the Taylor series with time step $\frac{\Delta t}{2}$ for both "predicted" and corrected values, we get
        \begin{equation}\label{49nn}
         u^{\overline{n+1}}=u^{n}+\frac{\Delta t}{2}u_{t}^{n}+\frac{\Delta t^{2}}{8}u_{tt}^{n}+O(\Delta t^{3}).
        \end{equation}
        Furthermore, we also get
        \begin{equation}\label{49n}
        u^{\overline{n+1}}=u^{n}+\frac{\Delta t}{2}u_{t}^{n}+\frac{\Delta t^{2}}{8}u_{tt}^{n}+\frac{\Delta t^{3}}{48}u_{ttt}^{n},\text{\,\,\,} u^{\overline{\overline{n+1}}}=u^{\overline{n+1}}
        +\frac{\Delta t}{2}u_{t}^{\overline{n+1}}+\frac{\Delta t^{2}}{8}u_{tt}^{\overline{n+1}}
        +\frac{\Delta t^{3}}{48}u_{ttt}^{\overline{n+1}},
        \end{equation}
        and
        \begin{equation}\label{50n}
        u_{t}^{n}=u_{t}^{\overline{n+1}}-\frac{\Delta t}{2}u_{tt}^{\overline{n+1}}+\frac{\Delta t^{2}}{8}u_{ttt}^{\overline{n+1}},\text{\,\,\,}u_{tt}^{n}=u_{tt}^{\overline{n+1}}
        -\frac{\Delta t}{2}u_{ttt}^{\overline{n+1}},\text{\,\,\,}u_{ttt}^{n}=u_{ttt}^{\overline{n+1}}.
        \end{equation}
        First, neglecting the term $O(\Delta t^{3}),$ it comes from approximation $(\ref{49nn})$ that
        \begin{equation}\label{52n}
         \frac{1}{2}u_{t}^{n}-\frac{u^{\overline{n+1}}-u^{n}}{\Delta t}=-\frac{\Delta t}{8}u_{tt}^{n}.
        \end{equation}
        Plugging equations $(\ref{52n})$ and $(\ref{85b}),$ it is not hard to see that
         \begin{equation}\label{53n}
         \Delta t\underset{n=0}{\overset{N-1}\sum}\left\|u_{t}^{n}-\frac{u^{\overline{n+1}}-u^{n}}{\Delta t}
         \right\|_{L^{2}_{f}}^{2}=\frac{\Delta t^{2}}{68}\left[\Delta t\underset{n=0}{\overset{N-1}\sum}
          \|u_{tt}^{n}\|_{L^{2}_{f}}^{2}\right]\leq\frac{\Delta t^{2}}{64}\||u_{tt}|\|_{L^{2}(0,T;W)}^{2}.
         \end{equation}
        Following the MacCormack approach, putting $u^{n+1}=\frac{1}{2}(u^{\overline{n+1}}+u^{\overline{\overline{n+1}}}),$ and adding side by side equation $(\ref{49n})$, we obtain
        \begin{equation}\label{54n}
        u^{n+1}=\frac{u^{\overline{n+1}}+u^{n}}{2}+\frac{\Delta t}{4}(u_{t}^{\overline{n+1}}+u_{t}^{n})+
        \frac{\Delta t^{2}}{16}(u_{tt}^{\overline{n+1}}+u_{tt}^{n})+\frac{\Delta t^{3}}{96}(u_{ttt}^{\overline{n+1}}+u_{ttt}^{n}).
        \end{equation}
         Combining equations $(\ref{50n})$ and $(\ref{54n}),$ straightforward calculations provide
         \begin{equation}\label{55n}
        u^{n+1}=\frac{u^{\overline{n+1}}+u^{n}}{2}+\frac{\Delta t}{2}u_{t}^{\overline{n+1}}+\frac{\Delta t^{3}}{48}u_{ttt}^{n},
         \end{equation}
         which is equivalent to
        \begin{equation*}
        u_{t}^{\overline{n+1}}-\frac{u^{n+1}-\frac{u^{\overline{n+1}}+u^{n}}{2}}{\frac{\Delta t}{2}}=-\frac{\Delta t^{2}}{24}u_{ttt}^{n}.
        \end{equation*}
         Utilizing this along with equation $(\ref{85b}),$ simple computations gives
         \begin{equation}\label{57n}
         \Delta t\cdot\underset{n=0}{\overset{N-1}\sum}\left\|u_{t}^{\overline{n+1}}-\frac{u^{n+1}-\frac{u^{\overline{n+1}}
          +u^{n}}{2}}{\frac{\Delta t}{2}}\right\|_{L^{2}_{f,p}}^{2}\leq \frac{\Delta t^{4}}{24^{2}}\||u_{ttt}|\|_{L^{2}(0,T;W)}^{2}.
         \end{equation}
        From relation $(\ref{84b}),$ we have that
        \begin{equation*}
        \|e_{2H}^{0}\|_{L^{2}_{f}}^{2}\leq \||e_{2H}|\|_{L^{\infty}(0,T;W)}^{2}.
        \end{equation*}
        This fact, combined with estimates $(\ref{48n}),$ $(\ref{53n}),$ and $(\ref{57n})$ yield
       \begin{equation*}
        \left\|e_{2H}^{N}\right\|_{L^{2}_{f}}^{2}+\frac{\Delta t}{3}\underset{n=0}{\overset{N-1}\sum}\left(\|e_{2H}^{\overline{n+1}}\|_{W}^{2}
        +\|e_{2H}^{n}\|_{W}^{2}\right)\leq C_{2}\left\{\||e_{2H}|\|_{L^{\infty}(0,T;W)}^{2}+
        \Delta t^{3}\||e_{2H,tt}|\|_{L^{2}(0,T;W)}^{2}+\Delta t\||e_{2H,t}|\|_{L^{2}(0,T;W)}^{2}\right.
       \end{equation*}
       \begin{equation*}
        +\Delta t^{2}\left(\||u_{tt}|\|_{L^{\infty}(0,T;W)}^{2}+\Delta t^{2}\|
        |u_{ttt}|\|_{L^{\infty}(0,T;W)}^{2}\right)+\||\nabla e_{1H}|\|_{L^{2}(0,T;W)}^{2}
        +\||e_{1H,t}|\|_{L^{2}(0,T;W)}^{2}+
       \end{equation*}
       \begin{equation}\label{58n}
        \left.\||\nabla u|\|_{L^{2}(0,T;W)}^{2}+\left(\||\nabla u|\|_{L^{\infty}(0,T;W)}^{2}+\||\nabla u_{H}|\|_{L^{\infty}(0,T;W)}^{2}+\||\nabla e_{2H}|\|_{L^{\infty}(0,T;W)}^{2}\right)\|
        |\nabla e_{2H}|\|_{L^{2}(0,T;W)}^{2}\right\},
       \end{equation}
         where all constants are absorbed into a new constant $C_{2}.$ We recall that the exact errors terms satisfy $e_{uH}^{m}=u^{m}-u_{H}^{m}=e_{1H}^{m}+e_{2H}^{m},$ and the "predicted" one $e_{uH}^{\overline{m}}=u^{\overline{m}}-u_{H}^{\overline{m}}=e_{1H}^{\overline{m}}
         +e_{2H}^{\overline{m}}.$ When using the Triangular inequality, it comes from
        estimate $(\ref{58n})$ and the inequality $\frac{1}{2}\|a+b\|^{2}\leq\|a\|^{2}+\|b\|^{2}$ that
        \begin{equation*}
        \left\|e_{2H}^{N}\right\|_{L^{2}_{f}}^{2}+\frac{\Delta t}{3}\underset{n=0}{\overset{N-1}\sum}\left(\|e_{2H}^{\overline{n+1}}\|_{W}^{2}
        +\|e_{2H}^{n}\|_{W}^{2}\right)\leq C_{3}\left\{\||e_{2H}|\|_{L^{\infty}(0,T;W)}^{2}+
        \Delta t^{3}\||e_{2H,tt}|\|_{L^{2}(0,T;W)}^{2}+\Delta t\||e_{2H,t}|\|_{L^{2}(0,T;W)}^{2}\right.
       \end{equation*}
       \begin{equation*}
        +\Delta t^{2}\left(\||u_{tt}|\|_{L^{\infty}(0,T;W)}^{2}+\Delta t^{2}\||u_{ttt}|\|_{L^{\infty}(0,T;W)}^{2}\right)
        +\||\nabla e_{1H}|\|_{L^{2}(0,T;W)}^{2}+\||e_{1H,t}|\|_{L^{2}(0,T;W)}^{2}+
       \end{equation*}
       \begin{equation*}
        \||\nabla u|\|_{L^{2}(0,T;W)}^{2}+\left(\||\nabla u|\|_{L^{\infty}(0,T;W)}^{2}+\||\nabla u_{H}|\|_{L^{\infty}(0,T;W)}^{2}+
        \||\nabla e_{2H}|\|_{L^{\infty}(0,T;W)}^{2}\right)\||\nabla e_{2H}|\|_{L^{2}(0,T;W)}^{2}+
       \end{equation*}
       \begin{equation}\label{59n}
        \left.\|e_{1H}^{N}\|_{L^{2}_{f}}^{2}+\Delta t\underset{n=0}{\overset{N-1}\sum}
        \left(\|e_{1}^{\overline{n+1}}\|_{W}^{2} +\|e_{1}^{n}\|_{W}^{2}\right)\right\}
       \end{equation}
        According to equations $(\ref{84b})$ and $(\ref{85b}),$ we have
        \begin{equation*}
        \|e_{1H}^{N}\|_{L^{2}_{f}}^{2}\leq \||e_{1H}|\|_{L^{\infty}(0,T;W)},\text{\,\,\,and\,\,\,}
         \Delta t\underset{n=0}{\overset{N-1}\sum}\left(\|e_{1}^{\overline{n+1}}\|_{W}^{2}+\|e_{1}^{n}\|_{W}^{2}\right)
         \leq2\nu\||\nabla e_{1H}|\|_{L^{2}(0,T;W)}^{2}.
       \end{equation*}
        Substituting this into relation $(\ref{59n}),$ we obtain
        \begin{equation*}
        \left\|e_{2H}^{N}\right\|_{L^{2}_{f}}^{2}+\frac{\Delta t}{3}\underset{n=0}{\overset{N-1}\sum}\left(\|e_{2H}^{\overline{n+1}}\|_{W}^{2}
        +\|e_{2H}^{n}\|_{W}^{2}\right)\leq C_{4}\left\{\||e_{1H}|\|_{L^{\infty}(0,T;W)}^{2}+\||e_{2H}|\|_{L^{\infty}(0,T;W)}^{2}+ \right.
       \end{equation*}
       \begin{equation*}
        \Delta t^{3}\||e_{2H,tt}|\|_{L^{2}(0,T;W)}^{2}+\Delta t\||e_{2H,t}|\|_{L^{2}(0,T;W)}^{2}+\||\nabla e_{1H}|\|_{L^{2}(0,T;W)}^{2}+\||e_{1H,t}|\|_{L^{2}(0,T;W)}^{2}+
       \end{equation*}
       \begin{equation*}
        \left(\||\nabla u|\|_{L^{\infty}(0,T;W)}^{2}+\||\nabla u_{H}|\|_{L^{\infty}(0,T;W)}^{2}+
        \||\nabla e_{2H}|\|_{L^{\infty}(0,T;W)}^{2}\right)\||\nabla e_{2H}|\|_{L^{2}(0,T;W)}^{2}+
       \end{equation*}
       \begin{equation}\label{61n}
        \left.\||\nabla u|\|_{L^{2}(0,T;W)}^{2}+\Delta t^{2}\left(\||u_{tt}|\|_{L^{\infty}(0,T;W)}^{2}+\Delta t^{2}\||u_{ttt}|\|_{L^{\infty}(0,T;W)}^{2}\right)\right\},
       \end{equation}
         where we absorbed all constants into a new constant $C_{4}.$ In addition, estimate $(\ref{61n})$ holds for any $\widetilde{u}\in V_{h}.$ From relation $(\ref{26}),$ $V_{h}$ is a subspace of $W_{h},$ so it comes from relation $(\ref{83b})$ that the infimum over the space $W_{h}$ is an upper bound of the infimum over the subspace $V_{h}.$ Hence, the following estimate holds for some positive constant $C_{5}$
        \begin{equation*}
        \left\|e_{uH}^{N}\right\|_{L^{2}_{f}}^{2}+\frac{\Delta t}{3}\underset{n=0}{\overset{N-1}\sum}
        \left(\|e_{uH}^{\overline{n+1}}\|_{W}^{2}+\|e_{uH}^{n}\|_{W}^{2}\right)\leq
        C_{5}\left\{\underset{\widetilde{u}\in W_{h}}{\inf}\||u-\widetilde{u}|\|_{L^{\infty}(0,T;W)}^{2}
        +\underset{\widetilde{u}\in W_{h}}{\inf}\||\widetilde{u}-u_{H}|\|_{L^{\infty}(0,T;W)}^{2}+\right.
       \end{equation*}
       \begin{equation*}
        \Delta t^{2}\left(\||u_{tt}|\|_{L^{2}(0,T;W)}^{2}+\Delta t^{2}\||u_{ttt}|\|_{L^{2}(0,T;W)}^{2}\right)
        +\||\nabla u|\|_{L^{2}(0,T;W)}^{2}+\underset{\widetilde{u}\in W_{h}}{\inf}\||(u-\widetilde{u})_{t}|\|_{L^{2}(0,T;W)}^{2}
       \end{equation*}
       \begin{equation*}
        +\underset{\widetilde{u}\in W_{h}}{\inf}\left(\||\nabla u|\|_{L^{\infty}(0,T;W)}^{2}+\||\nabla u_{H}|\|_{L^{\infty}(0,T;W)}^{2}+
        \||\nabla(\widetilde{u}-u_{H})|\|_{L^{\infty}(0,T;W)}^{2}\right)\||\nabla(\widetilde{u}
        -u_{H})|\|_{L^{2}(0,T;W)}^{2}
       \end{equation*}
       \begin{equation*}
        \left.+\Delta t\underset{\widetilde{u}\in W_{h}}{\inf}\||(\widetilde{u}-u_{H})_{t}|\|_{L^{\infty}(0,T;W)}^{2}+
        \Delta t^{3}\underset{\widetilde{u}\in W_{h}}{\inf}\||(\widetilde{u}-u_{H})_{tt}|\|_{L^{\infty}(0,T;W)}^{2} +
        \underset{\widetilde{u}\in W_{h}}{\inf}\||\nabla (u-\widetilde{u})|\|_{L^{2}(0,T;W)}^{2}\right\}.
       \end{equation*}
       Since $\underset{\widetilde{u}\in W_{h}}{\inf}[\||\nabla u|\|_{L^{\infty}(0,T;W)}^{2}+\||\nabla u_{H}|\|_{L^{\infty}(0,T;W)}^{2}+
       \||\nabla(\widetilde{u}-u_{H})|\|_{L^{\infty}(0,T;W)}^{2}]\||\nabla (\widetilde{u}-u_{H})|\|_{L^{2}(0,T;W)}^{2}=0,$
       $\underset{\widetilde{u}\in W_{h}}{\inf}\||\widetilde{u}-u_{H}|\|_{L^{\infty}(0,T;W)}=0$ and
       $\underset{\widetilde{u}\in W_{h}}{\inf}\left[\Delta t^{3}\||(\widetilde{u}-u_{H})_{tt}|\|_{L^{\infty}(0,T;W)}^{2}+
       \Delta t\||(\widetilde{u}-u_{H})_{t}|\|_{L^{\infty}(0,T;W)}^{2}\right]=0,$ the proof of Lemma $\ref{l3}$ is completed thanks to estimates $(\ref{80b})$ and $(\ref{81b}).$
        \end{proof}

         Using Lemmas $\ref{l2}$ and $\ref{l3},$ we are ready to state and prove the main result of this paper.

        \begin{theorem}\label{t2}
         Consider the discrete variational formulation of MCRS scheme $(\ref{36})$-$(\ref{40b}).$ Let $N$ be a nonnegative integer $N\geq 1,$ and $T$ be a positive parameter. Under the regularity condition $(\ref{86b}),$ there exists a positive constant $\widehat{C}>0$ such that, for any $t^{N}\in[0;\infty),$ with $t^{N}=N\Delta t,$ it holds
       \begin{equation*}
        \|e_{uh}^{N}\|_{L^{2}_{f}}^{2}+\Delta t\underset{n=1}{\overset{N-1}\sum}\left\|e_{uh}^{n+1}
        +e_{uh}^{n-1}\right\|_{W}^{2}\leq \widehat{C}\left\{h^{2r}\left(h^{2}\||u|\|_{L^{\infty}(0,T;W^{r+1})}^{2}
        +\||u|\|_{L^{2}(0,T;W^{r+1})}^{2}+\right.\right.
       \end{equation*}
       \begin{equation*}
       \left.h^{2}\||u_{t}|\|_{L^{2}(0,T;W^{r+1})}^{2}\right)+\Delta t^{2}h^{2r+2}\||p|
       \|_{L^{2}(0,T;W^{r+1})}^{2}+\Delta t^{4}\left[\||\nabla u_{tt}|\|_{L^{2}(0,T;W)}^{2}
       +\||u_{ttt}|\|_{L^{2}(0,T;W)}^{2}\right]
       \end{equation*}
       \begin{equation*}
        +\Delta t^{2}\left(\||\nabla u|\|_{L^{\infty}(0,T;W)}^{2}+\||\nabla u_{H}|\|_{L^{\infty}(0,T;W)}^{2}\right)
        \left(\||u_{tt}|\|_{L^{2}(0,T;W^{r+1})}^{2}+\Delta t^{2}\||u_{ttt}|\|_{L^{2}(0,T;W^{r+1})}^{2}\right)+
       \end{equation*}
       \begin{equation*}
       H^{2r}\left(\||\nabla u|\|_{L^{\infty}(0,T;W)}^{2}+\||\nabla u_{H}|\|_{L^{\infty}(0,T;W)}^{2}\right)
       \left(H^{2}\||u|\|_{L^{\infty}(0,T;W^{r+1})}^{2}+\||u|\|_{L^{2}(0,T;W^{r+1})}^{2}+\right.
        \end{equation*}
        \begin{equation*}
         \left.\left.H^{2}\||u_{t}|\|_{L^{2}(0,T;W^{r+1})}^{2}\right)+\||\nabla u|\|_{L^{2}(0,T;W^{r+1})}^{2}\left(\||\nabla u|\|_{L^{\infty}(0,T;W)}^{2}+\||\nabla u_{H}|\|_{L^{\infty}(0,T;W)}^{2}\right)\right\},
        \end{equation*}
        where $r$ is a nonnegative integer.
        \end{theorem}

       \begin{proof}
       We recall that $\chi_{h}^{n}(\frac{e_{2h}^{n+1}+e_{2h}^{n-1}}{2},p^{n}-\frac{p_{h}^{n+1}+p_{h}^{n-1}}{2})=
       \frac{1}{2}\chi(e_{2h}^{n+1}+e_{2h}^{n-1},p^{n}-\frac{p_{h}^{n+1}
       +p_{h}^{n-1}}{2})-\frac{1}{2}\chi(e^{n+1}_{uh}+e^{n-1}_{uh}-(u^{n+1}+u^{n-1}-2u^{n}),p^{n}-\frac{p_{h}^{n+1}+p_{h}^{n-1}}{2}).$
       This equality allows to see that $\chi_{h}^{n}(\frac{e_{2h}^{n+1}+e_{2h}^{n-1}}{2},p^{n}-\lambda_{h}^{n})
       =-\frac{1}{2}\chi(e_{1h}^{n+1}+e_{1h}^{n-1},p^{n}-\lambda_{h}^{n}),$ where $\lambda_{h}^{n}=\frac{p_{h}^{n+1}+p_{h}^{n-1}}{2}.$ Indeed, $u^{m}\in F,$ and $p^{m},$ $p_{h}^{m}\in Q_{h}\subset Q,$ imply $\chi(u^{n+1}+u^{n-1}-2u^{n},p^{n}-\frac{p_{h}^{n+1}+p_{h}^{n-1}}{2})=0,$ where $F$ and $Q$ are given by relations $(\ref{14})$ and $(\ref{5}),$ respectively. This fact along with equation $(\ref{21n})$ given by Lemma $\ref{l2},$ provide
       \begin{equation*}
        \frac{1}{4\Delta t}\left\{\|e_{2h}^{n+1}\|_{L^{2}_{f}}^{2}-\|e_{2h}^{n-1}\|_{L^{2}_{f}}^{2}\right\}+\frac{1}{4} \left\|e_{2h}^{n+1}+e_{2h}^{n-1}\right\|_{W}^{2}=\xi^{n}_{h}(\frac{e_{2h}^{n+1}+e_{2h}^{n-1}}{2})+
        \frac{\Delta t}{4}(e_{2h}^{n+1}+e_{2h}^{n-1},e_{2h,tt}^{n})-
       \end{equation*}
       \begin{equation*}
        \frac{1}{2}\chi(e_{1h}^{n+1}+e_{1h}^{n-1},p^{n}-\lambda_{h}^{n})-b(e_{1H}^{n},u^{n},
        \frac{e_{2h}^{n+1}+e_{2h}^{n-1}}{2})-b(e_{2H}^{n},u^{n},\frac{e_{2h}^{n+1}+e_{2h}^{n-1}}{2})-
        b(u_{H}^{n},e_{1H}^{n},\frac{e_{2h}^{n+1}+e_{2h}^{n-1}}{2})
       \end{equation*}
       \begin{equation*}
        -b(u_{H}^{n},e_{2H}^{n},\frac{e_{2h}^{n+1}+e_{2h}^{n-1}}{2})-\frac{1}{4}a\left(e_{1h}^{n+1}+e_{1h}^{n-1},e_{2h}^{n+1}
        +e_{2h}^{n-1}\right)-\frac{1}{2\Delta t}\left(e_{1h}^{n}-e_{1h}^{n-1},e_{2h}^{n+1}+e_{2h}^{n-1}\right),
       \end{equation*}
       which is equivalent to
       \begin{equation*}
        \frac{1}{4\Delta t}\left\{\|e_{2h}^{n+1}\|_{L^{2}_{f}}^{2}-\|e_{2h}^{n-1}\|_{L^{2}_{f}}^{2}\right\}+\frac{1}{4} \left\|e_{2h}^{n+1}+e_{2h}^{n-1}\right\|_{W}^{2}=\xi^{n}_{h}(\frac{e_{2h}^{n+1}+e_{2h}^{n-1}}{2})+
        \frac{\Delta t}{4}(e_{2h}^{n+1}+e_{2h}^{n-1},e_{2h,tt}^{n})-
       \end{equation*}
       \begin{equation*}
        \frac{1}{2}\chi(e_{1h}^{n+1}+e_{1h}^{n-1},p^{n}-\lambda_{h}^{n})-b(e_{uH}^{n},u^{n},
        \frac{e_{2h}^{n+1}+e_{2h}^{n-1}}{2})-b(u_{H}^{n},e_{uH}^{n},\frac{e_{2h}^{n+1}+e_{2h}^{n-1}}{2})
        -\frac{1}{4}a\left(e_{1h}^{n+1}+e_{1h}^{n-1},\right.
       \end{equation*}
       \begin{equation}\label{64n}
        \left. e_{2h}^{n+1}+e_{2h}^{n-1}\right)-\frac{1}{2\Delta t}\left(e_{1h}^{n}-e_{1h}^{n-1},e_{2h}^{n+1}+e_{2h}^{n-1}\right).
       \end{equation}
       Since $\xi^{n}_{h}(\frac{e_{2h}^{n+1}+e_{2h}^{n-1}}{2})=\frac{1}{2}\xi^{n}_{h}(e_{2h}^{n+1}
       +e_{2h}^{n-1}),$ multiplying both sides of $(\ref{64n})$ by $4\Delta t$ results in
       \begin{equation*}
        \|e_{2h}^{n+1}\|_{L^{2}_{f}}^{2}-\|e_{2h}^{n-1}\|_{L^{2}_{f}}^{2}+\Delta t\left\|e_{2h}^{n+1}+e_{2h}^{n-1}\right\|_{W}^{2}=
        2\Delta t\xi^{n}_{h}(e_{2h}^{n+1}+e_{2h}^{n-1})+\Delta t^{2}(e_{2h}^{n+1}+e_{2h}^{n-1},e_{2h,tt}^{n})
       \end{equation*}
       \begin{equation*}
        -2\Delta t\chi(e_{1h}^{n+1}+e_{1h}^{n-1},p^{n}-\lambda_{h}^{n})-2\Delta t b(e_{uH}^{n},u^{n},e_{2h}^{n+1}+e_{2h}^{n-1})
        -2\Delta tb(u_{H}^{n},e_{uH}^{n},e_{2h}^{n+1}+e_{2h}^{n-1})
       \end{equation*}
       \begin{equation}\label{65n}
        -2\left(e_{1h}^{n}-e_{1h}^{n-1},e_{2h}^{n+1}+e_{2h}^{n-1}\right)
        -\Delta t a\left(e_{1h}^{n+1}+e_{1h}^{n-1},e_{2h}^{n+1}+e_{2h}^{n-1}\right).
       \end{equation}
       To bound each term of the RHS of equation $(\ref{65n})$ we should apply the Cauchy-Schwarz and Young inequalities together with Poincar\'{e} inequality $(\ref{34})$ and estimate $(\ref{17}),$
       \begin{equation*}
        2\Delta t\xi_{h}^{n}(e_{2h}^{n+1}+e_{2h}^{n-1})=-2\Delta t\left(u_{t}^{\overline{n}}-\frac{u^{n}-u^{n-1}}{\Delta t}
        ,e_{2h}^{n+1}+e_{2h}^{n-1}\right)-2\Delta t a\left(u^{n}-\frac{u^{n+1}+u^{n-1}}{2},e_{2h}^{n+1}+\right.
       \end{equation*}
       \begin{equation*}
        \left. e_{2h}^{n-1}\right)\leq 2\Delta t\left\{\left\|u_{t}^{\overline{n}}-\frac{u^{n}-u^{n-1}}{\Delta t}\right\|_{L^{2}_{f}}
        \left\|e_{2h}^{n+1}+e_{2h}^{n-1}\right\|_{L^{2}_{f}}+\left\|u^{n}-\frac{u^{n+1}+u^{n-1}}{2}\right\|_{W}
        \left\|e_{2h}^{n+1}+e_{2h}^{n-1}\right\|_{W}\right\}\leq
       \end{equation*}
       \begin{equation*}
        2\Delta t\left\{\frac{1}{\sqrt{\lambda_{1}\nu}}\left\|u_{t}^{\overline{n}}-\frac{u^{n}-u^{n-1}}{\Delta t}\right\|_{L^{2}_{f}}
        \left\|e_{2h}^{n+1}+e_{2h}^{n-1}\right\|_{W}+\left\|u^{n}-\frac{u^{n+1}+u^{n-1}}{2}\right\|_{W}
        \left\|e_{2h}^{n+1}+e_{2h}^{n-1}\right\|_{W}\right\}\leq
       \end{equation*}
       \begin{equation}\label{66n}
        \frac{\Delta t}{12}\left\|e_{2h}^{n+1}+e_{2h}^{n-1}\right\|_{W}^{2}+24\Delta t
        \left(\frac{1}{\lambda_{1}\nu}\left\|u_{t}^{\overline{n}}-\frac{u^{n}-u^{n-1}}{\Delta t}
        \right\|_{L^{2}_{f}}^{2}+\left\|u^{n}-\frac{u^{n+1}+u^{n-1}}{2}\right\|_{W}^{2}\right);
       \end{equation}
       \begin{equation*}
        \Delta t^{2}(e_{2h}^{n+1}+e_{2h}^{n-1},e_{2h,tt}^{n})\leq \Delta t^{2}
        \|e_{2h}^{n+1}+e_{2h}^{n-1}\|_{L_{f}^{2}}\|e_{2h,tt}^{n}\|_{L_{f}^{2}}\leq
       \end{equation*}
       \begin{equation}\label{67n}
        2\Delta t\left(\frac{1}{\sqrt{12}}\|e_{2h}^{n+1}+e_{2h}^{n-1}\|_{W}\frac{\sqrt{12}\Delta t}{2\sqrt{\lambda_{1}\nu}}\|e_{2h,tt}^{n}\|_{L_{f}^{2}}\right)\leq\frac{\Delta t}{12}\|e_{2h}^{n+1}+e_{2h}^{n-1}\|_{W}^{2}+\frac{3\Delta t^{3}}
        {\lambda_{1}\nu}\|e_{2h,tt}^{n}\|_{L_{f}^{2}}^{2};
       \end{equation}
       \begin{equation*}
        2\Delta t b(e_{uH}^{n},u^{n},e_{2h}^{n+1}+e_{2h}^{n-1})\leq 2\Delta t\alpha
        \|e_{uH}^{n}\|_{L_{f}^{2}}^{\frac{1}{2}}\|e_{uH}^{n}\|_{W}^{\frac{1}{2}}\|u^{n}\|_{W}\|e_{2h}^{n+1}
        +e_{2h}^{n-1}\|_{L_{f}^{2}}^{\frac{1}{2}}\|e_{2h}^{n+1}+e_{2h}^{n-1}\|_{W}^{\frac{1}{2}}+
       \end{equation*}
       \begin{equation*}
        2\Delta t\alpha\|e_{uH}^{n}\|_{W}\|u^{n}\|_{L_{f}^{2}}^{\frac{1}{2}}\|u^{n}\|_{W}^{\frac{1}{2}}
        \|e_{2h}^{n+1}+e_{2h}^{n-1}\|_{L_{f}^{2}}^{\frac{1}{2}}\|e_{2h}^{n+1}+e_{2h}^{n-1}\|_{W}^{\frac{1}{2}}\leq
       \end{equation*}
       \begin{equation*}
        \frac{2\alpha\Delta t}{(\lambda_{1}\nu)^{\frac{1}{4}}}\|e_{uH}^{n}\|_{W}^{\frac{1}{2}}\|e_{uH}^{n}\|_{L_{f}^{2}}^{\frac{1}{2}}
        \|u^{n}\|_{W}\|e_{2h}^{n+1}+e_{2h}^{n-1}\|_{W}+\frac{2\alpha\Delta t}{(\lambda_{1}\nu)^{\frac{1}{2}}}
        \|e_{uH}^{n}\|_{W}\|u^{n}\|_{W}\|e_{2h}^{n+1}+e_{2h}^{n-1}\|_{W}\leq
       \end{equation*}
       \begin{equation}\label{68n}
        \frac{48\alpha^{2}\Delta t}{\lambda_{1}\nu}\|e_{uH}^{n}\|_{W}^{2}\|u^{n}\|_{W}^{2}+
        \frac{\Delta t}{12}\|e_{2h}^{n+1}+e_{2h}^{n-1}\|_{W}^{2}.
       \end{equation}
       Likewise,
       \begin{equation}\label{69n}
        2\Delta t b(u_{H}^{n},e_{uH}^{n},e_{2h}^{n+1}+e_{2h}^{n-1})\leq \frac{48\alpha^{2}\Delta t}{\lambda_{1}\nu}
        \|e_{uH}^{n}\|_{W}^{2}\|u_{H}^{n}\|_{W}^{2}+\frac{\Delta t}{12}\|e_{2h}^{n+1}+e_{2h}^{n-1}\|_{W}^{2};
       \end{equation}
       \begin{equation*}
        \Delta t a(e_{1h}^{n+1}+e_{1h}^{n-1},e_{2h}^{n+1}+e_{2h}^{n-1})\leq \Delta t
        \|e_{1h}^{n+1}+e_{1h}^{n-1}\|_{W}
        \|e_{2h}^{n+1}+e_{2h}^{n-1}\|_{W}\leq 3\Delta t\left\|e_{1h}^{n+1}+e_{1h}^{n-1}\right\|_{W}^{2}+
       \end{equation*}
       \begin{equation}\label{70n}
        \frac{\Delta t}{12}\left\|e_{2h}^{n+1}+e_{2h}^{n-1}\right\|_{W}^{2};
       \end{equation}
       \begin{equation*}
        2(e_{1h}^{n}-e_{1h}^{n-1},e_{2h}^{n+1}+e_{2h}^{n-1})\leq 2(\lambda_{1}\nu)^{-\frac{1}{2}}\|e_{1h}^{n}-e_{1h}^{n-1}\|_{L_{f}^{2}}
        \|e_{2h}^{n+1}+e_{2h}^{n-1}\|_{W}\leq \frac{12}{\lambda_{1}\nu\Delta t}\left\|e_{1h}^{n}-e_{1h}^{n-1}\right\|_{L_{f}^{2}}^{2}+
       \end{equation*}
       \begin{equation}\label{71n}
        \frac{\Delta t}{12}\left\|e_{2h}^{n+1}+e_{2h}^{n-1}\right\|_{W}^{2};
       \end{equation}
       \begin{equation}\label{71nn}
        -2\Delta t\chi(e_{1h}^{n+1}+e_{1h}^{n-1},p^{n}-\lambda_{h}^{n})\leq 2\Delta t\|\nabla\cdot(e_{1h}^{n+1}+e_{1h}^{n-1})\|_{L_{f}^{2}}
        \|p^{n}-\lambda_{h}^{n}\|_{L_{f}^{2}}\leq 2\Delta t\|e_{1h}^{n+1}+e_{1h}^{n-1}\|_{W}^{2}
        +\frac{2\Delta t}{\nu}\|p^{n}-\lambda_{h}^{n}\|_{L_{f}^{2}}^{2},
       \end{equation}
       where the last inequality comes from the Cauchy-Schwarz inequality and the estimate $\|\nabla\cdot(e_{1h}^{n+1}
       +e_{1h}^{n-1})\|_{L_{f}^{2}}\leq \sqrt{2}\|\nabla(e_{1h}^{n+1}+e_{1h}^{n-1})\|_{L_{f}^{2}}.$
       Replacing the bounds given by estimates $(\ref{66n})$-$(\ref{71nn})$ into relation $(\ref{65n})$ and grouping the remaining terms, the inequality becomes
       \begin{equation*}
        \|e_{2h}^{n+1}\|_{L^{2}_{f}}^{2}-\|e_{2h}^{n-1}\|_{L^{2}_{f}}^{2}+\frac{\Delta t}{2}\left\|e_{2h}^{n+1}
        +e_{2h}^{n-1}\right\|_{W}^{2}\leq 24\Delta t\left(\frac{1}{\lambda_{1}\nu}\left\|u_{t}^{\overline{n}}-\frac{u^{n}-u^{n-1}}{\Delta t}
        \right\|_{L^{2}_{f}}^{2}+\left\|u^{n}-\frac{u^{n+1}+u^{n-1}}{2}\right\|_{W}^{2}\right)
       \end{equation*}
       \begin{equation*}
        +\frac{3\Delta t^{3}}{\lambda_{1}\nu}\|e_{2h,tt}^{n}\|_{L_{f}^{2}}^{2}+\frac{48\alpha^{2}\Delta t}
        {\lambda_{1}\nu}\|e_{uH}^{n}\|_{W}^{2}\left\{\|u^{n}\|_{W}^{2}+\|u_{H}^{n}\|_{W}^{2}\right\}+ 5\Delta t\|e_{1h}^{n+1}+e_{1h}^{n-1}\|_{W}^{2}+\frac{12}{\lambda_{1}\nu\Delta t}\|e_{1h}^{n}-e_{1h}^{n-1}\|_{L_{f}^{2}}^{2}
       \end{equation*}
       \begin{equation}\label{72n}
        +\frac{2\Delta t}{\nu}\|p^{n}-\lambda_{h}^{n}\|_{L_{f}^{2}}^{2}.
       \end{equation}
       Summing relation $(\ref{72n})$ up from $n=1,...,N-1,$ this gives
       \begin{equation*}
        \|e_{2h}^{N}\|_{L^{2}_{f}}^{2}+\frac{\Delta t}{2}\underset{n=1}{\overset{N-1}\sum}\left\|e_{2h}^{n+1}
        +e_{2h}^{n-1}\right\|_{W}^{2}\leq 24\Delta t\underset{n=1}{\overset{N-1}\sum}\left(\frac{1}{\lambda_{1}\nu}\left\|u_{t}^{\overline{n}}-
        \frac{u^{n}-u^{n-1}}{\Delta t}\right\|_{L^{2}_{f}}^{2}+\left\|u^{n}-\frac{u^{n+1}+u^{n-1}}{2}\right\|_{W}^{2}\right)
       \end{equation*}
       \begin{equation*}
        +\frac{3\Delta t^{3}}{\lambda_{1}\nu}\underset{n=1}{\overset{N-1}\sum}\|e_{2h,tt}^{n}\|_{L_{f}^{2}}^{2}
        +\frac{48\alpha^{2}\Delta t}
        {\lambda_{1}\nu}\underset{n=1}{\overset{N-1}\sum}\|e_{uH}^{n}\|_{W}^{2}\left\{\|u^{n}\|_{W}^{2}
        +\|u_{H}^{n}\|_{W}^{2}\right\}+ 5\Delta t\underset{n=1}{\overset{N-1}\sum}\|e_{1h}^{n+1}+e_{1h}^{n-1}\|_{W}^{2}
       \end{equation*}
       \begin{equation}\label{73n}
       +\frac{12}{\lambda_{1}\nu\Delta t}\underset{n=1}{\overset{N-1}\sum}\|e_{1h}^{n}-e_{1h}^{n-1}\|_{L_{f}^{2}}^{2}
        +\frac{2\Delta t}{\nu}\underset{n=1}{\overset{N-1}\sum}\|p^{n}-\lambda_{h}^{n}\|_{L_{f}^{2}}^{2}+
        \|e_{2h}^{0}\|_{L^{2}_{f}}^{2}.
       \end{equation}
       Using the Taylor series with time step $\Delta t,$ it is easy to see that
        \begin{equation}\label{74n}
        \left\|u^{n}-\frac{u^{n+1}+u^{n-1}}{2}\right\|_{W}^{2}\leq \frac{\Delta t^{4}}{4}\|u_{tt}^{n}\|_{W}^{2}.
        \end{equation}
        Furthermore, expanding the Taylor series with time step $\frac{\Delta t}{2},$ to get
        \begin{equation}\label{75n}
        u^{n}=u^{\overline{n}}+\frac{\Delta t}{2}u_{t}^{\overline{n}}+\frac{\Delta t^{2}}{8}u_{tt}^{\overline{n}}+\frac{\Delta t^{3}}{48}u_{ttt}^{\overline{n}}+O(\Delta t^{4});
        \text{\,\,\,\,\,\,}u^{n-1}=u^{\overline{n}}-\frac{\Delta t}{2}u_{t}^{\overline{n}}+\frac{\Delta t^{2}}{8}u_{tt}^{\overline{n}}-
        \frac{\Delta t^{3}}{48}u_{ttt}^{\overline{n}}+O(\Delta t^{4}).
        \end{equation}
        Tracking the infinitesimal term $O(\Delta t^{4})$ in $(\ref{75n})$, straightforward computations yield
        \begin{equation}\label{77n}
        \left\|u_{t}^{\overline{n}}-\frac{u^{n}-u^{n-1}}{\Delta t}\right\|_{L_{f}^{2}}^{2}\leq \frac{\Delta t^{4}}{24^{2}}\|u_{ttt}^{\overline{n}}\|_{L_{f}^{2}}^{2}.
        \end{equation}
        Similar to $(\ref{40n}),$ it is easy to see that
         \begin{equation}\label{78n}
         \frac{1}{\Delta t}\underset{n=1}{\overset{N-1}\sum}\left\|e_{1h}^{n}-e_{1h}^{n-1}\right\|_{L_{f}^{2}}^{2}\leq
         \||e_{1h,t}|\|_{L^{2}(0,T;W)}^{2}.
         \end{equation}
         Furthermore, utilizing the inequality $\|u+v\|^{2}\leq 2(\|u\|^{2}+\|v\|^{2}),$ it comes from relation $(\ref{44n})$
         that
         \begin{equation}\label{79n}
         \Delta t\underset{n=1}{\overset{N-1}\sum}\left\|e_{1h}^{n+1}+e_{1h}^{n-1}\right\|_{W}^{2}\leq
         4\nu\||\nabla e_{1h}|\|_{L^{2}(0,T;W)}^{2}.
         \end{equation}
         In way similar,
        \begin{equation}\label{80n}
        \Delta t^{3}\underset{n=1}{\overset{N-1}\sum}\|e_{2h,tt}^{n}\|_{L_{f}^{2}}^{2}\leq \Delta t^{2}\||e_{2h,tt}|\|_{L^{2}(0,T;W)}^{2}
          \text{\,\,\,\,\,and\,\,\,\,\,}\Delta t\underset{n=1}{\overset{N-1}\sum}\|p^{n}-\lambda_{h}^{n}\|_{L_{f}^{2}}^{2}\leq
         \Delta t^{2}\||p-\lambda_{h}|\|_{L^{2}(0,T;W)}^{2}.
        \end{equation}
        From relation $(\ref{84b}),$ we have that
        \begin{equation}\label{50nn}
        \|e_{2h}^{0}\|_{L^{2}_{f}}^{2}\leq\||e_{2h}|\|_{L^{\infty}(0,T;W)}^{2}.
        \end{equation}
        This fact, combined with estimates $(\ref{73n})$ and $(\ref{74n})$ together with $(\ref{77n})$-$(\ref{50nn})$ give
       \begin{equation*}
        \|e_{2h}^{N}\|_{L^{2}_{f}}^{2}+\frac{\Delta t}{2}\underset{n=1}{\overset{N-1}\sum}\left\|e_{2h}^{n+1}
        +e_{2h}^{n-1}\right\|_{W}^{2}\leq \widetilde{C}_{1}\left\{\||e_{2h}|\|_{L^{\infty}(0,T;W)}^{2}
        +\Delta t^{5}\underset{n=1}{\overset{N-1}\sum}\left(\|u_{tt}^{n}\|_{W}^{2}+\|u_{ttt}^{\overline{n}}
        \|_{L_{f}^{2}}^{2}\right)+\right.
       \end{equation*}
       \begin{equation*}
        \||\nabla e_{1h}|\|_{L^{2}(0,T;W)}^{2}+\||e_{1h,t}|\|_{L^{2}(0,T;W)}^{2}+\Delta t^{2}
        \||e_{2h,tt}|\|_{L^{2}(0,T;W)}^{2}+\Delta t^{2}\||p-\lambda_{h}|\|_{L^{2}(0,T;W)}^{2}+
       \end{equation*}
       \begin{equation}\label{81n}
        \left.\Delta t\underset{n=1}{\overset{N-1}\sum}\|e_{uH}^{n}\|_{W}^{2}\left\{\|u^{n}\|_{W}^{2}
        +\|u_{H}^{n}\|_{W}^{2}\right\}\right\},
       \end{equation}
       where we absorbed all constants into a constant $\widetilde{C}_{1}.$ Plugging $(\ref{43})$ and $(\ref{84b})$-$(\ref{85b}),$ simple calculations provide
       \begin{equation*}
        \Delta t\underset{n=1}{\overset{N-1}\sum}\left\{\|u^{n}\|_{W}^{2}+\|u_{H}^{n}\|_{W}^{2}\right\}
        \|e_{uH}^{n}\|_{W}^{2}\leq \nu\Delta t\left(\||\nabla u|\|_{L^{\infty}(0,T;W)}^{2}+\||\nabla u_{H}|\|_{L^{\infty}(0,T;W)}^{2}\right)\underset{n=1}{\overset{N-1}\sum}\|e_{uH}^{n}\|_{W}^{2}\leq
       \end{equation*}
       \begin{equation*}
       \widetilde{C}_{2}\left(\||\nabla u|\|_{L^{\infty}(0,T;W)}^{2}+\||\nabla u_{H}|\|_{L^{\infty}(0,T;W)}^{2}\right)\left\{H^{2r}\left(H^{2}\||u|\|_{L^{\infty}(0,T;W^{r+1})}^{2}
       +\||u|\|_{L^{2}(0,T;W^{r+1})}^{2}+\right.\right.
        \end{equation*}
        \begin{equation}\label{82n}
        \left.\left.H^{2}\||u_{t}|\|_{L^{2}(0,T;W^{r+1})}^{2}\right)+\Delta t^{2}
        \left(\||u_{tt}|\|_{L^{2}(0,T;W^{r+1})}^{2}+\Delta t^{2}\||u_{ttt}|\|_{L^{2}(0,T;W^{r+1})}^{2}\right)
        +\||\nabla u|\|_{L^{2}(0,T;W^{r+1})}^{2}\right\},
        \end{equation}
        for every $r\in\mathbb{N},$ where all the constants are absorbed into a new constant $\widetilde{C}_{2}.$ Furthermore
        \begin{equation}\label{83n}
        \Delta t^{5}\underset{n=1}{\overset{N-1}\sum}\left(\|u_{tt}^{n}\|_{W}^{2}+\|u_{ttt}^{\overline{n}}
        \|_{L_{f}^{2}}^{2}\right)\leq \Delta t^{4}\left(\nu\||\nabla u_{tt}|\|_{L^{2}(0,T;W)}^{2}
        +\||u_{ttt}|\|_{L^{2}(0,T;W)}^{2}\right).
        \end{equation}
         Substituting inequalities $(\ref{83n})$ and $(\ref{82n})$ into $(\ref{81n})$ and absorbing all the constants into a constant $\widetilde{C}_{3}$ results in
       \begin{equation*}
        \|e_{2h}^{N}\|_{L^{2}_{f}}^{2}+\frac{\Delta t}{2}\underset{n=1}{\overset{N-1}\sum}\left\|e_{2h}^{n+1}
        +e_{2h}^{n-1}\right\|_{W}^{2}\leq \widetilde{C}_{3}\left\{\||e_{2h}|\|_{L^{\infty}(0,T;W)}^{2}+\||\nabla e_{1h}|\|_{L^{2}(0,T;W)}^{2}+\||e_{1h,t}|\|_{L^{2}(0,T;W)}^{2}+\right.
       \end{equation*}
       \begin{equation*}
       \Delta t^{2}\||e_{2h,tt}|\|_{L^{2}(0,T;W)}^{2}+\Delta t^{2}\||p-\lambda_{h}|\|_{L^{2}(0,T;W)}^{2}+
       \Delta t^{4}\left[\||\nabla u_{tt}|\|_{L^{2}(0,T;W)}^{2}+\||u_{ttt}|\|_{L^{2}(0,T;W)}^{2}\right]
       \end{equation*}
       \begin{equation*}
        +\Delta t^{2}\left(\||\nabla u|\|_{L^{\infty}(0,T;W)}^{2}+\||\nabla u_{H}|\|_{L^{\infty}
        (0,T;W)}^{2}\right)\left(\||u_{tt}|\|_{L^{2}(0,T;W^{r+1})}^{2}+\Delta t^{2}\||u_{ttt}|\|_{L^{2}
        (0,T;W^{r+1})}^{2}\right)+
       \end{equation*}
       \begin{equation*}
       \left(\||\nabla u|\|_{L^{\infty}(0,T;W)}^{2}+\||\nabla u_{H}|\|_{L^{\infty}(0,T;W)}^{2}\right)
       \left[H^{2r}\left(H^{2}\||u|\|_{L^{\infty}(0,T;W^{r+1})}^{2}+\||u|\|_{L^{2}(0,T;W^{r+1})}^{2}+\right.\right.
        \end{equation*}
        \begin{equation}\label{84n}
        \left.\left.\left.H^{2}\||u_{t}|\|_{L^{2}(0,T;W^{r+1})}^{2}\right)+\||\nabla u|\|_{L^{2}(0,T;W^{r+1})}^{2}\right]\right\}.
        \end{equation}
       Let us recall that the exact errors terms check $e_{uh}^{m}=u^{m}-u_{h}^{m}=e_{1h}^{m}+e_{2h}^{m}.$ Utilizing the triangular inequality, it comes from estimate $(\ref{84n}),$ $(\ref{84b}),$ $(\ref{85b})$ and estimate $\frac{1}{2}\|u+v\|^{2}\leq\|u\|^{2}+\|v\|^{2} $ that
       \begin{equation*}
        \|e_{uh}^{N}\|_{L^{2}_{f}}^{2}+\Delta t\underset{n=1}{\overset{N-1}\sum}\left\|e_{uh}^{n+1}
        +e_{uh}^{n-1}\right\|_{W}^{2}\leq \widetilde{C}_{4}\left\{\||e_{2h}|\|_{L^{\infty}(0,T;W)}^{2}+\||\nabla e_{1h}|\|_{L^{2}(0,T;W)}^{2}+\||e_{1h,t}|\|_{L^{2}(0,T;W)}^{2}+\right.
       \end{equation*}
       \begin{equation*}
       \Delta t^{2}\||e_{2h,tt}|\|_{L^{2}(0,T;W)}^{2}+\Delta t^{2}\||p-\lambda_{h}|\|_{L^{2}(0,T;W)}^{2}+
       \Delta t^{4}\left[\||\nabla u_{tt}|\|_{L^{2}(0,T;W)}^{2}+\||u_{ttt}|\|_{L^{2}(0,T;W)}^{2}\right]+
       \end{equation*}
       \begin{equation*}
        \Delta t^{2}\left(\||\nabla u|\|_{L^{\infty}(0,T;W)}^{2}+\||\nabla u_{H}|\|_{L^{\infty}(0,T;W)}^{2}
        \right)\left(\||u_{tt}|\|_{L^{2}(0,T;W^{r+1})}^{2}+\Delta t^{2}\||u_{ttt}|\|_{L^{2}(0,T;W^{r+1})}^{2}
        \right)+
       \end{equation*}
       \begin{equation*}
       \left(\||\nabla u|\|_{L^{\infty}(0,T;W)}^{2}+\||\nabla u_{H}|\|_{L^{\infty}(0,T;W)}^{2}\right)
       \left[H^{2r}\left(H^{2}\||u|\|_{L^{\infty}(0,T;W^{r+1})}^{2}+\||u|\|_{L^{2}(0,T;W^{r+1})}^{2}
       +\right.\right.
        \end{equation*}
        \begin{equation}\label{85n}
        \left.\left.\left.H^{2}\||u_{t}|\|_{L^{2}(0,T;W^{r+1})}^{2}\right)+\||\nabla u|\|_{L^{2}(0,T;W^{r+1})}^{2}\right]+\||e_{1h}|\|_{L^{\infty}(0,T;W)}^{2}\right\},
        \end{equation}
        where we absorbed all the constants into a new constant $\widetilde{C}_{4}.$ In addition, estimate $(\ref{85n})$ holds for any $\widetilde{u}\in V_{h}$ and $\lambda_{h}\in Q_{h}.$ From equation $(\ref{26}),$ $V_{h}$ is a subspace of $W_{h},$ so it comes from inequality $(\ref{83b})$ that the infimum over the space $W_{h}$ is an upper bound of the infimum over the subspace $V_{h}.$ Hence, the following estimate holds for some positive constant $\widetilde{C}_{5}$
       \begin{equation*}
        \|e_{uh}^{N}\|_{L^{2}_{f}}^{2}+\Delta t\underset{n=1}{\overset{N-1}\sum}\left\|e_{uh}^{n+1}
        +e_{uh}^{n-1}\right\|_{W}^{2}\leq \widetilde{C}_{5}\left\{\underset{\widetilde{u}\in W_{h}}\inf\||\widetilde{u}-u_{h}|\|_{L^{\infty}(0,T;W)}^{2}+\underset{\widetilde{u}\in W_{h}}\inf\||\nabla(u-\widetilde{u})|\|_{L^{2}(0,T;W)}^{2}+\right.
       \end{equation*}
       \begin{equation*}
       \Delta t^{2}(\underset{\widetilde{u}\in W_{h}}\inf\||(\widetilde{u}-u_{h})_{tt}|\|_{L^{2}(0,T;W)}^{2}+
       \underset{\widetilde{\lambda_{h}}\in Q_{h}}\inf\||p-\lambda_{h}|\|_{L^{2}(0,T;W)}^{2})+
       \Delta t^{4}\left[\||\nabla u_{tt}|\|_{L^{2}(0,T;W)}^{2}+\||u_{ttt}|\|_{L^{2}(0,T;W)}^{2}\right]
       \end{equation*}
       \begin{equation*}
        +\Delta t^{2}\left(\||\nabla u|\|_{L^{\infty}(0,T;W)}^{2}+\||\nabla u_{H}|\|_{L^{\infty}(0,T;W)}^{2}\right)
        \left(\||u_{tt}|\|_{L^{2}(0,T;W^{r+1})}^{2}+\Delta t^{2}\||u_{ttt}|\|_{L^{2}(0,T;W^{r+1})}^{2}\right)+
       \end{equation*}
       \begin{equation*}
       \left(\||\nabla u|\|_{L^{\infty}(0,T;W)}^{2}+\||\nabla u_{H}|\|_{L^{\infty}(0,T;W)}^{2}\right)\left[H^{2r}\left(H^{2}\||u|\|_{L^{\infty}
       (0,T;W^{r+1})}^{2}+\||u|\|_{L^{2}(0,T;W^{r+1})}^{2}+\right.\right.
        \end{equation*}
        \begin{equation}\label{86n}
        \left.\left.\left.H^{2}\||u_{t}|\|_{L^{2}(0,T;W^{r+1})}^{2}\right)+\||\nabla u|\|_{L^{2}(0,T;W^{r+1})}^{2}\right]+\underset{\widetilde{u}\in W_{h}}\inf\||u-\widetilde{u}|\|_{L^{\infty}(0,T;W)}^{2}+\underset{\widetilde{u}\in W_{h}}\inf\||(u-\widetilde{u})_{t}|\|_{L^{2}(0,T;W)}^{2}\right\},
        \end{equation}
        Applying estimates $(\ref{80b})$-$(\ref{82b})$ together with equalities $\underset{\widetilde{u}\in W_{h}}\inf\||\widetilde{u}-u_{h}|\|_{L^{\infty}(0,T;W)}^{2}=\underset{\widetilde{u}\in W_{h}}
        \inf\||(\widetilde{u}-u_{h})_{tt}|\|_{L^{\infty}(0,T;W)}^{2}\\=0,$ relation $(\ref{86n})$ becomes
       \begin{equation*}
        \|e_{uh}^{N}\|_{L^{2}_{f}}^{2}+\Delta t\underset{n=1}{\overset{N-1}\sum}\left\|e_{uh}^{n+1}
        +e_{uh}^{n-1}\right\|_{W}^{2}\leq\widehat{C}\left\{h^{2r}\left(h^{2}\||u|\|_{L^{\infty}(0,T;W^{r+1})}^{2}
        +\||u|\|_{L^{2}(0,T;W^{r+1})}^{2}+\right.\right.
       \end{equation*}
       \begin{equation*}
       \left.h^{2}\||u_{t}|\|_{L^{2}(0,T;W^{r+1})}^{2}\right)+\Delta t^{2}h^{2r+2}\||p|\|_{L^{2}(0,T;W)}^{2}
       +\Delta t^{4}\left[\||\nabla u_{tt}|\|_{L^{2}(0,T;W)}^{2}+\||u_{ttt}|\|_{L^{2}(0,T;W)}^{2}\right]
       \end{equation*}
       \begin{equation*}
        +\Delta t^{2}\left(\||\nabla u|\|_{L^{\infty}(0,T;W)}^{2}+\||\nabla u_{H}|\|_{L^{\infty}(0,T;W)}^{2}\right)\left(\||u_{tt}|\|_{L^{2}(0,T;W^{r+1})}^{2}+\Delta t^{2}
        \||u_{ttt}|\|_{L^{2}(0,T;W^{r+1})}^{2}\right)+
       \end{equation*}
       \begin{equation*}
       H^{2r}\left(\||\nabla u|\|_{L^{\infty}(0,T;W)}^{2}+\||\nabla u_{H}|\|_{L^{\infty}(0,T;W)}^{2}\right)
       \left(H^{2}\||u|\|_{L^{\infty}(0,T;W^{r+1})}^{2}+\||u|\|_{L^{2}(0,T;W^{r+1})}^{2}+\right.
        \end{equation*}
        \begin{equation*}
         \left.\left.H^{2}\||u_{t}|\|_{L^{2}(0,T;W^{r+1})}^{2}\right)+\||\nabla u|\|_{L^{2}
         (0,T;W^{r+1})}^{2}\left(\||\nabla u|\|_{L^{\infty}(0,T;W)}^{2}+\||\nabla u_{H}|\|_{L^{\infty}(0,T;W)}^{2}\right)\right\},
        \end{equation*}
        where we absorbed all the constants into a constant $\widehat{C}.$ This completes the proof of Theorem $\ref{t2}.$
         \end{proof}

         \section{Numerical experiments}\label{IV}
          This section considers a wide set of numerical experiments in two-dimensional case. We stress that in this situation we obtain satisfactory results, so our algorithm performances are not worse for multidimensional problems. Specifically, we consider a simple example which is a nonphysical example with the pressure $p=0,$ together with the other ones that is a practical example, case $p\neq0.$ Using the exact solutions introduced in \cite{hmr,br}, we confirm the predicted convergence rate from the theory. Furthermore, we look at errors over long time intervals to see the convergence rate of our proposed method for the parameters $\nu$ smaller than covered by the theory. To demonstrate this convergence, we list in Table $\ref{tab:1}$ the errors between the computed solution and the exact one with varying spacing $h$ and time step $\Delta t.$ We look at the error estimates of the method for the parameters $\nu\in\{1,10^{-1}\}$ and $T=2^{2}.$ Finally, the numerical evidences are performed using the Matlab building function $R2009a.$\\

       $\bullet$ \textbf{Test $1.$} we consider an artificial model on the two-dimensional unit square $\Omega_{f}=(0,1)^{2}$ and final time $T=2^{2}.$ We set $\nu=10^{-1}$ and we choose the force $f$ in such a way that the exact solution $(u=(u_{1},u_{2}),p)^{'}$ is given by
       \begin{eqnarray*}
       % \nonumber to remove numbering (before each equation)
         u_{1}(t,x,y) &=& \sin^{2}(\pi x)\sin(\pi y)\cos(\pi y)\sin(t),\\
         u_{2}(t,x,y) &=& -\sin(\pi x)\cos(\pi x)\sin^{2}(\pi y)\sin(t),\\
         p(t,x,y) &=& \sin(\pi x)\cos(\pi x)\sin(\pi y)\cos(\pi y)\sin(t).
       \end{eqnarray*}
       The initial and boundary conditions are set to
       \begin{equation*}
        u=0,\text{\,\,in\,\,}\{0\}\times\Omega_{f},\text{\,\,\,\,}u=0,\text{\,\,on\,\,}(0,1)\times\partial\Omega_{f}.
       \end{equation*}
       The finite element discretization uses a quadrilateral mesh with $Q_{2}/(Q_{1}+Q_{0})$ element. In the $Q_{2}/(Q_{1}+Q_{0})$ element, piecewise bilinear functions on quadrilaterals are used to approximate the velocity $u.$ To analyze the convergence rate of our numerical scheme, we take the mesh size and time step: $h,\text{\,}\Delta t\in\{\frac{1}{2^{2}},\frac{1}{2^{3}},\frac{1}{2^{4}}, \frac{1}{2^{5}},\frac{1}{2^{6}},\frac{1}{2^{7}},\frac{1}{2^{8}},
       \frac{1}{2^{9}},\frac{1}{2^{10}}\},$ by a mid-point refinement. We set $\Delta t=h$ and we compute the error estimates: $E^{N}(u),$ $E^{N}\left(\nabla u\right)$ and $E^{N}(p)$ related to the rapid solver method to see that the algorithm is of second order accuracy in both time and space. In addition, we plot the errors versus $n.$ From this analysis, MCRS is both efficient and effective than the two-level finite element Galerkin approach. In fact the two-level finite element Galerkin scheme is of first order convergent (see \cite{hmr}, Theorem $6.6,$ p. $19$). Furthermore, when $h$ varies in the given range, we observe from Tables that the approximation errors $O(\Delta t^{\beta})+O(h^{\theta})$ are dominated by the $h$-terms $O(h^{\theta})$ (or $\Delta t$-terms $O(\Delta t^{\beta})$). So, the ratio $r^{m}_{(\cdot)},$ where $(\cdot)$ designates $u,$ $\nabla u,$ $p$ of the approximation errors on two adjacent mesh levels $\Omega_{2h}$ and $\Omega_{h}$ is approximately $(2h)^{\theta}/h^{\theta}=2^{\theta},$ where $m$ refers to the $W^{m}_{2}(\Omega_{f})$-error norm. Hence, we can simply use $r^{m}_{(\cdot)}$ to estimate the corresponding convergence rate with respect to $h=\Delta t.$ Define the norms for the errors, $E^{N}(u),$ $E^{N}\left(\nabla u\right),$ $E^{N}(p),$ for each $N\in\mathrm{N},$ as follows

         \begin{equation*}
         % \nonumber to remove numbering (before each equation)
            E^{N}(u)=\||u-u_{h}|\|_{L^{2}(0,T;W)}=\left[\Delta t\underset{n=0}{\overset{N}\sum}\|u^{n}-u_{h}^{n}\|_{L_{f}^{2}}^{2}\right]^{\frac{1}{2}};
            \text{\,\,\,}E^{N}(p)=\||p-p_{h}\||_{L^{2}\left(0,T;W_{2}^{0}\right)}=\left[\Delta t\underset{n=0}{\overset{N}\sum}
            \|p^{n}-p_{h}^{n}\|_{L_{f}^{2}}^{2}\right]^{\frac{1}{2}}.
         \end{equation*}
         \begin{equation*}
            E^{N}\left(\nabla u\right)=\||\nabla(u-u_{h})\||_{L^{2}(0,T;W)}=\left[\nu\Delta t\underset{n=0}{\overset{N}\sum}
            \|\nabla(u^{n}-u_{h}^{n})\|_{L_{f}^{2}}^{2}\right]^{\frac{1}{2}};\\
         \end{equation*}

          Surprisingly, it comes from the Tests (see Table $\ref{tab:1}$ and Figure $\ref{figure}$) that the rapid solver method is second order convergent both for the Stokes velocity ($u$) and the Stokes pressure ($p$).

          \begin{table}
             % table caption is above the table
           \caption{Analysis of convergence rates $O(h^{\theta}+\Delta t^{\beta})$ for MCRS by $r^{1}_{(\cdot)},$ with varying spacing $h$ and time step $\Delta t$ (with $\Delta t=h$).}
            \label{tab:1}
            $$\begin{tabular}{c c}
              % after \\: \hline or \cline{col1-col2} \cline{col3-col4} ...
            \begin{tabular}{|c|c|c|c|c|c|c|}
            \noalign{\smallskip}\hline\noalign{\smallskip}
            % after \\: \hline or \cline{col1-col2} \cline{col3-col4} ...
            $\Delta t$ & $E^{N}(u)$ & $r^{1}_{u}$ & $E^{N}(p)$ & $r^{1}_{p}$ &
            $E^{N}(\nabla u)$ & $r^{1}_{\nabla u}$ \\
            \hline
            $2^{-2}$ & 0.7512 &        & 0.3756 &        & 3.6558 &  \\
            \hline
            $2^{-3}$ & 0.3715 & 2.0221 & 0.1857 & 2.0226 & 2.5569 & 1.4298 \\
            \hline
            $2^{-4}$ & 0.1848 & 2.0103 & 0.0924 & 2.0097 & 1.7984 & 1.4218\\
            \hline
            $2^{-5}$ & 0.0921 & 2.0065 & 0.0461 & 2.0043 & 1.2684 & 1.4178\\
            \hline
            $2^{-6}$ & 0.0460 & 2.0022 & 0.0230 & 2.0043 & 0.8957  & 1.4161 \\
            \hline
            $2^{-7}$ & 0.0230 & 2.0000 & 0.0115 & 2.0000 & 0.6330 & 1.4150\\
            \hline
            $2^{-8}$ & 0.0115 & 2.0000 & 0.0057 & 2.0175 & 0.4474 & 1.4148\\
            \hline
            \hline
            $2^{-9}$ & 0.0057 & 2.0175 & 0.0029 & 1.9655 & 0.3163 & 1.4145\\
            \hline
            \hline
            $2^{-10}$& 0.0029 & 1.9655& 0.0014 & 2.0714 & 0.2237 & 1.4139\\
            \hline\noalign{\smallskip}
          \end{tabular} & \begin{tabular}{|c|c|c|c|c|}
            \noalign{\smallskip}\hline\noalign{\smallskip}
            % after \\: \hline or \cline{col1-col2} \cline{col3-col4} ...
            $\Delta t$ & $E^{N}(u)$ & $r^{1}_{u}$ & $E^{N}(\nabla u)$ & $r^{1}_{\nabla u}$  \\
            \hline
            $2^{-2}$ & 1.6129 &        & 0.8321 &        \\
            \hline
            $2^{-3}$ & 0.7921 & 2.0362 & 0.4087 & 2.0360 \\
            \hline
            $2^{-4}$ & 0.3923 & 2.0191 & 0.2024 & 2.0193\\
            \hline
            $2^{-5}$ & 0.1952 & 2.0097 & 0.1007 & 2.0099\\
            \hline
            $2^{-6}$ & 0.0974 & 2.0041 & 0.0502 & 2.0060 \\
            \hline
            $2^{-7}$ & 0.0486 & 2.0041 & 0.0251 & 2.0000 \\
            \hline
            $2^{-8}$ & 0.0243 & 2.0000 & 0.0125 & 2.0080 \\
            \hline
            $2^{-9}$ & 0.0121 & 2.0083 & 0.0063 & 1.9841\\
            \hline
            $2^{-10}$& 0.0061 & 1.9836 & 0.0031 & 2.0323 \\
            \hline\noalign{\smallskip}
          \end{tabular} \\
          \text{\,}\\
                     Test 1: $\nu=10^{-1}$ and $\lambda_{1}=1$ & Test 2: $\nu=\lambda_{1}=1$ \\
            \end{tabular}$$
          \end{table}
          \text{\,}\\

       $\bullet$ \textbf{Test $2.$} Now, let $\Omega_{f}$ be the unit square $(0,1)\times(0,1)$ and $T$ be the final time $T=2^{2}.$ We assume that $\nu=\lambda_{1}=1,$ and we choose the force $f$ in such a way that the exact solution $u=(u_{1},u_{2})$ is given by
       \begin{eqnarray*}
       % \nonumber to remove numbering (before each equation)
         u_{1}(t,x,y) &=& 10x^{2}(x-1)^{2}y(y-1)(2y-1)\cos(t),\\
         u_{2}(t,x,y) &=& -10x(x-1)(2x-1)y^{2}(y-1)^{2}\cos(t) .
       \end{eqnarray*}
       The initial and boundary conditions are set to
       \begin{equation*}
        u=0,\text{\,\,in\,\,}\{0\}\times\Omega_{f},\text{\,\,\,\,}u=0,\text{\,\,on\,\,}(0,1)\times\partial\Omega_{f}.
       \end{equation*}
       Similar to \textbf{Test $1,$} we take the mesh size and time step: $h,\text{\,}\Delta t\in\{\frac{1}{2^{2}},\frac{1}{2^{3}},\frac{1}{2^{4}},
       \frac{1}{2^{5}},\frac{1}{2^{6}},\frac{1}{2^{7}},\frac{1}{2^{8}},\frac{1}{2^{9}},\frac{1}{2^{10}}\},$ by a mid-point refinement. We set $\Delta t=h$ and we compute the error estimates: $E^{N}(u),$ $E^{N}\left(\nabla u\right)$ and $E^{N}(p)$ related to the rapid solver method to see that the algorithm is of second order accuracy in both time and space (see Table $\ref{tab:1}$ and Figure $\ref{figure}$). In addition, we plot the errors versus $n.$ From our analysis, it is obvious that the two-level hybrid scheme is both efficient and effective than the two-level finite element Galerkin approach which is first order accuracy.\\

         \section{General conclusion and future works}\label{V}
          We have studied in detail the error estimates and the rate of convergence of MCRS algorithm for $2$D incompressible Navier-Stokes model over long time intervals. The theoretical result has suggested that our method is convergent and second order accurate in both space and time for Stokes parameters ($u$ and $p$). This convergence rate is confirmed by some numerical experiments (see Table $\ref{tab:1}$), but this also was predicted in a previous paper \cite{en}. Numerical evidences also suggest that the new algorithm is fast and robust tools for the integration of general systems of mixed model. The analysis of stability and error estimates of the two-level hybrid scheme for mixed Stokes-Darcy problem will be the subject of our future investigations.\\

         \begin{figure}
         \begin{center}
          Convergence rate of a two-level hybrid method.
          \begin{tabular}{c c}
          \psfig{file=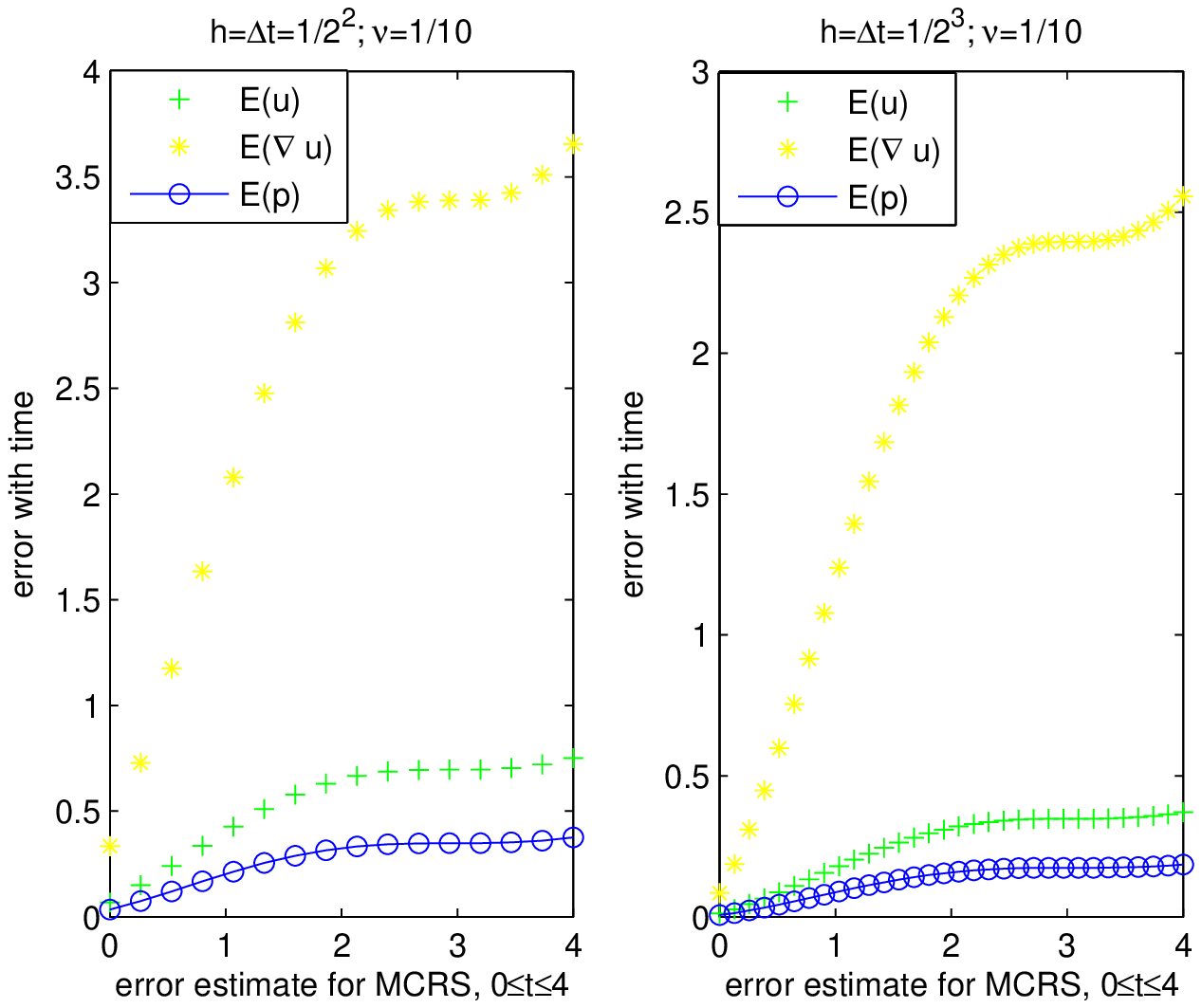,width=6.5cm} & \psfig{file=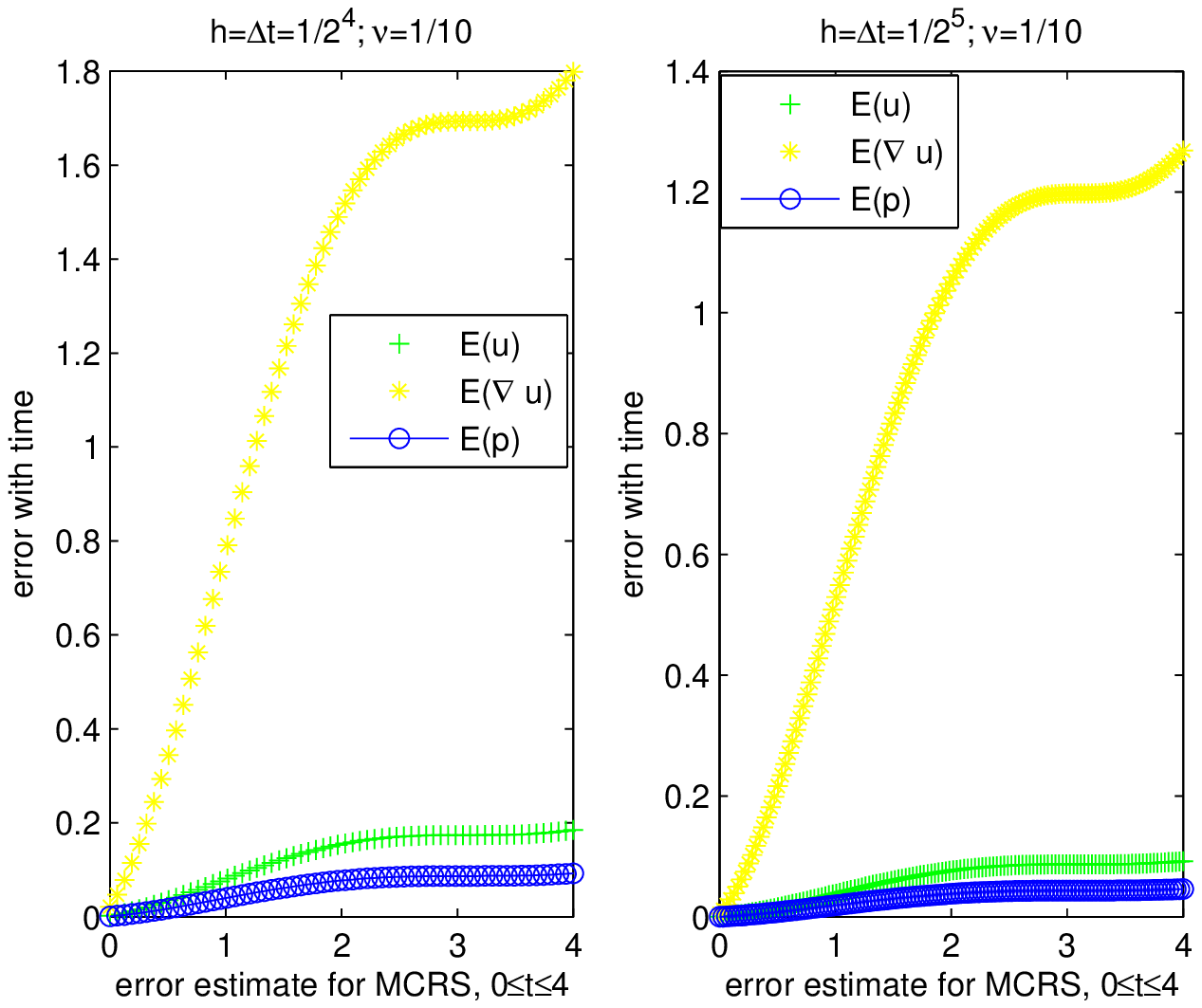,width=6.5cm}\\
          \psfig{file=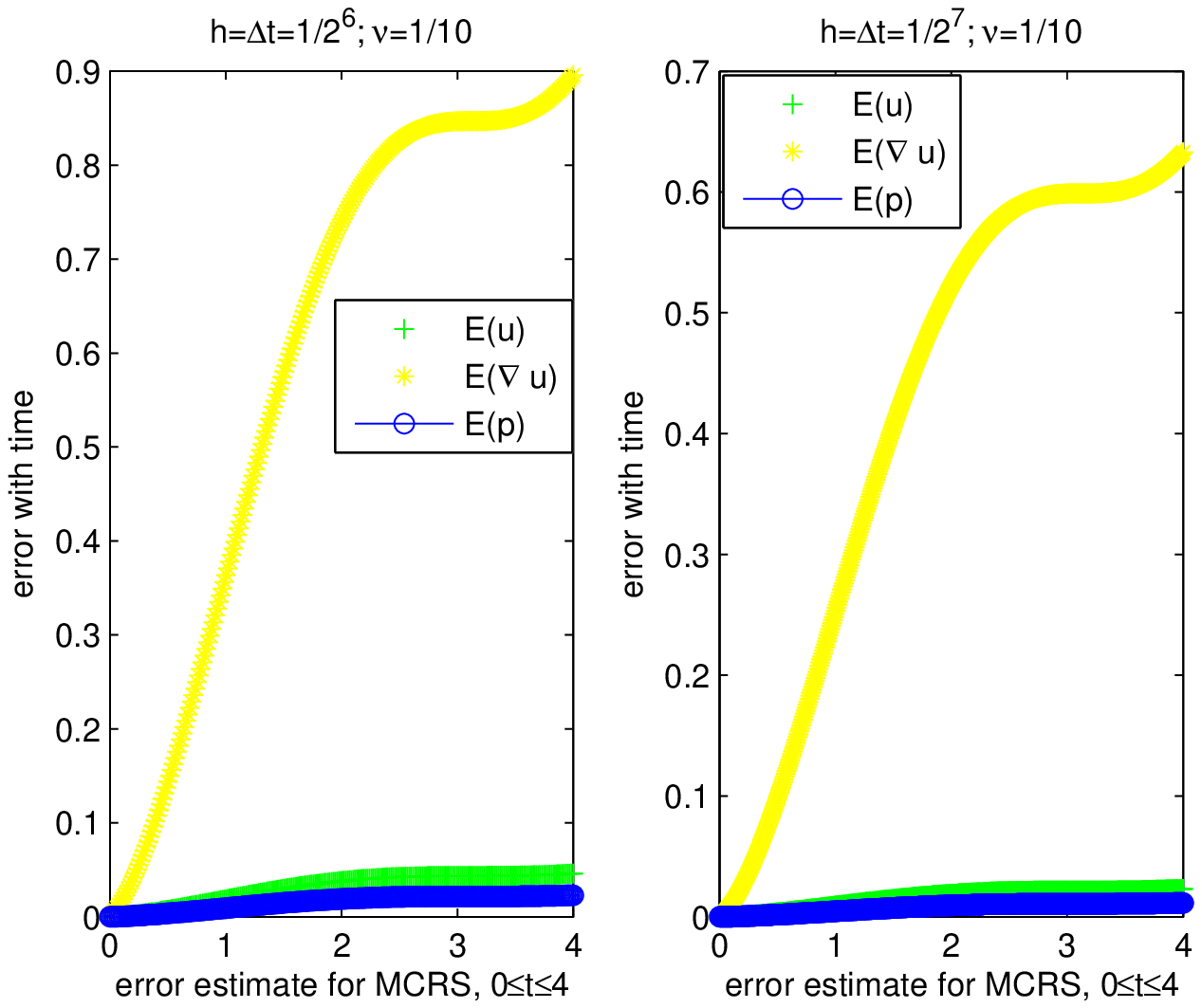,width=6.5cm} & \psfig{file=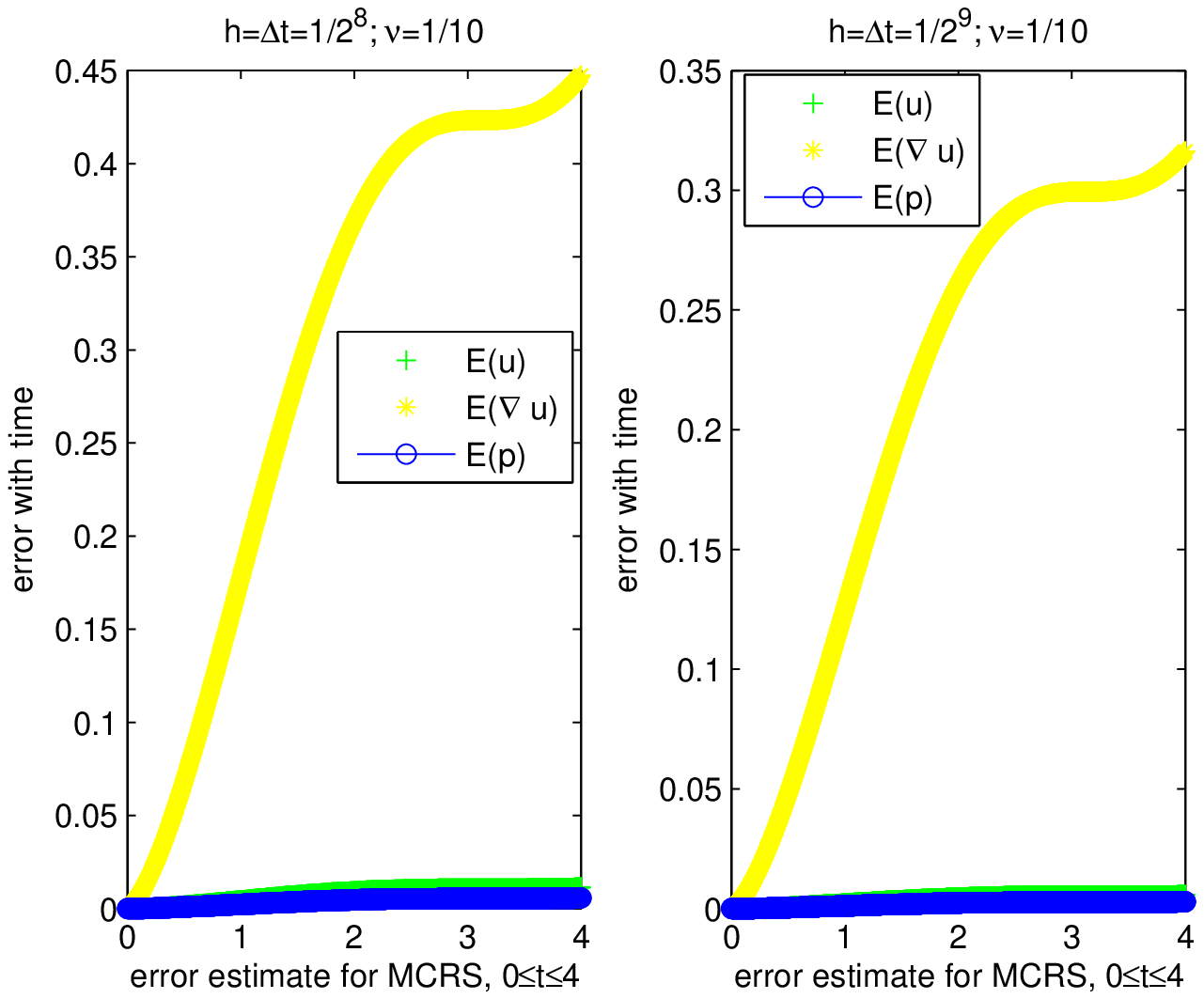,width=6.5cm}\\
          \textbf{Test 1:} $\nu=10^{-1}$ and $\lambda_{1}=1$                           \\
          \text{\,}   \\
          \psfig{file=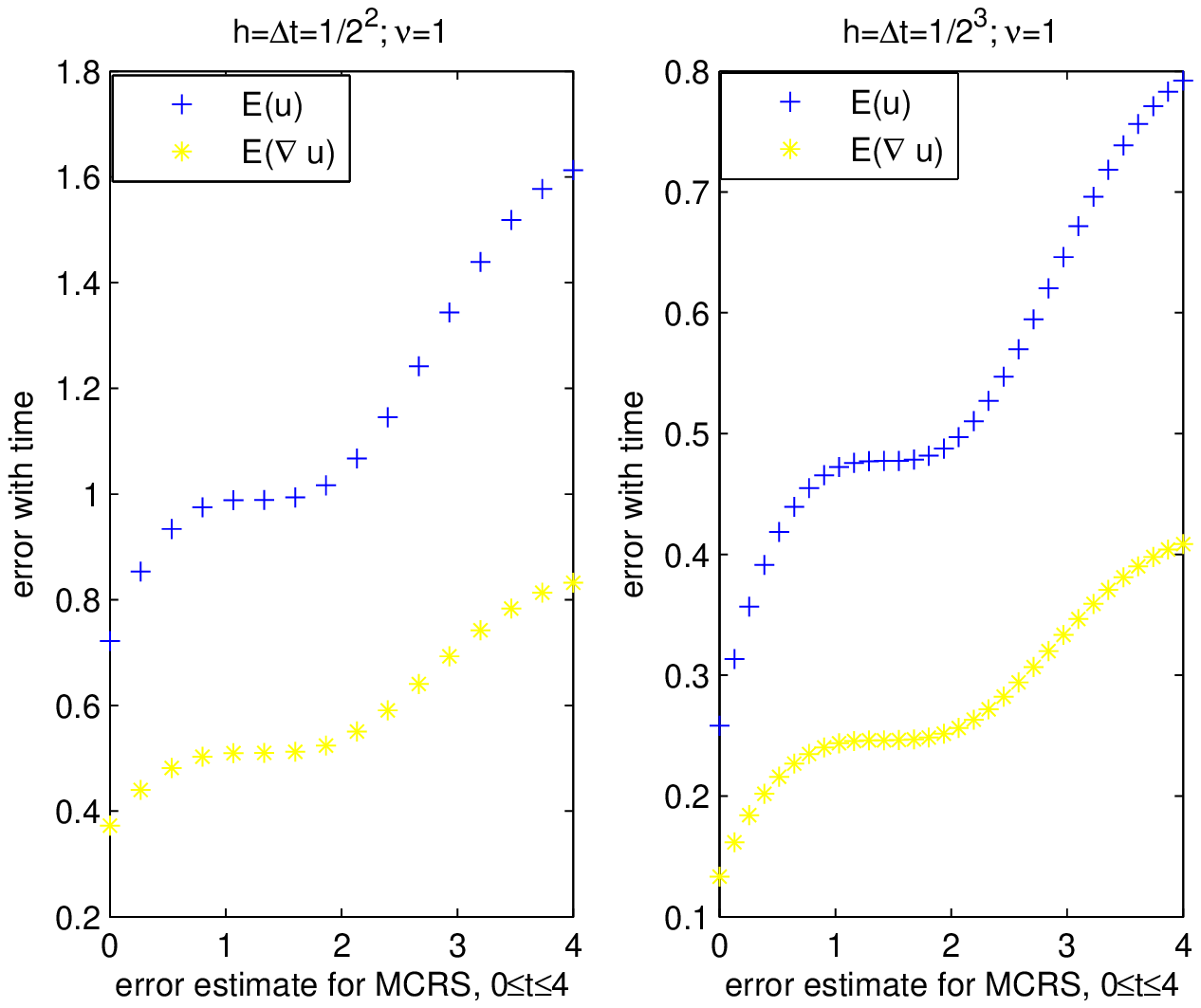,width=6.5cm} & \psfig{file=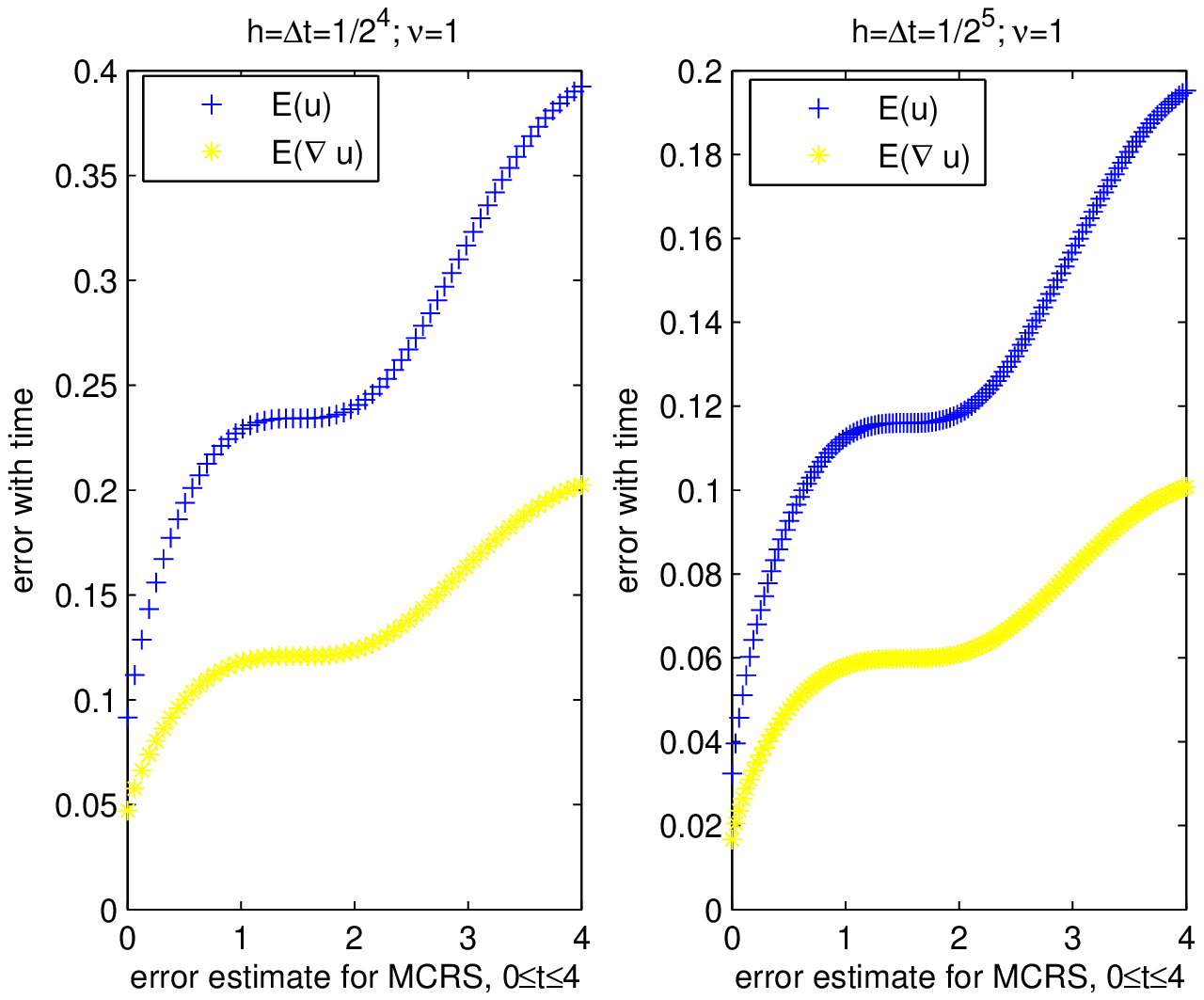,width=6.5cm}\\
          \psfig{file=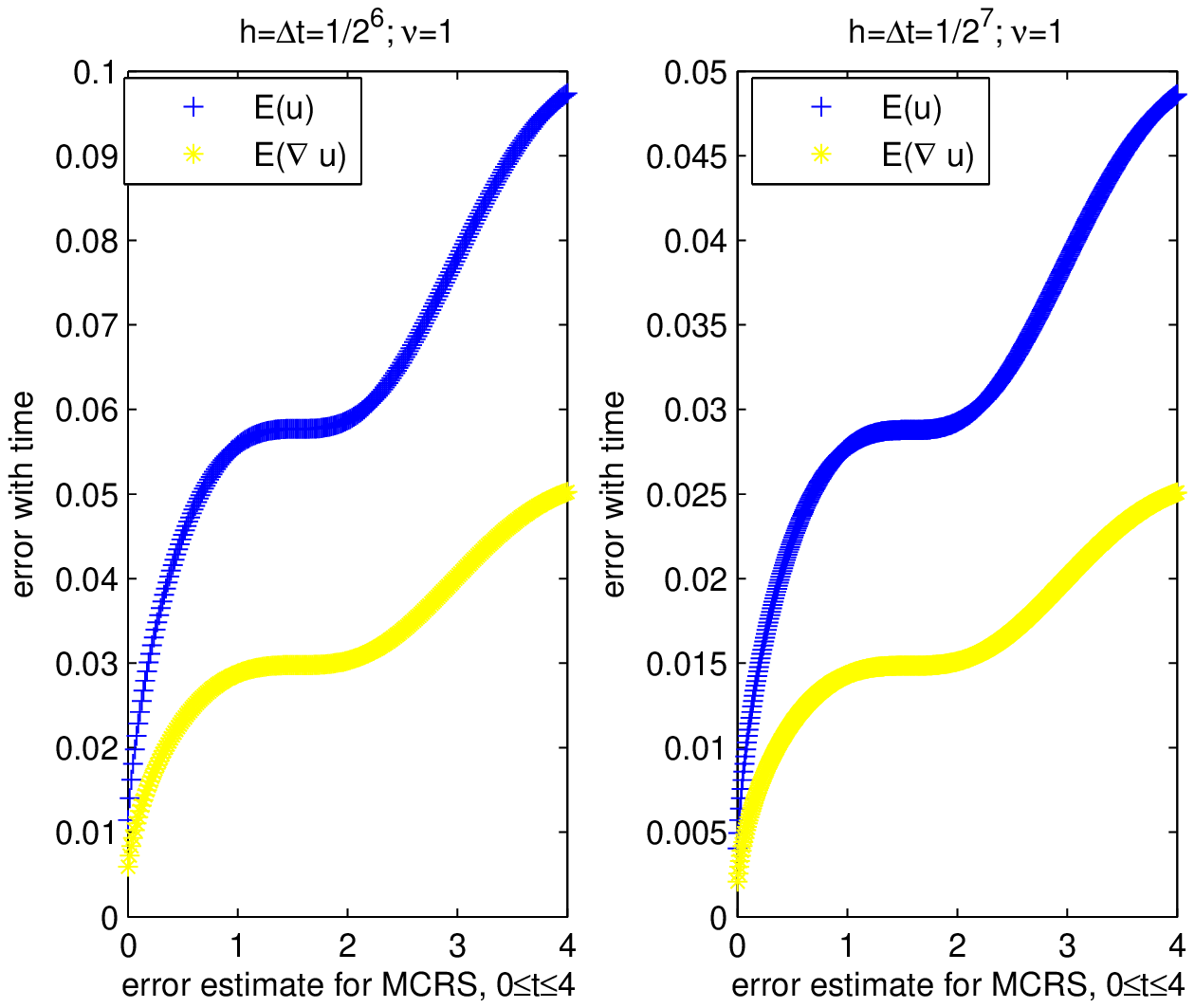,width=6.5cm} & \psfig{file=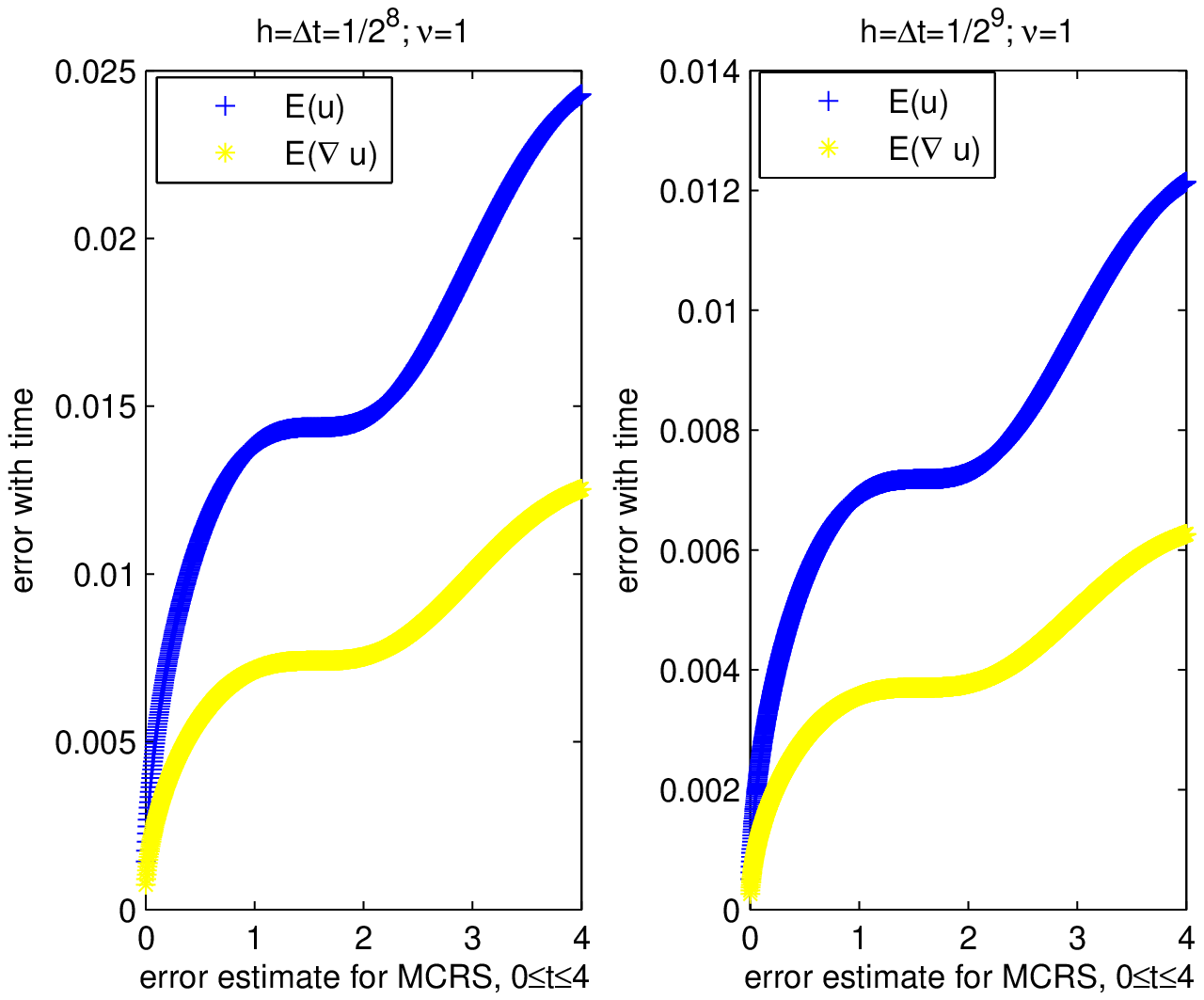,width=6.5cm}\\
          \textbf{Test 2}: $\nu=\lambda_{1}=1$                                   \\
         \end{tabular}
        \end{center}
         \caption{Analysis of error estimates for MCRS}
          \label{figure}
          \end{figure}

     \end{document}